\documentclass[11pt,english,onecolumn]{extarticle}
\usepackage[T3,T1]{fontenc}
\usepackage[latin9]{inputenc}
\usepackage{color}
\usepackage{babel}
\usepackage{mathtools}
\usepackage{amsmath}
\usepackage{amsthm}
\usepackage{amssymb}
\usepackage{undertilde}
\usepackage{graphicx}
\usepackage[unicode=true,pdfusetitle,
 bookmarks=true,bookmarksnumbered=true,bookmarksopen=true,bookmarksopenlevel=3,
 breaklinks=false,pdfborder={0 0 1},backref=false,colorlinks=true]
 {hyperref}
\hypersetup{
 pdfborderstyle=,pdfborderstyle={},pdfborderstyle={},pdfborderstyle={},pdfborderstyle={},pdfborderstyle={},pdfborderstyle={},pdfborderstyle={},pdfborderstyle={},pdfborderstyle={},pdfborderstyle={},pdfborderstyle={},pdfborderstyle={},pdfborderstyle={},pdfborderstyle={},pdfborderstyle={},pdfborderstyle={},pdfborderstyle={},pdfborderstyle={},pdfborderstyle={},pdfpagelayout=OneColumn,pdfnewwindow=true,pdfstartview=XYZ,plainpages=false,linkcolor=blue,urlcolor=blue,citecolor=red,anchorcolor=blue,linkcolor=blue,urlcolor=blue,citecolor=red,anchorcolor=blue}

\makeatletter

\providecommand{\tabularnewline}{\\}

\numberwithin{equation}{section}
\numberwithin{figure}{section}
\theoremstyle{plain}
\newtheorem{thm}{\protect\theoremname}
\theoremstyle{plain}
\newtheorem{lem}{\protect\lemmaname}
\theoremstyle{plain}
\newtheorem{prop}{\protect\propositionname}
\theoremstyle{plain}
\newtheorem{cor}{\protect\corollaryname}
\theoremstyle{remark}
\newtheorem{rem}{\protect\remarkname}

\usepackage{babel}
\usepackage{accents}
\usepackage{fullpage}
\usepackage{enumerate}
\usepackage{verbatim}
\usepackage{mathrsfs}
\usepackage{cite}
\usepackage{centernot}
\newcommand{\nll}{\centernot{\ll}}
\usepackage{mathtools}  
\mathtoolsset{showonlyrefs} 

\providecommand{\keywords}[1]{\textbf{Index terms---} #1}

\DeclareSymbolFont{tipa}{T3}{cmr}{m}{n}
\DeclareMathAccent{\invbreve}{\mathalpha}{tipa}{16}

\def\undertilde#1{\mathord{\vtop{\ialign{##\crcr
$\hfil\displaystyle{#1}\hfil$\crcr\noalign{\kern1.5pt\nointerlineskip}
$\hfil\tilde{}\hfil$\crcr\noalign{\kern1.5pt}}}}}

\newcommand{\subsetsim}{\mathrel{%
  \ooalign{\raise0.2ex\hbox{$\subset$}\cr\hidewidth\raise-0.8ex\hbox{\scalebox{0.9}{$\sim$}}\hidewidth\cr}}}
\newcommand{\supsetsim}{\mathrel{%
  \ooalign{\raise0.2ex\hbox{$\supset$}\cr\hidewidth\raise-0.8ex\hbox{\scalebox{0.9}{$\sim$}}\hidewidth\cr}}}

\newcommand{\subsetapprox}{\mathrel{%
  \ooalign{\raise0.4ex\hbox{$\subset$}\cr\hidewidth\raise-0.8ex\hbox{\scalebox{0.9}{$\approx$}}\hidewidth\cr}}}
\allowdisplaybreaks[1]
\def\Ent{\operatorname{Ent}}

\def\1{\mathbf{1}}

\newtheorem{remark}{Remark}[section]

\newtheorem{convention}{Convention}

\providecommand{\corollaryname}{Corollary}
\providecommand{\lemmaname}{Lemma}
\providecommand{\propositionname}{Proposition}
\providecommand{\remarkname}{Remark}
\providecommand{\theoremname}{Theorem}

\makeatother

\providecommand{\corollaryname}{Corollary}
\providecommand{\lemmaname}{Lemma}
\providecommand{\propositionname}{Proposition}
\providecommand{\remarkname}{Remark}
\providecommand{\theoremname}{Theorem}

\begin{document}
\title{Strong Brascamp--Lieb Inequalities}
\author{Lei Yu\thanks{L. Yu is with the School of Statistics and Data Science, LPMC \& KLMDASR,
Nankai University, Tianjin 300071, China (e-mail: leiyu@nankai.edu.cn). } }
\maketitle
\begin{abstract}
In this paper, we derive sharp nonlinear dimension-free Brascamp--Lieb
inequalities (including hypercontractivity inequalities) for distributions
on Polish spaces, which strengthen the classic Brascamp--Lieb inequalities.
Applications include the extension of Mrs. Gerber's lemma to the cases
of Rényi divergences and distributions on Polish spaces, the strengthening
of small-set expansion theorems, and the characterization of the exponent
of the $q$-stability. Our proofs in this paper are based on information-theoretic
and coupling techniques. 
\end{abstract}

\section{Introduction}

Let $\mathcal{X}$ and $\mathcal{Y}$ be two Polish spaces, and $\mathbb{B}_{\mathcal{X}}$
and $\mathbb{B}_{\mathcal{Y}}$ the Borel $\sigma$-algebras on $\mathcal{X}$
and $\mathcal{Y}$ respectively. Let $(\mathcal{X}\times\mathcal{Y},\mathbb{B}_{\mathcal{X}}\otimes\mathbb{B}_{\mathcal{Y}})$
be the product of the measurable spaces $(\mathcal{X},\mathbb{B}_{\mathcal{X}})$
and $(\mathcal{Y},\mathbb{B}_{\mathcal{Y}})$, and $P_{XY}$ a probability
measure (also termed \emph{distribution}) on $(\mathcal{X}\times\mathcal{Y},\mathbb{B}_{\mathcal{X}}\otimes\mathbb{B}_{\mathcal{Y}})$.
The (forward and reverse) Brascamp--Lieb (BL) inequalities\footnote{For simplicity, we consider BL inequalities that only involve a joint
distribution $P_{XY}$. These BL inequalities are special cases of
the original BL inequalities in \cite{brascamp1976best} and \cite{barthe1998reverse}.
In the forward BL inequality in \cite{brascamp1976best} is the same
to \eqref{eq:FBL} but with $\langle f,g\rangle:=\int fg\mathrm{d}\pi_{XY}$
induced by a joint distribution $\pi_{XY}$, and $\Vert f\Vert_{p},\Vert g\Vert_{q}$
respectively induced by distributions $\mu_{X},\mu_{Y}$. Here $\mu_{X},\mu_{Y}$
are not necessarily to be the marginals of $\pi_{XY}$. The reverse
BL inequality in \eqref{eq:RBL} can be also recovered from the one
introduced in \cite{barthe1998reverse}. We will discuss the generalization
of our results in Section \ref{sec:Generalization-to-Original}.} are as follows: Given the distribution $P_{XY}$ and $p,q\in\overline{\mathbb{R}}:=[-\infty,\infty]$,
for any nonnegative measurable functions $f:\mathcal{X}\to[0,\infty),g:\mathcal{Y}\to[0,\infty),$
\begin{align}
\langle f,g\rangle & \le e^{-\underline{C}}\Vert f\Vert_{p}\Vert g\Vert_{q}\label{eq:FBL}\\
\langle f,g\rangle & \geq e^{-\overline{C}}\Vert f\Vert_{p}\Vert g\Vert_{q},\label{eq:RBL}
\end{align}
where\footnote{Throughout this paper, when the integral is taken with respect to
the whole space, say $\mathcal{X}\times\mathcal{Y}$, we just write
$\int_{\mathcal{X}\times\mathcal{Y}}$ as $\int$. In addition, we
say an integral $\int\varphi\mathrm{d}\mu$ (or $\mathbb{E}[\varphi]$
for probability measure $\mu$) exists if either $\int_{\{\varphi>0\}}\varphi\mathrm{d}\mu$
or $\int_{\{\varphi<0\}}\varphi\mathrm{d}\mu$ is finite. } $\langle f,g\rangle:=\int fg\mathrm{d}P_{XY},$ 
\[
\Vert f\Vert_{p}:=\begin{cases}
(\int f^{p}\mathrm{d}P_{X})^{1/p} & p\notin\{0,\pm\infty\}\\
e^{\mathbb{E}[\log f]} & p=0\textrm{ and }\mathbb{E}[\log f]\textrm{ exists}\\
\infty & p=0\textrm{ and }\mathbb{E}[\log f]\textrm{ does not exist}\\
\lim_{p\to\infty}\Vert f\Vert_{p} & p=\infty\\
\lim_{p\to-\infty}\Vert f\Vert_{p} & p=-\infty
\end{cases}
\]
and $\Vert g\Vert_{q}$ is defined similarly but with $(f,P_{X},p)$
replaced by $(g,P_{Y},q)$, and $\underline{C}=\underline{C}_{p,q}$
and $\overline{C}=\overline{C}_{p,q}$ only depend on $p,q$ given
the distribution $P_{XY}$. We denote the optimal exponents $\underline{C}_{p,q}$
and $\overline{C}_{p,q}$ for the distribution $P_{XY}$ as $\underline{C}_{p,q}^{*}$
and $\overline{C}_{p,q}^{*}$. Throughout this paper, we exclude the
trivial case that $P_{X}$ or $P_{Y}$ is a Dirac measure.

Note that the BL inequalities are homogeneous, i.e., \eqref{eq:FBL}
and \eqref{eq:RBL} still hold if we scale the functions $f,g$ respectively
by two positive real values $a,b$. Hence, without changing the optimal
exponents, we can normalize both sides of the BL inequalities in \eqref{eq:FBL}
and \eqref{eq:RBL} by $\Vert f\Vert_{\hat{p}}\Vert g\Vert_{\hat{q}}$
for reals $\hat{p}\neq p,\hat{q}\neq q$. The resulting normalized
version of BL inequalities can be understood as a \emph{linear} tradeoff
between the correlation $\frac{\langle f,g\rangle}{\Vert f\Vert_{\hat{p}}\Vert g\Vert_{\hat{q}}}$
and the product $\frac{\Vert f\Vert_{p}}{\Vert f\Vert_{\hat{p}}}\frac{\Vert g\Vert_{q}}{\Vert g\Vert_{\hat{q}}}$.
In this paper, we aim at studying the optimal \emph{nonlinear} tradeoff
among $\frac{\langle f,g\rangle}{\Vert f\Vert_{\hat{p}}\Vert g\Vert_{\hat{q}}},\frac{\Vert f\Vert_{p}}{\Vert f\Vert_{\hat{p}}}$,
and $\frac{\Vert g\Vert_{q}}{\Vert g\Vert_{\hat{q}}}$. To this end,
we first need introduce the $L_{p}$-norm and a new concept termed
\emph{$(p,\hat{p})$-entropy}.

\subsection{Preliminaries on (Pseudo) $L_{p}$-Norms}

Throughout this paper, when we write a function $f:\mathcal{X}\to[0,\infty)$,
by default, we mean that it is measurable. A similar convention also
applies to a function $g:\mathcal{Y}\to[0,\infty)$. Let $\Phi(\mathcal{Z}):=\{h:\mathcal{Z}\to[0,\infty)\}\backslash\{h=0\textrm{ a.s.}\}$
for $\mathcal{Z}=\mathcal{X}$ or $\mathcal{Y}$, which is the set
of nonenegative measurable functions but excluding the zero function.
Throughout this paper, we suppose that $f\in\Phi(\mathcal{X})$ and
$g\in\Phi(\mathcal{Y})$. The $p$-norm (or $p$-pseudo norm) $\Vert\cdot\Vert_{p}$
satisfies the following properties. 
\begin{lem}
\label{lem:norm} The following hold. 
\begin{enumerate}
\item Given $f\in\Phi(\mathcal{X})$, $\Vert f\Vert_{p}$ is nondecreasing
in the order $p\in\overline{\mathbb{R}}$.
\item If $\mathbb{E}[\log f]$ exists, then, $\lim_{r\downarrow0}\Vert f\Vert_{r}=\Vert f\Vert_{0}$
if $\Vert f\Vert_{s}<\infty$ for some $s>0$, and $\lim_{r\uparrow0}\Vert f\Vert_{r}=\Vert f\Vert_{0}$
if $\Vert f\Vert_{s}>0$ for some $s<0$. Furthermore, if $\mathbb{E}[\log f]$
does not exist, then $\Vert f\Vert_{s}=\infty$ for all $s>0$, and
$\Vert f\Vert_{s}=0$ for all $s<0$. 
\item $\Vert f\Vert_{\infty}$ is equal to the essential supremum of $f$. 
\item $\Vert f\Vert_{-\infty}$ is equal to the essential infimum of $f$. 
\item If $P_{X}(f>0)<1$, i.e., $f$ is not positive a.e., then $\Vert f\Vert_{p}=0$
for $p\in[-\infty,0]$. 
\item Given $f\in\Phi(\mathcal{X})$, $\Vert f\Vert_{p}$ is continuous
in the order $p\in\{p\in\overline{\mathbb{R}}:\Vert f\Vert_{p}<\infty\}$. 
\item For $p\in(0,\infty]$, $\Vert f\Vert_{p}=0$ if and only if $f=0$
a.e. Hence, for $f\in\Phi(\mathcal{X})$, we always have $\Vert f\Vert_{p}>0$
for $p\in(0,\infty]$. 
\item For $p\in[-\infty,0)$, $\Vert f\Vert_{p}=\infty$ if and only if
$f=\infty$ a.e. Hence, for any $f:\mathcal{X}\to[0,\infty)$, we
always have $\Vert f\Vert_{p}<\infty$ for $p\in[-\infty,0)$. 
\end{enumerate}
\end{lem}

\begin{proof}
Statement 1 follows by Jensen's inequality. Statement 2 follows by
the following argument. For $t>0$, define $h(s)=\frac{t^{s}-1}{s}$
if $s\neq0$, and $\log t$ otherwise. Then, $h$ is nondecreasing
on the whole real line. This immediately implies that if $\Vert f\Vert_{s}<\infty$
for some $s>0$, then $\int_{\{f>1\}}\log f\mathrm{d}P_{X}$ is finite.
Hence, the second part of Statement 2 holds. Furthermore, combined
with the existence of $\mathbb{E}[\log f]$, it implies that $\mathbb{E}[|\log f|]$
is finite. Substituting $f\leftarrow f^{s}$ into the formula in
\cite[Exercise 2.12.91]{bogachev2007measure1} yields the first part
of Statement 2 for $s>0$. The part for $s<0$ follows similarly,
in which case, substituting $f\leftarrow f^{s}$ with $s<0$ into
the formula in \cite[Exercise 2.12.91]{bogachev2007measure1} yields
the first part of Statement 2 for $s<0$. Statement 6 follows by
the following argument. For $p=0$, it follows from Statement 2. We
next consider the case $p\neq0$. We partition $\mathcal{X}$ into
two disjoint sets $\{f\ge1\}$ and $\{f<1\}$, and decompose the integral
in the definition of $\Vert f\Vert_{p}$ into two parts according
to this partition. For $p>0$, we apply Lebesgue's dominated convergence
theorem to the integral on set $\{f<1\}$, and apply Beppo Levi's
lemma on the monotone convergence to the integral on set $\{f\ge1\}$.
Then, we have $\lim_{r\uparrow p}\Vert f\Vert_{r}=\Vert f\Vert_{p}$.
Similarly, Statement 6 for other cases can be proven.
\end{proof}
Based on Statements 1, 7, and 8 above, to ensure that $0<\Vert f\Vert_{p}<\infty$
for $f\in\Phi(\mathcal{X})$ and $p\in\overline{\mathbb{R}}\backslash\{0\}$,
it suffices to require $\Vert f\Vert_{p}>0$ if $p<0$, and $\Vert f\Vert_{p}<\infty$
if $p>0$. Hence, if we define 
\[
\mathcal{D}_{f}:=\{p\in\overline{\mathbb{R}}:0<\Vert f\Vert_{p}<\infty\},
\]
then for $f\in\Phi(\mathcal{X})$, $\mathcal{D}_{f}$ corresponds
to the union of $\{p\in[-\infty,0):\Vert f\Vert_{p}>0\}$, $\{p\in(0,\infty]:\Vert f\Vert_{p}<\infty\}$,
and also $\{0\}$ if $0<\Vert f\Vert_{0}<\infty$.  Note that $\mathcal{D}_{f}$
is an interval.

\subsection{Two-Parameter Version of Entropies}

Throughout this subsection, we assume $f\in\Phi(\mathcal{X})$. We
call $p$ as a \emph{simple order} if $p\in\mathbb{R}\backslash\{0\}$,
and call $p$ as an \emph{extended order} if $p\in\overline{\mathbb{R}}$.
We first define the $(p,\hat{p})$-entropy for a pair of simple orders.
For convenience, we denote 
\[
\mathcal{R}_{f}:=\mathcal{D}_{f}\cup\{0\}.
\]
Here, the $0$-pseudo norm is excluded since it will not be used in
the definition of $(p,\hat{p})$-entropy, and in fact, the $(p,0)$-
or $(0,p)$-entropy will be defined by the continuous extension. Given
$p,\hat{p}\in\mathcal{R}_{f}\backslash\{0,\pm\infty\}$ such that
$p\neq\hat{p}$, we define the \emph{$(p,\hat{p})$-entropy} of $f$
as\footnote{Throughout this paper, we use $\log$ to denote the logarithm of base
$e$. In fact, our results still hold if we use other bases, as long
as the base of the logarithm is consistent with the one of the exponentiation. } 
\begin{equation}
\Ent_{p,\hat{p}}(f):=\frac{p\hat{p}}{p-\hat{p}}\log\frac{\Vert f\Vert_{p}}{\Vert f\Vert_{\hat{p}}}.\label{eq:ent}
\end{equation}
For $p,\hat{p}$ with $p\neq\hat{p}$, one of them being in $\mathbb{R}\backslash\{0\}$,
and the other being in $\mathbb{R}\backslash\mathcal{R}_{f}$ (i.e.,
$\Vert f\Vert_{p}$ or $\Vert f\Vert_{\hat{p}}$ being zero or $\infty$),
we define $\Ent_{p,\hat{p}}(f):=\infty$ if $p\hat{p}>0$, and $\Ent_{p,\hat{p}}(f):=-\infty$
if $p\hat{p}<0$.

The $(p,\hat{p})$-entropies for other cases are defined by continuous
extensions. Specifically, for $p\in\mathcal{R}_{f}\backslash\{0,\pm\infty\}$,
the \emph{$(p,p)$-entropy} of $f$ is defined as 
\begin{equation}
\Ent_{p,p}(f):=\lim_{\hat{p}/p\uparrow1}\Ent_{p,\hat{p}}(f)=\Ent_{p}(f):=\frac{\mathbb{E}[f^{p}\log f^{p}]-\mathbb{E}[f^{p}]\log\mathbb{E}[f^{p}]}{\mathbb{E}[f^{p}]},\label{eq:Shannonentropy-1-2}
\end{equation}
and for $p\in\mathbb{R}\backslash\mathcal{R}_{f}$, 
\begin{equation}
\Ent_{p,p}(f):=\Ent_{p}(f):=\infty.\label{eq:Shannonentropy-1-2-2-1}
\end{equation}
Note that the usual entropy of $f^{p}$ corresponds to the numerator
in the right-hand side (RHS) of \eqref{eq:Shannonentropy-1-2}. Hence,
$\Ent_{p}(f)$ coincides the normalized version of the entropy of
$f^{p}$.

We next define the $(p,\hat{p})$-entropy for a pair of orders such
that at least one of them is an extended order. For $p\in\mathcal{R}_{f}$,
the \emph{$(p,0)$-entropy} and the \emph{$(0,p)$-entropy} of $f$
are defined by 
\[
\Ent_{p,0}(f),\;\Ent_{0,p}(f):=\Ent_{0}(f):=-\log P_{X}(f>0).
\]
For $p\notin\mathcal{R}_{f}$, we 
\[
\Ent_{p,0}(f),\;\Ent_{0,p}(f):=\begin{cases}
\infty & p\in(0,\infty]\\
-\infty & p\in[-\infty,0)
\end{cases}.
\]

For $p\in\overline{\mathbb{R}}\backslash\{0\}$, the \emph{$(p,\infty)$-entropy}
and the \emph{$(\infty,p)$-entropy} of $f$ satisfying $\Vert f\Vert_{\infty}<\infty$
are defined by 
\[
\Ent_{p,\infty}(f),\;\Ent_{\infty,p}(f):=\begin{cases}
\lim_{\hat{p}\uparrow\infty}\Ent_{p,\hat{p}}(f)=-p\log\frac{\Vert f\Vert_{p}}{\Vert f\Vert_{\infty}} & p\in\mathbb{R}\backslash\{0\}\\
\lim_{p\uparrow\infty}\Ent_{p,\infty}(f)=\Ent_{\infty}(f):=-\log P_{X}(f=\Vert f\Vert_{\infty}) & p=\infty\\
\lim_{p\downarrow-\infty}\Ent_{p,\infty}(f)=0 & f\textrm{ is constant},\;p=-\infty\\
\lim_{p\downarrow-\infty}\Ent_{p,\infty}(f)=-\infty & f\textrm{ is not constant},\;p=-\infty
\end{cases}.
\]
For $f$ satisfying $\Vert f\Vert_{\infty}=\infty$, 
\[
\Ent_{p,\infty}(f),\;\Ent_{\infty,p}(f):=\begin{cases}
\infty & p\in(0,\infty]\\
-\infty & p\in[-\infty,0)
\end{cases}.
\]

For $p\in[-\infty,\infty)\backslash\{0\}$, the \emph{$(p,-\infty)$-entropy}
and the \emph{$(-\infty,p)$-entropy} of $f$ satisfying $\Vert f\Vert_{-\infty}>0$
(which implies $f>0$ a.e.) are defined by 
\[
\Ent_{p,-\infty}(f),\;\Ent_{-\infty,p}(f):=\begin{cases}
\lim_{\hat{p}\downarrow-\infty}\Ent_{p,\hat{p}}(f)=-p\log\frac{\Vert f\Vert_{p}}{\Vert f\Vert_{-\infty}} & p\in\mathbb{R}\backslash\{0\}\\
\lim_{p\downarrow-\infty}\Ent_{p,-\infty}(f)=\Ent_{-\infty}(f):=-\log P_{X}(f=\Vert f\Vert_{-\infty}) & p=-\infty
\end{cases}.
\]
For $f$ satisfying $\Vert f\Vert_{-\infty}=0$, 
\[
\Ent_{p,-\infty}(f),\;\Ent_{-\infty,p}(f):=\begin{cases}
-\infty & p\in(0,\infty)\\
\infty & p\in[-\infty,0)
\end{cases}.
\]

\begin{lem}
It holds that $\Ent_{p,\hat{p}}(f)=\Ent_{\hat{p},p}(f)=\Ent_{p/s,\hat{p}/s}(f^{s})=\Ent_{\hat{p}/s,p/s}(f^{s})$
for $s\neq0$. 
\end{lem}

The $(p,\hat{p})$-entropy defined above is closely related to the
well-known \emph{Rényi divergences} in information theory. For $s\in\mathbb{R}\backslash\{0,1\}$,
the Rényi divergence of order $s$ is defined as $D_{s}(Q\|P):=\frac{1}{s-1}\log\int(\frac{\mathrm{d}Q}{\mathrm{d}\mu})^{s}(\frac{\mathrm{d}P}{\mathrm{d}\mu})^{1-s}\mathrm{d}\mu$
for a nonnegative finite measure $\mu$ (in fact, the value of $D_{s}(Q\|P)$
is independent of the choice of $\mu$), where $\frac{\mathrm{d}P}{\mathrm{d}\mu}$
denotes the Radon--Nikodym derivative of $P$ w.r.t. $\mu$. For
$s\in\{0,1,\pm\infty\}$, the Rényi divergence of order $s$ is defined
by the continuous extension. For $s=1$, the Rényi divergence of order
$1$, also termed the Kullback--Leibler divergence, or simply, the
KL divergence, is $D(Q\|P):=\int\log(\frac{\mathrm{d}Q}{\mathrm{d}P})\mathrm{d}Q$
if $Q\ll P$, and otherwise, $D(Q\|P):=\infty$. Here and throughout
this paper, we adopt the conventions that $\log0=-\infty,\log\infty=\infty$
and $0\log0=0$. The $(p,\hat{p})$-entropy and the Rényi divergence
of order $\hat{p}/p$ admit the following intimate relationship. 
\begin{lem}
For $f\in\Phi(\mathcal{X})$ and $p\in\mathcal{R}_{f}\backslash\{0,\pm\infty\}$,
we can write $\frac{f^{p}}{\Vert f\Vert_{p}^{p}}$ as a Radon--Nikodym
derivative $\frac{\mathrm{d}Q_{X}}{\mathrm{d}P_{X}}$. Then, for such
$p$ and together with $\hat{p}\in\overline{\mathbb{R}}$, 
\begin{equation}
\Ent_{p,\hat{p}}(f)=D_{\hat{p}/p}(Q_{X}\|P_{X}).\label{eq:}
\end{equation}
\end{lem}

Throughout this paper, we use the convention $\infty\cdot0=-\infty\cdot0=0$. 
\begin{lem}
For $f\in\Phi(\mathcal{X})$, the following hold. 
\begin{enumerate}
\item $\Ent_{p,\hat{p}}(f)\ge0$ for $p,\hat{p}\ge0$ or $p,\hat{p}<0$. 
\item $\Ent_{p,\hat{p}}(f)\le0$ for $p<0\le\hat{p}$ or $\hat{p}<0\le p$. 
\item For $p\hat{p}\neq0$, $\Ent_{p,\hat{p}}(f)=0$ if and only if $f$
is constant a.e. 
\item For $p\hat{p}=0$, $\Ent_{p,\hat{p}}(f)=0$ if and only if $f$ is
positive a.e. and $p,\hat{p}\in\mathcal{R}_{f}$. 
\item $\Ent_{p,\hat{p}}(f)$ is nondecreasing in $\hat{p}\in\overline{\mathbb{R}}$
given $p\in[0,\infty]$, and nonincreasing in $\hat{p}\in\overline{\mathbb{R}}$
given $p\in[-\infty,0)$. 
\end{enumerate}
\end{lem}


\begin{proof}
Statements 1)-4) follow by definition or by the equivalence in \eqref{eq:}.

We now prove Statement 5). For $p\in(0,\infty)\cap\mathcal{R}_{f}$,
\eqref{eq:} holds. Hence, by the monotonicity of the Rényi divergence
in its order \cite{Erven}, $\Ent_{p,\hat{p}}(f)$ is nondecreasing
in $\hat{p}\in\mathcal{R}_{f}$ given $p\in(0,\infty)\cap\mathcal{R}_{f}$.
The monotonicity for other $p$ can be verified similarly. In fact,
the monotonicity of $\hat{p}\mapsto\Ent_{p,\hat{p}}(f)$ (or $s\mapsto D_{s}(Q_{X}\|P_{X})$)
can be also obtained by differentiating this function with respect
to $\hat{p}$ (or $s$) and writing the derivative as the product
of a nonnegative factor and a relative entropy. 
\end{proof}
Given $f$, define $p_{\min}:=\inf\mathcal{R}_{f},\:p_{\max}:=\sup\mathcal{R}_{f}$. 
\begin{lem}
For $f\in\Phi(\mathcal{X})$, the following hold. 
\begin{enumerate}
\item For $p,\hat{p}\in\overline{\mathbb{R}}$, $|\Ent_{p,\hat{p}}(f)|=\infty$
if $p$ or $\hat{p}\in\overline{\mathbb{R}}\backslash\mathcal{R}_{f}$. 
\item For $p,\hat{p}\in\overline{\mathbb{R}}$, $|\Ent_{p,\hat{p}}(f)|<\infty$
if $(p,\hat{p})\in\mathcal{R}_{f}^{2}\backslash\{(p_{\min},p_{\min}),(p_{\max},p_{\max}),(\infty,-\infty),(-\infty,\infty)\}$. 
\end{enumerate}
\end{lem}

We define 
\[
\mathcal{S}_{f}:=\{(p,\hat{p})\in\overline{\mathbb{R}}^{2}:|\Ent_{p,\hat{p}}(f)|<\infty\}.
\]
Then, by this lemma, the region $\mathcal{S}_{f}$ is almost equal
to $\mathcal{R}_{f}^{2}$ up to excluding at most four points $(p_{\min},p_{\min}),(p_{\max},p_{\max})$,
$(\infty,-\infty),(-\infty,\infty)$. 
\begin{proof}
We now prove Statement 1). One can check every case with $p$ or $\hat{p}\in\overline{\mathbb{R}}\backslash\mathcal{R}_{f}$,
and will find that Statement 1) holds.

We next prove Statement 2). By definition, $|\Ent_{p,\hat{p}}(f)|<\infty$
for $p,\hat{p}\in\mathcal{R}_{f}\backslash\{\pm\infty\}$ but $p\neq\hat{p}$.
Utilizing this and by the monotonicity of $(p,\hat{p})$-entropy in
its orders, we have $|\Ent_{p,\hat{p}}(f)|<\infty$ for $p=\hat{p}\in(p_{\min},p_{\max})$.
For $\hat{p}\in\mathcal{R}_{f}\cap\{\pm\infty\}$ and $p\in\mathcal{R}_{f}\backslash\{\pm\infty\}$,
by definition, it holds that $|\Ent_{p,\hat{p}}(f)|<\infty$. Therefore,
Statement 2) holds. 
\end{proof}
We next show the continuity of the $(p,\hat{p})$-entropy in its orders. 
\begin{lem}
Given $f\in\Phi(\mathcal{X})$, $\Ent_{p,\hat{p}}(f)$ is continuous
in $(p,\hat{p})$ on $\mathcal{S}_{f}$. 
\end{lem}

\begin{proof}
Case 1: We first consider $p_{0},\hat{p}_{0}\in\mathcal{R}_{f}\backslash\{0,\pm\infty\}$
and $p_{0}\neq\hat{p}_{0}$. For this case, by the continuity of the
(pseudo) norm, we have that $\Ent_{p,\hat{p}}(f)$ is continuous at
$(p_{0},\hat{p}_{0})$.

Case 2: We next consider $p_{0}=\hat{p}_{0}\in\mathcal{R}_{f}\backslash\{0\}$.
By the monotonicity, for $p,\hat{p}\ge0$, 
\begin{equation}
\Ent_{\min\{p,\hat{p}\}}(f)\le\Ent_{p,\hat{p}}(f)\leq\Ent_{\max\{p,\hat{p}\}}(f),\label{eq:H-3-1-3}
\end{equation}
and for $p,\hat{p}<0$, 
\begin{equation}
\Ent_{\max\{p,\hat{p}\}}(f)\le\Ent_{p,\hat{p}}(f)\leq\Ent_{\min\{p,\hat{p}\}}(f).\label{eq:H-3-1-1-1}
\end{equation}
Hence, we only need prove the continuity of $p\mapsto\mathbb{E}[f^{p}]$
at $p=p_{0}$.

We first prove the left-continuity for this case. By Lemma \ref{lem:norm},
$p\mapsto\mathbb{E}[f^{p}]$ is left-continuous at $p=p_{0}$, and
$0<\mathbb{E}[f^{p_{0}}]<\infty$. Hence, to show the left-continuity
of $p\mapsto\Ent_{p}(f)$, it suffices to prove the left-continuity
of $p\mapsto\mathbb{E}[f^{p}\log f^{p}]$. To this end, we partition
$\mathcal{X}$ into two disjoint sets $\{f\ge1\}$ and $\{f<1\}$,
and decompose the integral in the definition of $\mathbb{E}[f^{p}\log f^{p}]$
into two parts $\int_{\{f\ge1\}}f^{p}\log f^{p}\mathrm{d}P_{X}+\int_{\{f<1\}}f^{p}\log f^{p}\mathrm{d}P_{X}$.
For $p_{0}>0$, $p\mapsto f^{p}(x)\log f^{p}(x)$ is nondecreasing
for $x$ such that $f(x)\ge1$ and this mapping is also continuous.
Hence, by Beppo Levi's lemma on the monotone convergence (or by Lebesgue's
dominated convergence theorem), 
\begin{equation}
\int_{\{f\ge1\}}f^{p}\log f^{p}\mathrm{d}P_{X}\to\int_{\{f\ge1\}}f^{p_{0}}\log f^{p_{0}}\mathrm{d}P_{X}\label{eq:-2}
\end{equation}
as $p\uparrow p_{0}$. On the other hand, since $t\log t\ge-e^{-1}$
for all $t\ge0$, we have 
\[
-e^{-1}\le f^{p}(x)\log f^{p}(x)\le0
\]
for $x$ such that $f(x)<1$. Hence, by Lebesgue's dominated convergence
theorem, 
\begin{equation}
\int_{\{f<1\}}f^{p}\log f^{p}\mathrm{d}P_{X}\to\int_{\{f<1\}}f^{p_{0}}\log f^{p_{0}}\mathrm{d}P_{X}\label{eq:-36}
\end{equation}
as $p\to p_{0}$. Combining the two points above, $\Ent_{p}(f)\to\Ent_{p_{0}}(f)$
as $p\uparrow p_{0}$ for $p_{0}>0$. For $p_{0}<0$, by Lebesgue's
dominated convergence theorem, \eqref{eq:-2} and \eqref{eq:-36}
as $p\uparrow p_{0}$ still hold. Hence, $\Ent_{p}(f)\to\Ent_{p_{0}}(f)$
as $p\uparrow p_{0}$ for all $p_{0}\in\mathcal{R}_{f}\backslash\{0\}$.

Similarly, one can show that $p\mapsto\Ent_{p}(f)$ is right-continuous
at $p_{0}\in\mathcal{R}_{f}\backslash\{0\}$. Hence, it is continuous
at $p_{0}\in\mathcal{R}_{f}\backslash\{0\}$.

Case 3: We next prove that $(p,\hat{p})\mapsto\Ent_{p,\hat{p}}(f)$
is continuous at $(0,0)$ through paths in $\mathcal{S}_{f}$. We
first assume that $\mathcal{R}_{f}$ is a subset of $[0,\infty]$.
By the monotonicity, for $p,\hat{p}\ge0$, 
\begin{equation}
\Ent_{0}(f)\le\Ent_{p,\hat{p}}(f)\leq\Ent_{\max\{p,\hat{p}\}}(f).\label{eq:H-3-2}
\end{equation}
Since $\max\{p,\hat{p}\}\to0$ if $(p,\hat{p})\to(0,0)$, it suffices
to show that $\Ent_{s}(f)\to\Ent_{0}(f)$ as $s\downarrow0$. By definition,
\begin{align}
\lim_{s\downarrow0}\Ent_{s}(f) & =\lim_{s\downarrow0}\frac{\mathbb{E}[f^{s}\log f^{s}]-\mathbb{E}[f^{s}]\log\mathbb{E}[f^{s}]}{\mathbb{E}[f^{s}]}\label{eq:Shannonentropy-1-2-1-1-1}\\
 & =\frac{\mathbb{E}[\lim_{s\downarrow0}f^{s}\log f^{s}]-\mathbb{E}[\lim_{s\downarrow0}f^{s}]\log\mathbb{E}[\lim_{s\downarrow0}f^{s}]}{\mathbb{E}[\lim_{s\downarrow0}f^{s}]}\\
 & =\Ent_{0}(f)
\end{align}
where swapping the limit and integration follows by Lebesgue's dominated
convergence theorem, similarly to the derivation of \eqref{eq:-2}
and \eqref{eq:-36}.

Similarly, if $\mathcal{R}_{f}$ is a subset of $[-\infty,0]$, then
$(p,\hat{p})\mapsto\Ent_{p,\hat{p}}(f)$ is also continuous at $(0,0)$
through paths in $\mathcal{S}_{f}$.

We next assume that $\mathcal{R}_{f}$ includes a neighborhood of
$0$. For this case, $\Ent_{0,0}(f)=\Ent_{0}(f)=0$. By the monotonicity,
\begin{equation}
\Ent_{\max\{p,\hat{p}\},-\max\{p,\hat{p}\}}(f)\le\Ent_{p,\hat{p}}(f)\leq\Ent_{\max\{p,\hat{p}\}}(f).\label{eq:H-3}
\end{equation}
Hence, it suffices to show that $\Ent_{s}(f)\to\Ent_{0}(f)$ as $s\downarrow0$
and $\Ent_{s,-s}(f)\to\Ent_{0}(f)$ as $s\downarrow0$. Similarly
to the above, $\Ent_{s}(f)\to\Ent_{0}(f)$ as $s\downarrow0$. On
the other hand, 
\begin{equation}
\Ent_{s,-s}(f)=-\frac{s}{2}\log\frac{\Vert f\Vert_{s}}{\Vert f\Vert_{-s}}=-\frac{1}{2}(\log\mathbb{E}[f^{s}]+\log\mathbb{E}[f^{-s}])\to0\label{eq:ent-1}
\end{equation}
as $s\downarrow0$. Hence, for this case, $(p,\hat{p})\mapsto\Ent_{p,\hat{p}}(f)$
is also continuous at $(0,0)$ through paths in $\mathcal{S}_{f}$.

Case 4: For a point $(p_{0},\hat{p}_{0})\in\mathcal{S}_{f}$ such
that one of $p_{0},\hat{p}_{0}$ is in $\mathbb{R}\backslash\{0\}$
and the other is zero, one can follow steps similar to the derivations
for the case 3 to prove this case.

Case 5: For a point $(p_{0},\hat{p}_{0})\in\mathcal{S}_{f}$ such
that one of $p_{0},\hat{p}_{0}$ is $\pm\infty$, by definition, it
is easily verified the continuity for this case.
\end{proof}
Define 
\begin{align}
\alpha_{\max} & :=\sup_{A\in\mathbb{B}_{\mathcal{X}}:P_{X}(A)>0}-\log P_{X}(A),\label{eq:amax}
\end{align}
if the most RHS is finite; otherwise, define\footnote{For ease of presentation, here we introduce the notation $\infty^{-}$,
and we will later use ``$x\le\infty^{-}$'' to denote ``$x<\infty$'',
``$[0,\infty^{-}]$'' to denote ``$[0,\infty)$'', and ``$x\neq\infty^{-}$''
to denote ``$x\neq\infty$''.} $\alpha_{\max}:=\infty^{-}.$ It is well-known that for a probability
measure on a Polish space, each atom of the probability measure is
equivalent to a singleton. Hence, $\alpha_{\max}:=\sup_{x\in\mathcal{X}:P_{X}(x)>0}-\log P_{X}(x)$.
Define $\beta_{\max}$ for $P_{Y}$ similarly. If the supremum in
the definition of $\alpha_{\max}$ is finite, then 1) $P_{X}$ is
a (discrete) distribution with finite support, and 2) the supremum
is attained (and hence, it is actually a maximum). This is because,
if $P_{X}$ is not purely atomic, then there is a measurable set with
positive probability in which these is no atom, and hence, the conditional
distribution of $P_{X}$ on this set is atomless. Hence, for this
case, the supremum is $\infty$. Moreover, if $P_{X}$ is purely atomic
(or equivalently, discrete) with infinite support, then the infimum
of the probabilities of these atoms must be zero, contradicting with
the finiteness of the supremum above. Hence, the point 1) holds. The
point 2) is implied by the point 1).

For $p,\hat{p}\in\overline{\mathbb{R}}$, define the range of $(p,\hat{p})$-entropies
on $\mathcal{X}$ as
\begin{equation}
\mathcal{E}_{p,\hat{p}}(\mathcal{X}):=\{\Ent_{p,\hat{p}}(f):f\in\Phi(\mathcal{X})\}\backslash\{\pm\infty\}.\label{eq:E}
\end{equation}
Similarly, define the range of $(q,\hat{q})$-entropies on $\mathcal{Y}$
as $\mathcal{E}_{q,\hat{q}}(\mathcal{Y})$. 
\begin{prop}
\label{prop:E} If $P_{X}$ is a Dirac measure, then $\alpha_{\max}=0$
and $\mathcal{E}_{p,\hat{p}}(\mathcal{X})=\{0\}$ for $p,\hat{p}\in\overline{\mathbb{R}}$.
Otherwise, we have that $\alpha_{\max}>0$ and the following hold: 
\begin{enumerate}
\item For $p,\hat{p}\in(0,\infty]$, we have $\mathcal{E}_{p,\hat{p}}(\mathcal{X})=[0,\alpha_{\max}]\backslash\{\infty\}$.
\item For $p,\hat{p}\in[-\infty,0)$, we have  $\mathcal{E}_{p,\hat{p}}(\mathcal{X})=[0,\alpha_{\max})$. 
\item For $(p,\hat{p})$ or $(\hat{p},p)\in[-\infty,0)\times(0,\infty]\backslash\{(-\infty,\infty)\}$,
we have $\mathcal{E}_{p,\hat{p}}(\mathcal{X})=(-\infty,0]$. 
\item For $(p,\hat{p})$ or $(\hat{p},p)=(-\infty,\infty)$, we have $\mathcal{E}_{p,\hat{p}}(\mathcal{X})=\{0\}$. 
\item For $(p,\hat{p})$ or $(\hat{p},p)\in\{0\}\times[0,\infty]$, we have
$\mathcal{E}_{p,\hat{p}}(\mathcal{X})=-\log P_{X}(\mathbb{B}_{\mathcal{X}})\backslash\{\infty\}:=\{-\log P_{X}(A):A\in\mathbb{B}_{\mathcal{X}}\}\backslash\{\infty\}$.
\item For $(p,\hat{p})$ or $(\hat{p},p)\in[-\infty,0)\times\{0\}$, we
have $\mathcal{E}_{p,\hat{p}}(\mathcal{X})=\{0\}$. 
\end{enumerate}
\end{prop}

\begin{proof}
We first prove Statement 1. For $p\in(0,\infty)$, by the equivalence
in \eqref{eq:}, $\mathcal{E}_{p,\hat{p}}(\mathcal{X})=\{D_{\hat{p}/p}(Q_{X}\|P_{X}):Q_{X}\ll P_{X}\}\backslash\{\pm\infty\}$.
Since by assumption, $P_{X}$ is not a Dirac measure, there is a measurable
set $A$ such that $P_{X}(A),P_{X}(A^{c})>0$. Denote the conditional
distribution of $P_{X}$ on $A$ as $P_{X|A}(\cdot)=P_{X}(\cdot\cap A)/P_{X}(A)$.
Then, we construct a distribution $Q_{X}=\lambda P_{X|A}+\bar{\lambda}P_{X|A^{c}}$
where $\bar{\lambda}=1-\lambda$ with $\lambda\in[0,1]$. For this
distribution, $D_{\hat{p}/p}(Q_{X}\|P_{X})=D_{\hat{p}/p}((\lambda,\bar{\lambda})\|(P_{X}(A),P_{X}(A^{c})))$,
which is continuous in $\hat{p}$ for $\hat{p}\in(0,\infty]$. Note
that when $\lambda=P_{X}(A)$, it holds that $D_{\hat{p}/p}(Q_{X}\|P_{X})=0$,
and when $\lambda=1$, it holds that $D_{\hat{p}/p}(Q_{X}\|P_{X})=-\log P_{X}(A)$.
By taking supremum of $-\log P_{X}(A)$ over all $A$ such that $P_{X}(A),P_{X}(A^{c})>0$,
we have $\mathcal{E}_{p,\hat{p}}(\mathcal{X})\supseteq[0,\alpha_{\max}]\backslash\{\infty\}$.
Obviously, if $\alpha_{\max}=\infty^{-}$, $\mathcal{E}_{p,\hat{p}}(\mathcal{X})=[0,\alpha_{\max}]\backslash\{\infty\}$.
Otherwise, $P_{X}$ is finitely supported. For such $P_{X}$ and any
$Q_{X}$, it holds that $D_{\hat{p}/p}(Q_{X}\|P_{X})\le D_{\infty}(Q_{X}\|P_{X})=\alpha_{\max}$,
which means $\mathcal{E}_{p,\hat{p}}(\mathcal{X})=[0,\alpha_{\max}]\backslash\{\infty\}$.
Hence, Statement 1 holds for either cases. If $p=\infty,\hat{p}\in(0,\infty)$,
then by symmetry, Statement 1 still holds. If $p=\hat{p}=\infty$,
then by definition, $\Ent_{\infty,\infty}(f)=-\log P_{X}(f=\Vert f\Vert_{\infty})$.
Hence, by setting $f=1_{A}$, Statement 1 still holds.

Statement 2 follows similarly, but note that for this case, 
\begin{equation}
\mathcal{E}_{p,\hat{p}}(\mathcal{X})=\{D_{\hat{p}/p}(Q_{X}\|P_{X}):Q_{X}\ll\gg P_{X}\}\backslash\{\pm\infty\},\label{eq:-94}
\end{equation}
 i.e., it additionally requires $Q_{X}\gg P_{X}$. This is because
we require that $f$ cannot take the value of $\infty$.  This difference
further leads to that if $\alpha_{\max}$ is finite, then $D_{\hat{p}/p}(Q_{X}\|P_{X})$
cannot achieve the upper bound $\alpha_{\max}$. This is because it
holds that $D_{\hat{p}/p}(Q_{X}\|P_{X})\le D_{\infty}(Q_{X}\|P_{X})\le\alpha_{\max}$
and the last inequality is strict unless $Q_{X}=1_{\{x^{*}\}}$ where
$x^{*}$ is an element that minimizes $P_{X}(x)$. Obviously, $P_{X}$
is not absolutely continuous with respect to $1_{\{x^{*}\}}$. Hence,
$\alpha_{\max}$ cannot be achieved. 

We now prove Statement 3. For this case, \eqref{eq:-94} still holds.
Observe that for any $s\in[-\infty,\infty]\backslash\{0,1\}$, $D_{s}(Q_{X}\|P_{X})=\frac{s}{1-s}D_{1-s}(P_{X}\|Q_{X})$
(called \emph{skew symmetry}) \cite{Erven}. Moreover, by the choice
of $Q_{X}$ same to the one in proof of Statement 1, $D_{1-s}(P_{X}\|Q_{X})=D_{1-s}((P_{X}(A),P_{X}(A^{c}))\|(\lambda,\bar{\lambda}))$
with $s=\hat{p}/p<0$ if $p\neq\pm\infty$, the RHS of which is continuous
in $\lambda$ and tends to infinity as $\lambda\downarrow0$. Therefore,
$\mathcal{E}_{p,\hat{p}}(\mathcal{X})=(-\infty,0]$ for $p\neq\pm\infty$.
If $p=\infty$ or $-\infty$, then by assumption, $\hat{p}\neq\pm\infty$.
For this case, we swap $p,\hat{p}$, apply the result above, and still
obtain $\mathcal{E}_{p,\hat{p}}(\mathcal{X})=(-\infty,0]$. 

Statements 4-6 follow by definition. In particular, to prove Statement
5, we set $f=1_{A}$ for some measurable $A$. To prove Statement
6, we need the fact that $\Vert f\Vert_{p}>0$ for $p<0$ implies
that $f>0$ $P_{X}$-a.e. and hence, $\Ent_{p,0}(f)=0$ for this case. 
\end{proof}
Define 
\[
\mathcal{E}_{p,\hat{p}}^{(n)}(\mathcal{X}):=\{\frac{1}{n}\Ent_{p,\hat{p}}(f):f\in\Phi(\mathcal{X}^{n})\}\backslash\{\pm\infty\}.
\]
Then, by Proposition \ref{prop:E}, we have the following consequence. 
\begin{cor}
For any $n$, $\mathcal{E}_{p,\hat{p}}^{(n)}(\mathcal{X})=\mathcal{E}_{p,\hat{p}}(\mathcal{X})$
for all $p,\hat{p}\in\overline{\mathbb{R}}$ except that $(p,\hat{p})$
or $(\hat{p},p)\in\{0\}\times[0,\infty]$. Moreover, for $(p,\hat{p})$
or $(\hat{p},p)\in\{0\}\times[0,\infty]$, $\mathcal{E}_{p,\hat{p}}^{(n)}(\mathcal{X})$
is equal to the Minkowski mean ($1/n$ times the Minkowski sum) of
$n$-copies of $\mathcal{E}_{p,\hat{p}}(\mathcal{X})$, and it also
holds that $\bigcup_{n\ge1}\mathcal{E}_{p,\hat{p}}^{(n)}(\mathcal{X})$
is a dense subset of $[0,\alpha_{\max}]$. 
\end{cor}

\begin{proof}
The last statement above follows since if $a<b$ are in $\mathcal{E}_{p,\hat{p}}(\mathcal{X})$,
then $\bar{\lambda}a+\lambda b=a+\lambda(b-a)$ is in $\mathcal{E}_{p,\hat{p}}^{(n)}(\mathcal{X})$
where $\lambda=k/n$ for any $k\in[0:n]$. As $n\to\infty$, $\bigcup_{n\ge1}\bigcup_{k\in[0:n]}\left\{ a+\frac{k}{n}(b-a)\right\} $
is dense in $[a,b]$, since any number $c$ in $[a,b]$ can be approached
by a sequence $\{a+\frac{k_{n}}{n}(b-a)\}_{n\ge1}$ for  $k_{n}=\bigl\lfloor\frac{c-a}{b-a}n\bigr\rfloor$. 
\end{proof}

\subsection{Problem Formulation}

Given $p,q,\hat{p},\hat{q}\in\overline{\mathbb{R}}$ and $P_{XY}$,
for $\alpha,\beta\in\mathbb{R}$, we define the \emph{optimal forward
and reverse BL exponents }as 
\begin{align}
\underline{\Delta}_{p,q,\hat{p},\hat{q}}(\alpha,\beta|P_{XY}) & :=-\sup_{f\in\mathcal{F}_{\alpha},g\in\mathcal{G}_{\beta}}\log\frac{\langle f,g\rangle}{\Vert f\Vert_{\hat{p}}\Vert g\Vert_{\hat{q}}}\label{eq:FBL-1-1-2}\\
\overline{\Delta}_{p,q,\hat{p},\hat{q}}(\alpha,\beta|P_{XY}) & :=-\inf_{f\in\mathcal{F}_{\alpha},g\in\mathcal{G}_{\beta}}\log\frac{\langle f,g\rangle}{\Vert f\Vert_{\hat{p}}\Vert g\Vert_{\hat{q}}},\label{eq:RBL-1-1-2}
\end{align}
where\footnote{Throughout this paper, when we write a function $f:\mathcal{X}\to[0,\infty)$,
by default, we mean that it is measurable. A similar convention also
applies to a function $g:\mathcal{Y}\to[0,\infty)$.} $\mathcal{F}_{\alpha}:=\{f\in\Phi(\mathcal{X}):\Ent_{p,\hat{p}}(f)=\alpha\}$
and $\mathcal{G}_{\beta}:=\{g\in\Phi(\mathcal{Y}):\Ent_{q,\hat{q}}(g)=\beta\}$.
In the definition above, we can restrict our attention to the case
of $\alpha\in\mathcal{E}_{p,\hat{p}}(\mathcal{X})$ and $\beta\in\mathcal{E}_{q,\hat{q}}(\mathcal{Y})$,
since otherwise, the optimal forward and reverse BL exponents above
are equal to $\infty$ or $-\infty$.

Define for $p,q\in\overline{\mathbb{R}},\hat{p},\hat{q}\in\overline{\mathbb{R}}\backslash\{0\}$,
\begin{align}
\underline{\Lambda}_{p,q,\hat{p},\hat{q}}(\alpha,\beta|P_{XY}) & :=\underline{\Delta}_{p,q,\hat{p},\hat{q}}(\alpha,\beta|P_{XY})+\frac{\alpha}{\hat{p}}+\frac{\beta}{\hat{q}}\label{eq:-20}\\
\overline{\Lambda}_{p,q,\hat{p},\hat{q}}(\alpha,\beta|P_{XY}) & :=\overline{\Delta}_{p,q,\hat{p},\hat{q}}(\alpha,\beta|P_{XY})+\frac{\alpha}{\hat{p}}+\frac{\beta}{\hat{q}}.\label{eq:-23}
\end{align}
Since the $(p,\hat{p})$-entropy is symmetric w.r.t. $(p,\hat{p})$,
it is easy to verify the following lemma. 
\begin{lem}
\label{lem:symmetry} For $p,q,\hat{p},\hat{q}\in\overline{\mathbb{R}}\backslash\{0\}$,
both $\underline{\Lambda}_{p,q,\hat{p},\hat{q}}$ and $\overline{\Lambda}_{p,q,\hat{p},\hat{q}}$
are symmetric w.r.t. $(p,\hat{p})$ and also symmetric w.r.t. $(q,\hat{q})$. 
\end{lem}

\begin{proof}
We focus on the functions $f\in\mathcal{F}_{\alpha},g\in\mathcal{G}_{\beta}$.
Otherwise, $\underline{\Lambda}_{p,q,\hat{p},\hat{q}}(\alpha,\beta|P_{XY})=-\infty$
and $\overline{\Delta}_{p,q,\hat{p},\hat{q}}(\alpha,\beta|P_{XY})=\infty$,
no matter what are $p,q,\hat{p},\hat{q}$. Hence, obviously they are
symmetric.

For $f\in\mathcal{F}_{\alpha},g\in\mathcal{G}_{\beta}$ and $\alpha,\beta\in\mathbb{R}$,
we have $(p,\hat{p})\in\mathcal{S}_{f}$. By definition, for $(p,\hat{p})\in(\mathbb{R}\backslash\{0\})^{2}\cap\mathcal{S}_{f}$,
we have 
\[
-\log\frac{\langle f,g\rangle}{\Vert f\Vert_{\hat{p}}\Vert g\Vert_{\hat{q}}}+\frac{\Ent_{p,\hat{p}}(f)}{\hat{p}}+\frac{\Ent_{q,\hat{q}}(g)}{\hat{q}}=-\log\frac{\langle f,g\rangle}{\Vert f\Vert_{p}\Vert g\Vert_{q}}+\frac{\Ent_{p,\hat{p}}(f)}{p}+\frac{\Ent_{q,\hat{q}}(g)}{q}.
\]
By taking limits, we can extend this equality to the case of $(p,\hat{p})\in(\overline{\mathbb{R}}\backslash\{0\})^{2}\cap\mathcal{S}_{f}$.
Taking $\sup_{f\in\mathcal{F}_{\alpha},g\in\mathcal{G}_{\beta}}$
or $\inf_{f\in\mathcal{F}_{\alpha},g\in\mathcal{G}_{\beta}}$, we
obtain this lemma. 
\end{proof}
We now extend the definitions in \eqref{eq:-20} and \eqref{eq:-23}
to the case of $(p,\hat{p}),(q,\hat{q})\in\overline{\mathbb{R}}^{2}\backslash\{(0,0)\}$
by the symmetric extension. Such an extension ensures that $\underline{\Lambda}_{p,q,\hat{p},\hat{q}}$
and $\overline{\Lambda}_{p,q,\hat{p},\hat{q}}$ are still symmetric
w.r.t. $(p,\hat{p})$ and symmetric w.r.t. $(q,\hat{q})$ for $(p,\hat{p}),(q,\hat{q})\in\overline{\mathbb{R}}^{2}\backslash\{(0,0)\}$.
Instead of $\underline{\Delta}_{p,q,\hat{p},\hat{q}}(\alpha,\beta|P_{XY})$
and $\overline{\Delta}_{p,q,\hat{p},\hat{q}}(\alpha,\beta|P_{XY})$,
equivalently, in this paper we aim at characterizing $\underline{\Lambda}_{p,q,\hat{p},\hat{q}}(\alpha,\beta|P_{XY})$
and $\overline{\Lambda}_{p,q,\hat{p},\hat{q}}(\alpha,\beta|P_{XY})$.
For brevity, sometimes we omit the underlying distribution $P_{XY}$
in these notations, and write them as $\underline{\Lambda}_{p,q,\hat{p},\hat{q}}(\alpha,\beta)$
and $\overline{\Lambda}_{p,q,\hat{p},\hat{q}}(\alpha,\beta)$.

Let $P_{XY}^{\otimes n}$ be the product of $n$ copies of $P_{XY}$.
For $P_{XY}^{\otimes n}$, we define the \emph{forward} \emph{and
reverse $n$-BL exponents} respectively as $\underline{\Lambda}_{p,q,\hat{p},\hat{q}}^{(n)}(\alpha,\beta):=\frac{1}{n}\underline{\Lambda}_{p,q,\hat{p},\hat{q}}(n\alpha,n\beta|P_{XY}^{\otimes n})$
and $\overline{\Lambda}_{p,q,\hat{p},\hat{q}}^{(n)}(\alpha,\beta):=\frac{1}{n}\overline{\Lambda}_{p,q,\hat{p},\hat{q}}(n\alpha,n\beta|P_{XY}^{\otimes n})$,
as well as, the \emph{forward} \emph{and reverse asymptotic BL exponents
}respectively as $\underline{\Lambda}_{p,q,\hat{p},\hat{q}}^{(\infty)}(\alpha,\beta):=\lim_{n\to\infty}\underline{\Lambda}_{p,q,\hat{p},\hat{q}}^{(n)}(\alpha,\beta)$
and $\overline{\Lambda}_{p,q,\hat{p},\hat{q}}^{(\infty)}(\alpha,\beta):=\lim_{n\to\infty}\overline{\Lambda}_{p,q,\hat{p},\hat{q}}^{(n)}(\alpha,\beta)$.
Obviously, $\underline{\Lambda}_{p,q,\hat{p},\hat{q}}^{(1)}(\alpha,\beta)=\underline{\Lambda}_{p,q,\hat{p},\hat{q}}(\alpha,\beta)$
and $\overline{\Lambda}_{p,q,\hat{p},\hat{q}}^{(1)}(\alpha,\beta)=\overline{\Lambda}_{p,q,\hat{p},\hat{q}}(\alpha,\beta)$.
Upon these definitions, two natural questions arise: Can we derive
sharp dimension-free bounds (also termed single-letter bounds) for
$\underline{\Lambda}_{p,q,\hat{p},\hat{q}}^{(n)}(\alpha,\beta)$ and
$\overline{\Lambda}_{p,q,\hat{p},\hat{q}}^{(n)}(\alpha,\beta)$? Does
the tensorization property holds for $\underline{\Lambda}_{p,q,\hat{p},\hat{q}}^{(n)}(\alpha,\beta)$
and $\overline{\Lambda}_{p,q,\hat{p},\hat{q}}^{(n)}(\alpha,\beta)$
(i.e., $\underline{\Lambda}_{p,q,\hat{p},\hat{q}}^{(n)}(\alpha,\beta)=\underline{\Lambda}_{p,q,\hat{p},\hat{q}}(\alpha,\beta)$
and $\overline{\Lambda}_{p,q,\hat{p},\hat{q}}^{(n)}(\alpha,\beta)=\overline{\Lambda}_{p,q,\hat{p},\hat{q}}(\alpha,\beta)$
for any $n$)? In this paper, we study these two questions. Specifically,
by information-theoretic methods, we derive dimension-free bounds
for $\underline{\Lambda}_{p,q,\hat{p},\hat{q}}^{(n)}(\alpha,\beta)$
and $\overline{\Lambda}_{p,q,\hat{p},\hat{q}}^{(n)}(\alpha,\beta)$,
which are asymptotically sharp for the case of finite $\mathcal{X},\mathcal{Y}$,
as the dimension $n\to\infty$. We observe that these expressions
differ from $\underline{\Lambda}_{p,q,\hat{p},\hat{q}}(\alpha,\beta)$
and $\overline{\Lambda}_{p,q,\hat{p},\hat{q}}(\alpha,\beta)$, which
implies that the tensorization property does not hold in general.

The work in this paper is motivated by the works in \cite{polyanskiy2019improved,kirshner2019moment}.
As special cases of BL inequalities, the forward hypercontractivity
inequalities were strengthened in the same spirit in \cite{polyanskiy2019improved,kirshner2019moment}.
Both works in \cite{polyanskiy2019improved,kirshner2019moment} only
focused on strengthening the single-function version of forward hypercontractivity.
Polyanskiy and Samorodnitsky's inequalities in \cite{polyanskiy2019improved}
are sharp only for extreme cases, while Kirshner and Samorodnitsky
only focused on binary symmetric distributions in \cite{kirshner2019moment}.
Furthermore, hypercontractivity inequalities for distributions on
finite alphabets $\mathcal{X},\mathcal{Y}$ were recently studied
by the present author, Anantharam, and Chen \cite{yu2021Graphs},
in which they focused on the case of $\hat{p}=\hat{q}=0$. Our work
here generalizes all these works to the BL inequalities for all $\hat{p},\hat{q}\in\overline{\mathbb{R}}$
and for arbitrary distributions on Polish spaces. Moreover, our strong
BL inequalities are exponentially sharp at least for distributions
on finite alphabets.

The forward version of BL inequalities was originally studied by Brascamp
and Lieb in the 1970s \cite{brascamp1976best}, motivated by problems
from particle physics; the reverse version was initially studied by
Barthe in \cite{barthe1998reverse}. Hypercontractivity inequalities
are an important class of BL inequalities, which were sequentially
investigated in \cite{bonami1968ensembles,kiener1969uber,schreiber1969fermeture,bonami1970etude,gross1975logarithmic,ahlswede1976spreading,borell1982positivity,mossel2013reverse}.
Information-theoretic formulations of BL inequalities or hypercontractivity
inequalities can be found in \cite{ahlswede1976spreading,carlen2009subadditivity,nair2014equivalent,beigi2016equivalent,kamath2015reverse,liu2018information}.
Euclidean versions of BL inequalities and their generalizations were
also studied in \cite{bennett2008brascamp,courtade2020euclidean}.

\subsection{Main Contributions}

In this paper, we study nonlinear Brascamp--Lieb inequalities, and
derive sharp dimension-free version of nonlinear Brascamp--Lieb inequalities
(including hypercontractivity inequalities) for distributions on Polish
spaces, which strengthen the classic Brascamp--Lieb inequalities.
As applications of our nonlinear Brascamp--Lieb inequalities, we
strengthen and extend the Mr. and Mrs. Gerber's lemmas \cite{wyner1973theorem,hsu2018generalizing}
to the sharp Rényi divergence versions, and the small-set expansion
theorems to strong versions for arbitrary distributions on Polish
spaces. Also, by our nonlinear Brascamp--Lieb inequalities, we obtain
sharp dimension-free bounds on the $q$-stability of Boolean functions.
Our proofs in this paper are based on information-theoretic techniques
and coupling techniques.

\subsection{Preliminaries and Notations}

We use $P_{X},Q_{X},R_{X},S_{X}$ to denote probability measures on
$(\mathcal{X},\mathbb{B}_{\mathcal{X}})$. For a joint probability
measure\footnote{The notation $Q_{XY}$ denotes a joint distribution, rather than the
distribution of the product of $X$ and $Y$. This is not ambiguous
in this paper, since throughout this paper, we never consider the
product of two random variables. } $Q_{XY}$ on $(\mathcal{X}\times\mathcal{Y},\mathbb{B}_{\mathcal{X}}\otimes\mathbb{B}_{\mathcal{Y}})$,
its marginal on $\mathcal{X}$ is denoted as $Q_{X}$, and the  
Markov kernel from $(\mathcal{Y},\mathbb{B}_{\mathcal{Y}})$ to $(\mathcal{X},\mathbb{B}_{\mathcal{X}})$
is denoted as $Q_{X|Y}$, where $Q_{X|Y=y}:=Q_{X|Y}(\cdot|y)$ is
a probability measure on $(\mathcal{X},\mathbb{B}_{\mathcal{X}})$
for every $y\in\mathcal{Y}$, and $Q_{X|Y}(B|\cdot)$ for each $B\in\mathbb{B}_{\mathcal{X}}$
is $\mathbb{B}_{\mathcal{Y}}$-measurable. We denote $R_{Y}R_{X|Y}$
as the joint distribution induced by $R_{Y}$ and $R_{X|Y}$, and
$R_{Y}\circ R_{X|Y}$ as the marginal distribution on $\mathcal{X}$
of the joint distribution $R_{Y}R_{X|Y}$. We use $Q_{X}\ll P_{X}$
to denote that the distribution $Q_{X}$ is absolutely continuous
w.r.t. $P_{X}$. We use $X^{n}$ to denote a random vector $(X_{1},X_{2},...,X_{n})$
defined on $(\mathcal{X}^{n},\mathbb{B}_{\mathcal{X}}^{\otimes n})$,
and use $x^{n}:=(x_{1},x_{2},...,x_{n})$ to denote its realization.
For an $n$-length vector $x^{n}$, we use $x^{i}$ to denote the
subvector consisting of the first $i$ components of $x^{n}$, and
$x_{i+1}^{n}$ to denote the subvector consisting of the last $n-i$
components. We use $Q_{X}^{\otimes n}$ to denote the product of $n$
copies of $Q_{X}$, and use $Q_{X^{n}}$ to denote an arbitrary probability
measure on $(\mathcal{X}^{n},\mathbb{B}_{\mathcal{X}}^{\otimes n})$.
For a conditional probability measure $P_{X|Y}$, define the conditional
expectation operator induced by $P_{X|Y}$ as $P_{X|Y}(f)(y):=\int f\mathrm{d}P_{X|Y=y}$
for any measurable function $f:(\mathcal{X},\mathbb{B}_{\mathcal{X}})\to(\mathbb{R},\mathbb{B}_{\mathbb{R}})$
if the integral is well-defined for every $y$. We say $U,W,V$ forms
a Markov chain, denoted as $U\leftrightarrow W\leftrightarrow V$,
if $U,V$ are conditionally independent given $W$. We use $\mathcal{C}(Q_{X},Q_{Y})$
to denote the set of couplings (joint probability measures) $Q_{XY}$
with marginals $Q_{X},Q_{Y}$, and $\mathcal{C}(Q_{X|UW},Q_{Y|VW})$
to denote the set of conditional couplings $Q_{XY|UVW}$ with conditional
marginals $Q_{X|UW},Q_{Y|VW}$. Note that for any $Q_{XY|UVW}\in\mathcal{C}(Q_{X|UW},Q_{Y|VW}),$
its marginals satisfy $Q_{X|UVW}=Q_{X|UW},Q_{Y|UVW}=Q_{Y|VW}$, i.e.,
under the conditional distribution $Q_{XY|UVW}$, $X\leftrightarrow(U,W)\leftrightarrow V$
and $Y\leftrightarrow(V,W)\leftrightarrow U$. For a sequence $x^{n}$,
we use $T_{x^{n}}$ to denote the empirical distribution (i.e., type)
of $x^{n}$.


Throughout this paper, we use the following convention.

\begin{convention}\label{Convention-:-When} When we write an optimization
problem with distributions as the variables, we by default require
that the distributions satisfy that all the constraint functions and
the objective function exist and also are finite. If there is no such
a distribution, by default, the value of the optimization problem
is set to $\infty$ if the optimization is an infimization, and $-\infty$
if the optimization is a supremization.\emph{ }\end{convention}


Throughout this paper, when we talk about distributions $Q_{X},R_{X},S_{X}$
and conditional distributions $Q_{X|W},R_{X|W},S_{X|W}$, we mean
that $Q_{X},R_{X},S_{X},Q_{X|W=w},R_{X|W=w},S_{X|W=w}\ll P_{X}$ for
each $w$. For brevity, we do not mention these underlying constraints
throughout this paper. The same convention also applies to $Q_{Y},R_{Y},S_{Y},Q_{Y|W},R_{Y|W},S_{Y|W}$
($\ll P_{Y}$). 

We denote $[m:n]:=\{m,m+1,...,n\}$ and $[n]:=[1:n]$. Throughout
this paper, we use the conventions $\inf\emptyset=\infty$ and $\sup\emptyset=-\infty$.
We denote $x\vee y:=\max\{x,y\}$ and $[x]^{+}:=x\vee0$. Denote $q':=\frac{q}{q-1}$
as the Hölder conjugate of $q\neq1$, and both $\pm\infty$ are the
Hölder conjugates of $1$.

If $f:S\to\mathbb{R}$ is a real-valued convex function defined on
a convex set $S\subseteq\mathbb{R}^{n}$, a vector \textbf{$\mathbf{v}\in\mathbb{R}^{n}$}
is called a \emph{subgradient} of $f$ at a point $\mathbf{x}_{0}$
in $S$ if for any $\mathbf{x}$ in $S$ one has $f(\mathbf{x})-f(\mathbf{x}_{0})\geq\langle\mathbf{v},\mathbf{x}-\mathbf{x}_{0}\rangle$
where the $\langle\cdot,\cdot\rangle$ denotes the inner product.
Equivalently, $\mathbf{x}\in S\mapsto f(\mathbf{x}_{0})+\langle\mathbf{v},\mathbf{x}-\mathbf{x}_{0}\rangle$
forms a supporting hyperplane of the epigraph of $f$ at $\mathbf{x}_{0}$.
If additionally, $f$ is differentiable at $\mathbf{x}_{0}$, then
$\mathbf{v}$ is the gradient of $f$ at $\mathbf{x}_{0}$. A vector
\textbf{$\mathbf{v}$} is called a \emph{supergradient} of a concave
function $g$ at a point $\mathbf{x}_{0}$ if \textbf{$-\mathbf{v}$}
is a subgradient of $-g$ at $\mathbf{x}_{0}$.

We say an inequality $a_{n}(x_{n})\le b_{n}(x_{n}),\forall x_{n}\in A_{n}$
for two positive sequences of functions $\{a_{n}\},\{b_{n}\}$ to
be exponentially sharp, if there exists a sequence $x_{n}^{*}\in A_{n}$
such that $\frac{1}{n}\log\frac{b_{n}(x_{n}^{*})}{a_{n}(x_{n}^{*})}\to0$
as $n\to\infty$.

\section{\label{sec:Strong-BL-and}Strong BL and HC Inequalities: Two-Function
Version}

\subsection{Information-Theoretic Characterizations}

Define 
\begin{align}
\phi(Q_{X},Q_{Y}) & :=\inf_{R_{XY}}\{D(R_{XY}\|P_{XY})+\frac{1}{p}D(R_{X}\|Q_{X})-\frac{1}{p}D(R_{X}\|P_{X})+\frac{1}{q}D(R_{Y}\|Q_{Y})-\frac{1}{q}D(R_{Y}\|P_{Y})\},\label{eq:phi}
\end{align}
where according to Convention \ref{Convention-:-When}, the infimization
is taken over all $R_{XY}$ such that all the relative entropies appearing
in the objective function are finite. We now provide information-theoretic
characterizations of $\underline{\Lambda}_{p,q,\hat{p},\hat{q}}$
and $\overline{\Lambda}_{p,q,\hat{p},\hat{q}}$, in the following
proposition. The proof is provided in Appendix \ref{sec:Proof-of-Proposition-ITChara}. 
\begin{prop}[Information-Theoretic Characterizations]
\label{prop:ITcharacterization} For $p,q\in\mathbb{R}\backslash\{0\},\hat{p},\hat{q}\in\overline{\mathbb{R}}$,
\begin{align*}
\underline{\Lambda}_{p,q,\hat{p},\hat{q}}(\alpha,\beta) & =\inf_{\substack{Q_{X},Q_{Y}:\\
D_{{\hat{p}}/{p}}(Q_{X}\|P_{X})=\alpha,\\
D_{{\hat{q}}/{q}}(Q_{Y}\|P_{Y})=\beta
}
}\phi(Q_{X},Q_{Y})+\frac{\alpha}{p}+\frac{\beta}{q},\\
\overline{\Lambda}_{p,q,\hat{p},\hat{q}}(\alpha,\beta) & =\sup_{\substack{Q_{X},Q_{Y}:\\
D_{{\hat{p}}/{p}}(Q_{X}\|P_{X})=\alpha,\\
D_{{\hat{q}}/{q}}(Q_{Y}\|P_{Y})=\beta
}
}\phi(Q_{X},Q_{Y})+\frac{\alpha}{p}+\frac{\beta}{q},
\end{align*}
where in these two optimization problems, $Q_{X}\ll P_{X}$ if $p>0,\hat{p}\ge0$,
otherwise, $Q_{X}\ll\gg P_{X}$, and similarly, $Q_{Y}\ll P_{Y}$
if $q>0,\hat{q}\ge0$, otherwise, $Q_{Y}\ll\gg P_{Y}$. 
\end{prop}

\subsection{Brascamp--Lieb Exponents}

Utilizing the information-theoretic expressions in Proposition \ref{prop:ITcharacterization},
we next provide dimension-free bounds for the $n$-dimensional versions
$\underline{\Lambda}_{p,q,\hat{p},\hat{q}}^{(n)}$ and $\overline{\Lambda}_{p,q,\hat{p},\hat{q}}^{(n)}$.
Let $(\mathcal{W},\mathbb{B}_{\mathcal{W}},Q_{W})$ be a Polish probability
measure space. Define the \emph{minimum relative entropy} over couplings
of $(Q_{X},Q_{Y})$ with respect to $P_{XY}$ as 
\begin{equation}
\mathbb{D}(Q_{X},Q_{Y}\|P_{XY}):=\inf_{Q_{XY}\in\mathcal{C}(Q_{X},Q_{Y})}D(Q_{XY}\|P_{XY}),\label{eq:mathbbD}
\end{equation}
and the \emph{minimum conditional relative entropy} over couplings
of $(Q_{X|W},Q_{Y|W})$, with respect to $P_{XY}$ and conditionally
on $Q_{W}$, as 
\begin{equation}
\mathbb{D}(Q_{X|W},Q_{Y|W}\|P_{XY}|Q_{W}):=\mathbb{E}_{\hat{W}\sim Q_{W}}\mathbb{D}(Q_{X|W=\hat{W}},Q_{Y|W=\hat{W}}\|P_{XY}).\label{eq:mathbbD-1-1}
\end{equation}
The infimization in \eqref{eq:mathbbD} is termed the \emph{Schr\"odinger
problem }or the\emph{ maximum entropy problem} \cite{leonard2001minimization,leonard2013survey}.
In fact, this problem admits the following dual formula. 
\begin{prop}[Duality of Schr\"odinger Problem]
\cite{leonard2001minimization,leonard2013survey} For Polish $\mathcal{X},\mathcal{Y}$,
\[
\mathbb{D}(Q_{X},Q_{Y}\|P_{XY})=\sup_{f,g}\varphi(f,g),
\]
where $\varphi(f,g):=-\log\int_{\mathcal{X}\times\mathcal{Y}}fg\mathrm{d}P_{XY}+\int_{\mathcal{X}}\log f\mathrm{d}\mu_{X}+\int_{\mathcal{Y}}\log g\mathrm{d}\mu_{Y}$
and the supremum is taken over all continuous bounded functions $f:\mathcal{X}\to\mathbb{R},g:\mathcal{Y}\to\mathbb{R}$.
Moreover, if the infimum in \eqref{eq:mathbbD} is finite, then it
is attained, and hence, is a minimum. 
\end{prop}

Define 
\begin{align*}
\underline{\Lambda}_{p,q,\hat{p},\hat{q}}^{*}(\alpha,\beta) & :=\inf_{Q_{W},Q_{X|W},Q_{Y|W}}\mathbb{D}(Q_{X|W},Q_{Y|W}\|P_{XY}|Q_{W})\\
 & \qquad+\eta_{p,\hat{p}}(\alpha,D(Q_{X|W}\|P_{X}|Q_{W}))+\eta_{q,\hat{q}}(\beta,D(Q_{Y|W}\|P_{Y}|Q_{W}))
\end{align*}
where\footnote{Throughout this paper, we interpret $\frac{a}{0}$ as $\infty$ for
$a>0$, as $-\infty$ for $a<0$, and as $0$ for $a=0$. We interpret
$\frac{a}{\pm\infty}$ as $0$ for finite $a$.} 
\[
\eta_{p,\hat{p}}(\alpha,s):=\begin{cases}
\frac{\alpha-s}{p}\vee\frac{\alpha-s}{\hat{p}} & p,\hat{p}\ge0\\
\frac{\alpha-s}{p} & p>0>\hat{p}\\
\frac{\alpha-s}{\hat{p}} & \hat{p}>0>p\\
-\infty & p,\hat{p}\le0\textrm{ and }(p,\hat{p})\neq(0,0)
\end{cases}.
\]
Here the infimization is taken over all probability measure spaces
$(\mathcal{W},\mathbb{B}_{\mathcal{W}},Q_{W})$ on Polish spaces $\mathcal{W}$
and regular conditional probability measures $Q_{X|W},Q_{Y|W}$. However,
by Carathéodory's theorem, without loss of optimality, it suffices
to restrict $|\mathcal{W}|\le4$. Furthermore, obviously, $\eta_{p,\hat{p}}(\alpha,s)$
is nonincreasing in $s$.

According to the signs of $p,q,\hat{p},\hat{q}$, we partition the
distributions $\{Q_{X|W},Q_{Y|W},\hat{Q}_{X|W},\hat{Q}_{Y|W}\}$ into
two disjoint subsets $\mathbf{Q}^{+},\mathbf{Q}^{-}$. Specifically,
$Q_{X|W}\in\mathbf{Q}^{+}$ if $p\ge0$; $Q_{X|W}\in\mathbf{Q}^{-}$
otherwise. Similarly, $Q_{Y|W},\hat{Q}_{X|W},\hat{Q}_{Y|W}$ are assigned
into $\mathbf{Q}^{+},\mathbf{Q}^{-}$ respectively according to the
signs of $q,\hat{p},\hat{q}$. Define 
\begin{align*}
\theta_{p,q}(Q_{W},Q_{X|W},Q_{Y|W}) & :=\mathbb{D}(Q_{X|W},Q_{Y|W}\|P_{XY}|Q_{W})+\frac{\alpha-D(Q_{X|W}\|P_{X}|Q_{W})}{p}+\frac{\beta-D(Q_{Y|W}\|P_{Y}|Q_{W})}{q},
\end{align*}
and 
\begin{align*}
\alpha_{\lambda}^{-} & :=\begin{cases}
\alpha & \lambda\le1;\\
-\infty & \textrm{otherwise},
\end{cases}\qquad\alpha_{\lambda}^{+}:=\begin{cases}
\alpha & \lambda\ge1;\\
\infty & \textrm{otherwise}
\end{cases}\\
\beta_{\mu}^{-} & :=\begin{cases}
\beta & \mu\le1;\\
-\infty & \textrm{otherwise},
\end{cases}\qquad\beta_{\mu}^{+}:=\begin{cases}
\beta & \mu\ge1;\\
\infty & \textrm{otherwise}.
\end{cases}
\end{align*}
Define 
\begin{align*}
\overline{\Lambda}_{p,q,\hat{p},\hat{q}}^{*}(\alpha,\beta) & :=\sup_{Q_{W},\mathbf{Q}^{+}}\inf_{\mathbf{Q}^{-}}\min\Bigl\{\theta_{p,q}(Q_{W},Q_{X|W},Q_{Y|W}),\theta_{p,\hat{q}}(Q_{W},Q_{X|W},\hat{Q}_{Y|W}),\\
 & \qquad\theta_{\hat{p},q}(Q_{W},\hat{Q}_{X|W},Q_{Y|W}),\theta_{\hat{p},\hat{q}}(Q_{W},\hat{Q}_{X|W},\hat{Q}_{Y|W})\Bigr\},
\end{align*}
where the infimization is taken over all the tuples of distributions
in $\mathbf{Q}^{-}$, and the supremization is taken over all Polish
probability measure spaces $(\mathcal{W},\mathbb{B}_{\mathcal{W}},Q_{W})$
and all the tuples of distributions in $\mathbf{Q}^{+}$ under the
constraints\footnote{Although here the relative entropies are constrained to be $\le\infty$
or $\ge-\infty$, we should notice that the relative entropies are
in fact always $\ge0$ and by Convention \ref{Convention-:-When},
they are indeed constrained to be $<\infty$. } $\alpha_{\hat{p}/p}^{-}\le D(Q_{X|W}\|P_{X}|Q_{W})\le\alpha_{\hat{p}/p}^{+}$
and $\alpha_{p/\hat{p}}^{-}\le D(\hat{Q}_{X|W}\|P_{X}|Q_{W})\le\alpha_{p/\hat{p}}^{+}$
if $p,\hat{p}\ge0$, as well as, $\beta_{\hat{q}/q}^{-}\le D(Q_{Y|W}\|P_{Y}|Q_{W})\leq\beta_{\hat{q}/q}^{+}$
and $\beta_{q/\hat{q}}^{-}\le D(\hat{Q}_{Y|W}\|P_{Y}|Q_{W})\leq\beta_{q/\hat{q}}^{+}$
if $q,\hat{q}\ge0$. Similarly to the case of $\underline{\Lambda}_{p,q,\hat{p},\hat{q}}^{*}$,
by Carathéodory's theorem, without loss of optimality, it suffices
to restrict $|\mathcal{W}|\le9$.

We provide dimension-free bounds for $\underline{\Lambda}_{p,q,\hat{p},\hat{q}}^{(n)}$
and $\overline{\Lambda}_{p,q,\hat{p},\hat{q}}^{(n)}$ in the following
theorem. The proof is given in Appendix \ref{sec:Proof-of-Theorem-BLexponent}. 
\begin{thm}[Brascamp--Lieb Exponents]
\label{thm:BLexponent} For any $n\ge1$, $p,q\in\mathbb{R}\backslash\{0\},\hat{p},\hat{q}\in\overline{\mathbb{R}}$,
$\alpha\in\mathcal{E}_{p,\hat{p}}^{(n)}(\mathcal{X})$, and $\beta\in\mathcal{E}_{q,\hat{q}}^{(n)}(\mathcal{Y})$,
we have 
\begin{align}
\underline{\Lambda}_{p,q,\hat{p},\hat{q}}^{(n)}(\alpha,\beta) & \geq\underline{\Lambda}_{p,q,\hat{p},\hat{q}}^{*}(\alpha,\beta)\label{eq:FBLE}\\
\overline{\Lambda}_{p,q,\hat{p},\hat{q}}^{(n)}(\alpha,\beta) & \leq\overline{\Lambda}_{p,q,\hat{p},\hat{q}}^{*}(\alpha,\beta),\label{eq:RBLE}
\end{align}
where $\mathcal{E}_{p,\hat{p}}(\mathcal{X})$ was defined in \eqref{eq:E}.
Moreover, for finite alphabets $\mathcal{X},\mathcal{Y}$, these two
inequalities are asymptotically tight as $n\to\infty$ for given $(p,q,\hat{p},\hat{q},\alpha,\beta)$
as described above but with $\hat{p},\hat{q}\neq0$, $\alpha\neq0,\alpha_{\max}$
and $\beta\neq0,\beta_{\max}$, i.e., 
\begin{align}
\underline{\Lambda}_{p,q,\hat{p},\hat{q}}^{(\infty)}(\alpha,\beta) & =\underline{\Lambda}_{p,q,\hat{p},\hat{q}}^{*}(\alpha,\beta),\qquad\overline{\Lambda}_{p,q,\hat{p},\hat{q}}^{(\infty)}(\alpha,\beta)=\overline{\Lambda}_{p,q,\hat{p},\hat{q}}^{*}(\alpha,\beta).\label{eq:-95}
\end{align}
\end{thm}

\begin{rem}
Note that $\underline{\Lambda}_{p,q,\hat{p},\hat{q}}^{*}(\alpha,\beta)$
(resp. $\overline{\Lambda}_{p,q,\hat{p},\hat{q}}^{*}(\alpha,\beta)$)
may be neither always positive nor always negative. 
\end{rem}

\begin{rem}
Since for discrete distributions, $\mathcal{E}_{p,0}^{(n)}(\mathcal{X})$
is a countable set and it changes with $n$. Hence, for the case of
$\hat{p}=0$ or $\hat{q}=0$, the asymptotic tightness as in \eqref{eq:-95}
cannot hold for fixed $\alpha,\beta$. In fact, for this case, for
any $(\alpha,\beta)$ such that $\alpha\neq0,\alpha_{\max}$ and $\beta\neq0,\beta_{\max}$,
it holds that $\underline{\Lambda}_{p,q,0,0}^{(n)}(\alpha_{n},\beta_{n})\to\underline{\Lambda}_{p,q,0,0}^{*}(\alpha,\beta)$
and $\overline{\Lambda}_{p,q,0,0}^{(n)}(\alpha_{n},\beta_{n})\to\overline{\Lambda}_{p,q,0,0}^{*}(\alpha,\beta)$
as $n\to\infty$ for some sequence $(\alpha_{n},\beta_{n})\to(\alpha,\beta)$. 
\end{rem}

\subsection{Strong Brascamp--Lieb Inequalities}

Theorem \ref{thm:BLexponent} can be rewritten as the following form. 
\begin{cor}[Strong Brascamp--Lieb Inequalities (Two-Function Version)]
\label{cor:BL} For $p,q\in\mathbb{R}\backslash\{0\},\hat{p},\hat{q}\in\overline{\mathbb{R}}$
and $P_{XY}^{\otimes n}$, 
\begin{align}
\langle f,g\rangle & \leq e^{-n(\underline{\Lambda}_{p,q,\hat{p},\hat{q}}^{*}(\alpha,\beta)-\frac{\alpha}{p}-\frac{\beta}{q})}\Vert f\Vert_{p}\Vert g\Vert_{q}\label{eq:-6}\\
\langle f,g\rangle & \geq e^{-n(\overline{\Lambda}_{p,q,\hat{p},\hat{q}}^{*}(\alpha,\beta)-\frac{\alpha}{p}-\frac{\beta}{q})}\Vert f\Vert_{p}\Vert g\Vert_{q},\label{eq:-7}
\end{align}
for all functions $f:\mathcal{X}^{n}\to[0,\infty),g:\mathcal{Y}^{n}\to[0,\infty)$
such that $\alpha,\beta\in\mathbb{R}$, where $\alpha=\frac{1}{n}\Ent_{p,\hat{p}}(f),\beta=\frac{1}{n}\Ent_{q,\hat{q}}(g)$. 
\end{cor}

\begin{rem}
By symmetry, the inequalities \eqref{eq:-6} and \eqref{eq:-7} still
hold if $p,\hat{p}$ (or $q,\hat{q}$) are swapped. 
\end{rem}

The expressions of $\underline{\Lambda}_{p,q,\hat{p},\hat{q}}^{*}$
and $\overline{\Lambda}_{p,q,\hat{p},\hat{q}}^{*}$ are somewhat complicated.
We next simplify the inequalities \eqref{eq:-6} and \eqref{eq:-7}
 by choosing $(\hat{p},\hat{q})$ as some specific function of $(\alpha,\beta)$.

Define the \emph{minimum-relative-entropy region} of $P_{XY}$ as
\[
\mathcal{D}(P_{XY}):=\bigcup_{Q_{X}\ll P_{X},Q_{Y}\ll P_{Y}}\{(D(Q_{X}\|P_{X}),D(Q_{Y}\|P_{Y}),\mathbb{D}(Q_{X},Q_{Y}\|P_{XY}))\}.
\]
Define its lower and upper envelopes as for $s,t\ge0$, 
\begin{align}
\underline{\varphi}(s,t) & :=\inf_{Q_{XY}:D(Q_{X}\|P_{X})=s,D(Q_{Y}\|P_{Y})=t}D(Q_{XY}\|P_{XY})\nonumber \\
 & =\inf_{Q_{X},Q_{Y}:D(Q_{X}\|P_{X})=s,D(Q_{Y}\|P_{Y})=t}\mathbb{D}(Q_{X},Q_{Y}\|P_{XY}),\label{eq:phiUnderline}
\end{align}
\begin{align}
\overline{\varphi}(s,t) & :=\sup_{Q_{X},Q_{Y}:D(Q_{X}\|P_{X})=s,D(Q_{Y}\|P_{Y})=t}\mathbb{D}(Q_{X},Q_{Y}\|P_{XY}).\label{eq:phiOverline}
\end{align}
We also define for $q<0$ and $s\ge0$, 
\begin{equation}
\varphi_{q}(s):=\sup_{Q_{X}:D(Q_{X}\|P_{X})=s}\inf_{Q_{Y}}\mathbb{D}(Q_{X},Q_{Y}\|P_{XY})-\frac{D(Q_{Y}\|P_{Y})}{q}.\label{eq:phi_q}
\end{equation}
Define $\breve{\varphi}(s,t)$ as the lower convex envelope of $\underline{\varphi}(s,t)$,
and $\invbreve\varphi(s,t),\invbreve\varphi_{q}(s)$ respectively
as the upper concave envelopes of $\overline{\varphi}(s,t),\varphi_{q}(s)$.
The \emph{increasing} lower convex envelope of $\underline{\varphi}(s,t)$,
and the \emph{increasing} upper concave envelopes of $\overline{\varphi}(s,t),\varphi_{q}(s)$
are defined as 
\begin{align}
\underline{\Theta}(\alpha,\beta) & :=\inf_{s\ge\alpha,t\ge\beta}\breve{\varphi}(s,t)\label{eq:ThetaUnderline}\\
\overline{\Theta}(\alpha,\beta) & :=\sup_{0\le s\leq\alpha,0\le t\leq\beta}\invbreve\varphi(s,t)\label{eq:ThetaOverline}\\
\Theta_{q}(\alpha) & :=\sup_{0\le s\leq\alpha}\invbreve\varphi_{q}(s).\label{eq:ThetaOverline_q}
\end{align}
Note that by using function $\underline{\varphi}$, we can rewrite
$\underline{\Lambda}_{p,q,\hat{p},\hat{q}}^{*}$ as 
\begin{align}
\underline{\Lambda}_{p,q,\hat{p},\hat{q}}^{*}(\alpha,\beta) & =\inf_{s,t}\underline{\varphi}(s,t)+\eta_{p,\hat{p}}(\alpha,s)+\eta_{q,\hat{q}}(\beta,t).\label{eq:-27}
\end{align}
The objective function of the infimization in \eqref{eq:-27} is difference
between the surface $\underline{\varphi}(s,t)$ and the surface $-\eta_{p,\hat{p}}(\alpha,s)-\eta_{q,\hat{q}}(\beta,t)$,
where the latter surface consists of parts of several planes (here
for the case of $p,\hat{p}<0$ or $q,\hat{q}<0$, we regard the constant
function $(x,y)\in\mathbb{R}^{2}\mapsto\infty$ as a ``plane'' as
well). 
\begin{lem}
$\underline{\Theta}(\alpha,\beta)$ is convex in $(\alpha,\beta)\in[0,\infty)^{2}$,
$\overline{\Theta}(\alpha,\beta)$ is concave in $(\alpha,\beta)\in[0,\infty)^{2}$,
and $\Theta_{q}(\alpha)$ is concave in $\alpha\in[0,\infty)$ for
each $q<0$. All of $\underline{\Theta},\overline{\Theta},$ and $\Theta_{q}$
are nondecreasing. 
\end{lem}

Indeed, the operation of taking lower convex envelope and the operation
$\inf_{s\ge\alpha,t\ge\beta}$ (as well as the operation of taking
upper concave envelope and $\sup_{s\leq\alpha,t\leq\beta}$) can be
swapped, as shown in the following lemma. 
\begin{lem}
\label{lem:exchange}For a function $f:S\to\mathbb{R}$ with $S:=\prod_{i=1}^{n}[0,\infty)$,
the lower convex envelope of $\mathbf{y}\in S\mapsto\inf_{\mathbf{x}\in S_{\mathbf{y}}}f(\mathbf{x})$
with $S_{\mathbf{y}}=\prod_{i=1}^{n}[y_{i},\infty)$ is equal to $\mathbf{y}\in S\mapsto\inf_{\mathbf{x}\in S_{\mathbf{y}}}\breve{f}(\mathbf{x})$.
Moreover, the same is true if we reset $S_{\mathbf{y}}=\prod_{i=1}^{n}[0,y_{i}]$. 
\end{lem}

This lemma immediately implies that the upper concave envelope of
$\mathbf{y}\in S\mapsto\sup_{\mathbf{x}\in S_{\mathbf{y}}}f(\mathbf{x})$
is equal to $\mathbf{y}\in S\mapsto\sup_{\mathbf{x}\in S_{\mathbf{y}}}\invbreve f(\mathbf{x})$,
where $S_{\mathbf{y}}=\prod_{i=1}^{n}[0,y_{i}]$ or $S_{\mathbf{y}}=\prod_{i=1}^{n}[y_{i},\infty)$. 
\begin{proof}
This lemma follows directly by observing that the lower convex envelope
of $\mathbf{y}\in S\mapsto\inf_{\mathbf{x}\in S_{\mathbf{y}}}f(\mathbf{x})$
and the function $\mathbf{y}\in S\mapsto\inf_{\mathbf{x}\in S_{\mathbf{y}}}\breve{f}(\mathbf{x})$
are both equal to 
\[
\mathbf{y}\in S\mapsto\inf_{\{(q_{i},\mathbf{x}_{i}):i\in[n+2]\}:q_{i}\ge0,\sum_{i=1}^{n+2}q_{i}=1,\sum_{i=1}^{n+2}q_{i}\mathbf{x}_{i}\in S_{\mathbf{y}}}\sum_{i=1}^{n+2}q_{i}f(\mathbf{x}_{i}).
\]
Here Carathéodory's theorem is applied. 
\end{proof}
By the convexity (or concavity) and monotonicity, one can easily verify
that $\underline{\Theta}$ is continuous on $[0,\alpha_{\max})\times[0,\beta_{\max})$,
$\overline{\Theta}$ is continuous on $(0,\alpha_{\max}]\times(0,\beta_{\max}]$,
and $\Theta_{q}$ is continuous on $(0,\alpha_{\max}]$ for each given
$q<0$. In particular, for the finite alphabet case, by the compactness
of the probability simplex, $\underline{\Theta}$ and $\overline{\Theta}$
are continuous on $[0,\alpha_{\max}]\times[0,\beta_{\max}]$, and
$\Theta_{q}$ is continuous on $[0,\alpha_{\max}]$ for each given
$q<0$.  If we introduce a time-sharing random variable $W$ on $[3]$,
then we can rewrite 
\begin{align*}
\underline{\Theta}(\alpha,\beta) & =\inf_{\substack{Q_{W},Q_{X|W},Q_{Y|W}:\\
D(Q_{X|W}\|P_{X}|Q_{W})\ge\alpha,\\
D(Q_{Y|W}\|P_{Y}|Q_{W})\ge\beta
}
}\mathbb{D}(Q_{X|W},Q_{Y|W}\|P_{XY}|Q_{W})\\
\overline{\Theta}(\alpha,\beta) & =\sup_{\substack{Q_{W},Q_{X|W},Q_{Y|W}:\\
D(Q_{X|W}\|P_{X}|Q_{W})\leq\alpha,\\
D(Q_{Y|W}\|P_{Y}|Q_{W})\leq\beta
}
}\mathbb{D}(Q_{X|W},Q_{Y|W}\|P_{XY}|Q_{W})\\
\Theta_{q}(\alpha) & =\sup_{Q_{W},Q_{X|W}:D(Q_{X|W}\|P_{X}|Q_{W})\leq\alpha}\inf_{Q_{Y|W}}\mathbb{D}(Q_{X|W},Q_{Y|W}\|P_{XY}|Q_{W})-\frac{D(Q_{Y|W}\|P_{Y}|Q_{W})}{q}.
\end{align*}

We now discuss a property of $\underline{\Theta}(\alpha,\beta)$.
One can observe that $\underline{\Theta}(\alpha,\beta)\ge\alpha$,
since 
\begin{align*}
D(Q_{XY|W}\|P_{XY}|Q_{W}) & =D(Q_{X|W}\|P_{X}|Q_{W})+D(Q_{Y|XW}\|P_{Y|X}|Q_{XW})\\
 & \ge D(Q_{X|W}\|P_{X}|Q_{W})\ge\alpha.
\end{align*}
When does $\underline{\Theta}(\alpha,\beta)=\alpha$ hold? For the
finite alphabet case, the infimum in the definition of $\underline{\Theta}(\alpha,\beta)$
is attained. Suppose $(Q_{XY|W},Q_{W})$ is a pair of optimal distributions
such that $\underline{\Theta}(\alpha,\beta)=\alpha$. Then, $D(Q_{X|W}\|P_{X}|Q_{W})=\alpha,D(Q_{Y|XW}\|P_{Y|X}|Q_{XW})=0$.
Hence, $Q_{Y|XW}=P_{Y|X}$ and $D(Q_{X|W}\|P_{X}|Q_{W})=\alpha$.
However, $(Q_{XY|W},Q_{W})$ must satisfy $D(Q_{Y|W}\|P_{Y}|Q_{W})\ge\beta$,
i.e., $D(Q_{X|W}\circ P_{Y|X}\|P_{Y}|Q_{W})\ge\beta$. It means that
$\underline{\Theta}(\alpha,\beta)=\alpha$ if and only if $\beta\le\invbreve\eta(\alpha)$,
where for $s\ge0$, 
\[
\eta(s):=\sup_{Q_{X}:D(Q_{X}\|P_{X})=s}D(Q_{X}\circ P_{Y|X}\|P_{Y}).
\]
Hence, for sufficiently small $\beta$, $\underline{\Theta}(\alpha,\beta)=\alpha$,
i.e., $\underline{\Theta}$ is a part of the plane $(\alpha,\beta)\mapsto\alpha$
for $(\alpha,\beta)$ such that $\beta\le\invbreve\eta(\alpha)$.
The function $\eta$ is called Shannon concentration function, which
was characterized by Mrs Gerber's lemma for binary symmetric channels,
and by the entropy-power inequality for Gaussian channels; see more
discussions in Section \ref{subsec:R=0000E9nyi-Concentration-Functions}.

Define several conditions on $(p,q,p^{*},q^{*},\alpha,\beta,f,g)$
as follows. 
\begin{enumerate}
\item Condition 0+: $(\frac{1}{p^{*}},\frac{1}{q^{*}})$ is a subgradient
of $\underline{\Theta}$ at the point $(\alpha,\beta)$. 
\item Condition 0-: $(\frac{1}{p^{*}},\frac{1}{q^{*}})$ is a supergradient
of $\overline{\Theta}$ at the point $(\alpha,\beta)$. 
\item Condition 0-{}-: $\frac{1}{p^{*}}$ is a supergradient of $\overline{\Theta}_{q}$
at the point $\alpha$. 
\item Condition 1+: $\frac{1}{n}\Ent_{p,p^{*}}(f)$ is in $[\alpha,\infty)$
if $p\ge p^{*}$, and in $(-\infty,\alpha]$ if $0\le p\le p^{*}$.
Similarly, $\frac{1}{n}\Ent_{q,q^{*}}(g)$ is in $[\beta,\infty)$
if $q\ge q^{*}$, and in $(-\infty,\beta]$ if $0\le q\le q^{*}$. 
\item Condition 1-: $\frac{1}{n}\Ent_{p,p^{*}}(f)$ is in $(-\infty,\alpha]$
if $p\ge p^{*}$, and in $[\alpha,\infty)$ if $0\le p\le p^{*}$.
Similarly, $\frac{1}{n}\Ent_{q,q^{*}}(g)$ is in $(-\infty,\beta]$
if $q\ge q^{*}$, and in $[\beta,\infty)$ if $0\le q\le q^{*}$. 
\item Condition 1-{}-: $\frac{1}{n}\Ent_{p,p^{*}}(f)$ is in $(-\infty,\alpha]$
if $p\ge p^{*}$, and in $[\alpha,\infty)$ if $0\le p\le p^{*}$. 
\end{enumerate}
By Corollary \ref{cor:BL}, we obtain the following simple forms of
strong (forward and reverse) Brascamp--Lieb inequalities. We define
the \emph{effective region} of $\overline{\Theta}$ as the set of
$(\alpha,\beta)$ such that $\overline{\Theta}(\alpha,\beta)=\invbreve{\varphi}(\alpha,\beta)$. 
\begin{thm}[Strong Brascamp--Lieb Inequalities (Two-Function Version)]
\label{thm:strongBL} Let $\alpha\in[0,\alpha_{\max}],\beta\in[0,\beta_{\max}]$.
Then for $P_{XY}^{\otimes n}$, the following hold. 
\begin{enumerate}
\item Let $p^{*},q^{*}\in[0,\infty]$ satisfy Condition 0+. Let $p,q\in(0,\infty)$.
Then 
\begin{align}
\langle f,g\rangle & \leq e^{-n(\underline{\Theta}(\alpha,\beta)-\frac{\alpha}{p}-\frac{\beta}{q})}\Vert f\Vert_{p}\Vert g\Vert_{q}\label{eq:FBL-1}
\end{align}
for all functions $f:\mathcal{X}^{n}\to[0,\infty),g:\mathcal{Y}^{n}\to[0,\infty)$
satisfying Condition 1+. 
\item Let $p^{*},q^{*}\in[0,\infty]$ satisfy Condition 0-. Let $p,q\in(0,\infty)$.
Then 
\begin{align}
\langle f,g\rangle & \geq e^{-n(\overline{\Theta}(\alpha,\beta)-\frac{\alpha}{p}-\frac{\beta}{q})}\Vert f\Vert_{p}\Vert g\Vert_{q}\label{eq:RBL-1}
\end{align}
for all functions $f:\mathcal{X}^{n}\to[0,\infty),g:\mathcal{Y}^{n}\to[0,\infty)$
satisfying Condition 1-. 
\item Let $p^{*}\in[0,\infty],q\in(-\infty,0)$ satisfy Condition 0-{}-.
Let $p\in(0,\infty)$. Then 
\begin{align}
\langle f,g\rangle & \geq e^{-n(\Theta_{q}(\alpha)-\frac{\alpha}{p})}\Vert f\Vert_{p}\Vert g\Vert_{q}\label{eq:RBL-2}
\end{align}
for all functions $f:\mathcal{X}^{n}\to[0,\infty)$ satisfying Condition
1-{}- and all $g:\mathcal{Y}^{n}\to[0,\infty)$. 
\item For any Polish spaces $\mathcal{X},\mathcal{Y}$, the inequality in
Statement 1) is exponentially sharp as $n\to\infty$ for $\alpha\in[0,\alpha_{\max}],\beta\in[0,\beta_{\max}]$,
the inequality in Statement 2) is exponentially sharp as $n\to\infty$
for $(\alpha,\beta)$ in the effective region of $\overline{\Theta}$,
and the inequality in Statement 3) is exponentially sharp as $n\to\infty$
for $\alpha$ in the effective region of $\Theta_{q}$. 
\end{enumerate}
\end{thm}

\begin{proof}
We first prove Statements 1)-3). Observe that 
\begin{align*}
\underline{\Lambda}_{p,q,p^{*},q^{*}}^{*}(a,b) & =\inf_{s,t\ge0}\underline{\varphi}(s,t)+\eta_{p,p^{*}}(a,s)+\eta_{q,q^{*}}(b,t)\\
 & \geq\inf_{s,t\ge0}\underline{\Theta}(s,t)+\eta_{p,p^{*}}(a,s)+\eta_{q,q^{*}}(b,t).
\end{align*}
For nonnegative $(f,g)$ satisfying Condition 1+, denote $a=\Ent_{p,p^{*}}(f)$
and $b=\Ent_{q,q^{*}}(g)$. Then we have 
\begin{align*}
\underline{\Lambda}_{p,q,p^{*},q^{*}}^{*}(a,b)-\frac{a}{p}-\frac{b}{q} & \geq\inf_{s,t\ge0}\underline{\Theta}(s,t)+\frac{-s}{p}\lor((\frac{1}{p^{*}}-\frac{1}{p})a-\frac{s}{p^{*}})+\frac{-t}{q}\lor((\frac{1}{q^{*}}-\frac{1}{q})b-\frac{t}{q^{*}})\\
 & \geq\inf_{s,t\ge0}\underline{\Theta}(s,t)+\frac{-s}{p}\lor((\frac{1}{p^{*}}-\frac{1}{p})\alpha-\frac{s}{p^{*}})+\frac{-t}{q}\lor((\frac{1}{q^{*}}-\frac{1}{q})\beta-\frac{t}{q^{*}})\\
 & =\inf_{s,t\ge0}\underline{\Theta}(s,t)+\eta_{p,p^{*}}(\alpha,s)+\eta_{q,q^{*}}(\beta,t)-\frac{\alpha}{p}-\frac{\beta}{q}\\
 & \ge\inf_{s,t\ge0}\underline{\Theta}(s,t)+\frac{\alpha-s}{p^{*}}+\frac{\beta-t}{q^{*}}-\frac{\alpha}{p}-\frac{\beta}{q}\\
 & =\underline{\Theta}(\alpha,\beta)-\frac{\alpha}{p}-\frac{\beta}{q},
\end{align*}
where the last line follows by Condition 0+. Substituting this into
\eqref{eq:-6} yields \eqref{eq:FBL-1}.

We next prove \eqref{eq:RBL-1}. Since $p,q,p^{*},q^{*}\ge0$, we
have that 
\begin{align}
\overline{\Lambda}_{p,q,p^{*},q^{*}}^{*}(a,b) & =\sup_{Q_{W},\mathbf{Q}^{+}}\min\biggl\{\theta_{p,q}(Q_{W},Q_{X|W},Q_{Y|W}),\theta_{p,q^{*}}(Q_{W},Q_{X|W},\hat{Q}_{Y|W}),\nonumber \\
 & \qquad\theta_{p^{*},q}(Q_{W},\hat{Q}_{X|W},Q_{Y|W}),\theta_{p^{*},q^{*}}(Q_{W},\hat{Q}_{X|W},\hat{Q}_{Y|W})\biggr\}\nonumber \\
 & \leq\sup_{\substack{Q_{W},\hat{Q}_{X|W},\hat{Q}_{Y|W}:\\
\alpha_{p/p^{*}}^{-}\le D(\hat{Q}_{X|W}\|P_{X}|Q_{W})\le\alpha_{p/p^{*}}^{+},\\
\beta_{q/q^{*}}^{-}\le D(\hat{Q}_{Y|W}\|P_{Y}|Q_{W})\leq\beta_{q/q^{*}}^{+}
}
}\theta_{p^{*},q^{*}}(Q_{W},\hat{Q}_{X|W},\hat{Q}_{Y|W})\nonumber \\
 & \leq\sup_{\alpha_{p/p^{*}}^{-}\le\hat{s}\le\alpha_{p/p^{*}}^{+},\beta_{q/q^{*}}^{-}\le\hat{t}\leq\beta_{q/q^{*}}^{+}}\overline{\Theta}(\hat{s},\hat{t})+\frac{a-\hat{s}}{p^{*}}+\frac{b-\hat{t}}{q^{*}}.\label{eq:-44}
\end{align}

We first assume $p\ge p^{*},q\ge q^{*}$ (and hence by assumption,
$0\le a\le\alpha,0\le b\le\beta$). Then, we further have 
\begin{align*}
\overline{\Lambda}_{p,q,p^{*},q^{*}}^{*}(a,b) & \le\sup_{\hat{s}\le a,\hat{t}\le b}\overline{\Theta}(\hat{s},\hat{t})+\frac{a-\hat{s}}{p^{*}}+\frac{b-\hat{t}}{q^{*}}.
\end{align*}
Hence, 
\begin{align}
\overline{\Lambda}_{p,q,p^{*},q^{*}}^{*}(a,b)-\frac{a}{p}-\frac{b}{q} & \le[\sup_{\hat{s}\le a,\hat{t}\le b}\overline{\Theta}(\hat{s},\hat{t})-\frac{\hat{s}}{p^{*}}-\frac{\hat{t}}{q^{*}}]+[(\frac{1}{p^{*}}-\frac{1}{p})a+(\frac{1}{q^{*}}-\frac{1}{q})b]\label{eq:-45}\\
 & \le\overline{\Theta}(\alpha,\beta)-\frac{\alpha}{p}-\frac{\beta}{q},
\end{align}
where the last inequality above follows by the fact that the supremum
term in \eqref{eq:-45} is upper bounded by $\sup_{\hat{s}\le\alpha,\hat{t}\le\beta}\overline{\Theta}(\hat{s},\hat{t})-\frac{\hat{s}}{p^{*}}-\frac{\hat{t}}{q^{*}}=\overline{\Theta}(\alpha,\beta)-\frac{\alpha}{p^{*}}-\frac{\beta}{q^{*}}$
and the remaining terms are nondecreasing in $a,b$. Substituting
\eqref{eq:-45} into \eqref{eq:-7} yields \eqref{eq:RBL-1}. The
case of $p\ge p^{*},q\le q^{*}$, the case of $p\le p^{*},q\ge q^{*}$,
and the case of $p\le p^{*},q\le q^{*}$ can be proven similarly.

We lastly prove \eqref{eq:RBL-2}. For $p\ge p^{*}$ (and hence by
assumption, $a\le\alpha$), we have 
\begin{align*}
\overline{\Lambda}_{p,q,p^{*},q^{*}}^{*}(a,b) & \le\sup_{\hat{s}\le a}\sup_{Q_{W},Q_{X|W}:D(Q_{X|W}\|P_{X}|Q_{W})=\hat{s}}\inf_{Q_{Y|W}}\mathbb{D}(Q_{X|W},Q_{Y|W}\|P_{XY}|Q_{W})\\
 & \qquad+\frac{b-D(Q_{Y|W}\|P_{Y}|Q_{W})}{q}+\frac{a-\hat{s}}{p^{*}}\\
 & \le\sup_{\hat{s}\le a}\Theta_{q}(\hat{s})+\frac{a-\hat{s}}{p^{*}}+\frac{b}{q}.
\end{align*}
Hence, 
\begin{align}
\overline{\Lambda}_{p,q,p^{*},q^{*}}^{*}(a,b)-\frac{a}{p}-\frac{b}{q} & \le\sup_{\hat{s}\le a}\Theta_{q}(\hat{s})-\frac{\hat{s}}{p^{*}}+(\frac{1}{p^{*}}-\frac{1}{p})a\le\Theta_{q}(\alpha)-\frac{\alpha}{p}.\label{eq:-45-1}
\end{align}
Substituting this into \eqref{eq:-7} yields \eqref{eq:RBL-2}. The
case $p\le p^{*}$ can be proven similarly.
\end{proof}
We know that for $p,p^{*}\ge0$, $\Ent_{p,p^{*}}(f)\ge\Ent_{0}(f)=-\log P_{X}^{\otimes n}(f>0)$.
Hence $P_{X}^{\otimes n}(f>0)\leq e^{-n\alpha}$ implies $\Ent_{p,p^{*}}(f)\ge\alpha$,
which immediately yields the following corollary. 
\begin{cor}
Let $\alpha\in[0,\alpha_{\max}],\beta\in[0,\beta_{\max}]$. Let $f:\mathcal{X}^{n}\to[0,\infty),g:\mathcal{Y}^{n}\to[0,\infty)$
be nonnegative functions such that $P_{X}^{\otimes n}(f>0)\leq e^{-n\alpha},P_{Y}^{\otimes n}(g>0)\leq e^{-n\beta}$.
Then the following hold. 
\begin{enumerate}
\item Let $p^{*},q^{*}\in[0,\infty]$ satisfy Condition 0+. Then for any
$p,q\in(0,\infty)$ such that $p\ge p^{*},q\ge q^{*}$, 
\begin{align}
\langle f,g\rangle & \leq e^{-n(\underline{\Theta}(\alpha,\beta)-\frac{\alpha}{p}-\frac{\beta}{q})}\Vert f\Vert_{p}\Vert g\Vert_{q}.\label{eq:FBL-2}
\end{align}
\item Let $p^{*},q^{*}\in[0,\infty]$ satisfy Condition 0-. Then for any
$p,q\in(0,\infty)$ such that $0<p\le p^{*},0<q\le q^{*}$, 
\begin{align}
\langle f,g\rangle & \geq e^{-n(\overline{\Theta}(\alpha,\beta)-\frac{\alpha}{p}-\frac{\beta}{q})}\Vert f\Vert_{p}\Vert g\Vert_{q}.\label{eq:RBL-3}
\end{align}
\item Let $p^{*}\in[0,\infty],q\in(-\infty,0)$ satisfy Condition 0-{}-.
Then for any $p\in(0,\infty)$ such that $0<p\le p^{*}$, 
\begin{align}
\langle f,g\rangle & \geq e^{-n(\Theta_{q}(\alpha)-\frac{\alpha}{p})}\Vert f\Vert_{p}\Vert g\Vert_{q}.\label{eq:RBL-4}
\end{align}
\end{enumerate}
\end{cor}

\subsection{\label{subsec:Strong-Hypercontractivity-Inequa}Strong Hypercontractivity
Inequalities}

We next derive a strong version of hypercontractivity inequalities,
i.e., a special class of BL inequalities with the factors set to $1$
(or the BL exponents set to $0$). To this end, it suffices to find
conditions under which $\underline{\Theta}(\alpha,\beta)-\frac{\alpha}{p}-\frac{\beta}{q}\ge0$
holds for \eqref{eq:FBL-1}, $\overline{\Theta}(\alpha,\beta)-\frac{\alpha}{p}-\frac{\beta}{q}\le0$
holds for \eqref{eq:RBL-1}, and $\Theta_{q}(\alpha)-\frac{\alpha}{p}\le0$
holds for \eqref{eq:RBL-2}.

For $\alpha\in[0,\alpha_{\max}],\beta\in[0,\beta_{\max}]$, we define
the (forward and reverse) hypercontractivity regions as 
\begin{align}
\mathcal{R}_{\mathtt{FH}}^{(\alpha,\beta)}(P_{XY}) & :=\{(p,q)\in(0,\infty)^{2}:\underline{\Theta}(s,t)\ge\frac{s}{p}+\frac{t}{q},\forall s\ge\alpha,t\ge\beta\}\label{eq:-67}\\
\mathcal{R}_{\mathtt{RH}}^{(\alpha,\beta)}(P_{XY}) & :=\{(p,q)\in(0,\infty)^{2}:\overline{\Theta}(s,t)\le\frac{s}{p}+\frac{t}{q},\forall s\ge\alpha,t\ge\beta\}\label{eq:-68}\\
\mathcal{R}_{\mathtt{RH},q}^{(\alpha)}(P_{XY}) & :=\{p\in(0,\infty):\Theta_{q}(s)\le\frac{s}{p},\forall s\ge\alpha\}.\label{eq:RH}
\end{align}
When specialized to the case of $\alpha=\beta=0$, $\mathcal{R}_{\mathtt{FH}}^{(\alpha,\beta)}(P_{XY})$
and $\mathcal{R}_{\mathtt{RH}}^{(\alpha,\beta)}(P_{XY})$ reduce to
the usual hypercontractivity regions\footnote{Rigorously speaking, only hypercontractivity ribbons, rather than
hypercontractivity regions, were defined in \cite{anantharam2014hypercontractivity,kamath2015reverse}.
However, they are defined similarly, but with taking the Hölder conjugate
of $p$ or $q$ and excluding the Hölder region $\{(p,q)\in[1,\infty)^{2}:q\ge p'\}$
or $\{(p,q)\in(-\infty,1]^{2}:q\le p'\}$.} \cite{anantharam2014hypercontractivity,nair2014equivalent,kamath2015reverse,liu2018information}.

Following similar steps to the proof of Theorem \ref{thm:strongBL},
it is easy to obtain the following new version of hypercontractivity
from Theorem \ref{thm:BLexponent}. 
\begin{cor}[Strong Hypercontractivity Inequalities]
\label{thm:stronghypercontractivity-1} Let $\alpha\in[0,\alpha_{\max}],\beta\in[0,\beta_{\max}]$.
Then for $P_{XY}^{\otimes n}$, the following hold. 
\begin{enumerate}
\item Let $p^{*},q^{*}\in[0,\infty]$ satisfy Condition 0+. Let $(p,q)\in\mathcal{R}_{\mathtt{FH}}^{(\alpha,\beta)}(P_{XY})$.
Then 
\begin{align}
\langle f,g\rangle & \leq\Vert f\Vert_{p}\Vert g\Vert_{q}\label{eq:fh-1}
\end{align}
for all $f:\mathcal{X}^{n}\to[0,\infty),g:\mathcal{Y}^{n}\to[0,\infty)$
satisfying $\frac{1}{n}\Ent_{p,p^{*}}(f)\in[\alpha,\infty)$ and $\frac{1}{n}\Ent_{q,q^{*}}(g)\in[\beta,\infty)$. 
\item Let $p^{*},q^{*}\in[0,\infty]$ satisfy Condition 0-. Let $(p,q)\in\mathcal{R}_{\mathtt{RH}}^{(\alpha,\beta)}(P_{XY})$.
Then 
\begin{align}
\langle f,g\rangle & \geq\Vert f\Vert_{p}\Vert g\Vert_{q}\label{eq:rh-1}
\end{align}
for all $f:\mathcal{X}^{n}\to[0,\infty),g:\mathcal{Y}^{n}\to[0,\infty)$
satisfying $\frac{1}{n}\Ent_{p,p^{*}}(f)\in[\alpha,\infty)$ and $\frac{1}{n}\Ent_{q,q^{*}}(g)\in[\beta,\infty)$. 
\item Let $p^{*}\in[0,\infty],q\in(-\infty,0)$ satisfy Condition 0-{}-.
Let $p\in\mathcal{R}_{\mathtt{RH},q}^{(\alpha)}(P_{XY})$. Then \eqref{eq:rh-1}
still holds for all $f:\mathcal{X}^{n}\to[0,\infty)$ satisfying $\frac{1}{n}\Ent_{p,p^{*}}(f)\in[\alpha,\infty)$
and all $g:\mathcal{Y}^{n}\to[0,\infty)$. 
\end{enumerate}
\end{cor}

\begin{rem}
For any Polish spaces $\mathcal{X},\mathcal{Y}$, given $(\alpha,\beta)\in[0,\alpha_{\max}]\times[0,\beta_{\max}]$,
the inequality \eqref{eq:fh-1} is exponentially sharp, in the sense
that for any $(p,q)\in(0,\infty)^{2}\backslash\mathrm{cl}\mathcal{R}_{\mathtt{FH}}^{(\alpha,\beta)}(P_{XY})$,
there exist a positive integer $n$ and a pair of functions $(f,g)$
on $\mathcal{X}^{n},\mathcal{Y}^{n}$ respectively that satisfy the
assumption in Theorem \ref{thm:BLexponent} but violates \eqref{eq:fh-1}.
Here $\mathrm{cl}\mathcal{A}$ denotes the closure of a set $\mathcal{A}$.
For any Polish spaces $\mathcal{X},\mathcal{Y}$, given $(\alpha,\beta)$
in the effective region of $\overline{\Theta}$ and under the assumptions
in Statement 2, the inequality \eqref{eq:rh-1} is exponentially sharp
in a similar sense. Given $\alpha$ in the effective region of $\Theta_{q}$
and under the assumptions in Statement 3, the inequality \eqref{eq:rh-1}
 is also exponentially sharp in a similar sense. 
\end{rem}

Conceptually, the hypercontractivity regions in \eqref{eq:-67}-\eqref{eq:RH}
can be seen as a set of $(p,q)$ on which the hypercontractivity inequalities
(or the reverse versions) hold in the large-deviation regime; while
the usual hypercontractivity regions correspond to the moderate-deviation
regime. In large deviation theory, it is well known that the moderate-deviation
rate function can be recovered from the large-deviation rate function,
by letting the threshold parameter go to zero in a certain speed \cite{Dembo}.
Our hypercontractivity inequalities here are stronger than the usual
ones in a similar sense, if $P_{XY}$ satisfies that $\underline{\Theta}(s,t)\ge\frac{s}{p}+\frac{t}{q}$,
$\overline{\Theta}(s,t)\le\frac{s}{p}+\frac{t}{q}$, and $\overline{\Theta}_{q}(s)\le\frac{s}{p}$
hold for all $s,t\ge0$ if and only if they hold for a neighborhood
of the origin.

\section{\label{sec:Strong-BL-and-1}Strong BL and HC Inequalities: Single-Function
Version}

We next consider the single-function version of BL exponents. Given
$p,q,\hat{p}\in\overline{\mathbb{R}}$ and $P_{XY}$, we define the
single-function version of \emph{optimal forward and reverse BL exponents
}as 
\begin{align}
\underline{\Upsilon}_{p,q,\hat{p}}(\alpha|P_{XY}) & :=-\sup_{f\in\mathcal{F}_{\alpha}}\log\frac{\Vert P_{X|Y}(f)\Vert_{q}}{\Vert f\Vert_{\hat{p}}}\label{eq:-30}\\
\overline{\Upsilon}_{p,q,\hat{p}}(\alpha|P_{XY}) & :=-\inf_{f\in\mathcal{F}_{\alpha}}\log\frac{\Vert P_{X|Y}(f)\Vert_{q}}{\Vert f\Vert_{\hat{p}}}.\label{eq:-31}
\end{align}
Similar to Lemma \ref{lem:symmetry}, we also have the following lemma. 
\begin{lem}
\label{lem:symmetry-1} For $p,q,\hat{p}\in\overline{\mathbb{R}}\backslash\{0\}$,
both $\underline{\Gamma}_{p,q,\hat{p}}(\alpha|P_{XY}):=\underline{\Upsilon}_{p,q,\hat{p}}(\alpha|P_{XY})+\frac{\alpha}{\hat{p}}$
and $\overline{\Gamma}_{p,q,\hat{p}}(\alpha|P_{XY}):=\overline{\Upsilon}_{p,q,\hat{p}}(\alpha|P_{XY})+\frac{\alpha}{\hat{p}}$
are symmetric w.r.t. $(p,\hat{p})$. 
\end{lem}

We now extend the definitions in \eqref{eq:-30} and \eqref{eq:-31}
to the case of $(p,\hat{p})\in\overline{\mathbb{R}}^{2}\backslash\{(0,0)\}$
by the symmetric extension. Such an extension ensures that $\underline{\Gamma}_{p,q,\hat{p}}$
and $\overline{\Gamma}_{p,q,\hat{p}}$ are still symmetric w.r.t.
$(p,\hat{p})$.  For the product distribution $P_{XY}^{\otimes n}$,
we define the\emph{ forward and reverse $n$-BL exponents} respectively
as $\underline{\Gamma}_{p,q,\hat{p}}^{(n)}(\alpha):=\frac{1}{n}\underline{\Gamma}_{p,q,\hat{p}}(n\alpha|P_{XY}^{\otimes n})$
and $\overline{\Gamma}_{p,q,\hat{p}}^{(n)}(\alpha):=\frac{1}{n}\overline{\Gamma}_{p,q,\hat{p}}(n\alpha|P_{XY}^{\otimes n})$,
as well as, the \emph{forward and reverse asymptotic BL exponents
}respectively as $\underline{\Gamma}_{p,q,\hat{p}}^{(\infty)}(\alpha):=\lim_{n\to\infty}\underline{\Gamma}_{p,q,\hat{p}}^{(n)}(\alpha)$
and $\overline{\Gamma}_{p,q,\hat{p}}^{(\infty)}(\alpha):=\lim_{n\to\infty}\overline{\Gamma}_{p,q,\hat{p}}^{(n)}(\alpha)$.

For $q\in[1,\infty)$, define 
\begin{equation}
\underline{\Gamma}_{p,q,\hat{p}}^{*}(\alpha):=\inf_{Q_{W},Q_{X|W},Q_{Y|W}}\theta_{p,q',\hat{p}}(Q_{W},Q_{X|W},Q_{Y|W}),\label{eq:GammaForward}
\end{equation}
where $q':=\frac{q}{q-1}$ is the Hölder conjugate of $q$, and 
\begin{align*}
\theta_{p,q',\hat{p}}(Q_{W},Q_{X|W},Q_{Y|W}) & :=\mathbb{D}(Q_{X|W},Q_{Y|W}\|P_{XY}|Q_{W})-\frac{D(Q_{Y|W}\|P_{Y}|Q_{W})}{q'}+\eta_{p,\hat{p}}(\alpha,D(Q_{X|W}\|P_{X}|Q_{W})).
\end{align*}
By Carathéodory's theorem, without loss of optimality, it suffices
to restrict $|\mathcal{W}|\le3$.

According to the signs of $p,\hat{p}$, we partition the distributions
$\{Q_{X|W},\hat{Q}_{X|W}\}$ into two subsets $\mathbf{S}^{+},\mathbf{S}^{-}$.
Specifically, $Q_{X|W}\in\mathbf{S}^{+}$ if $p\ge0$; $Q_{X|W}\in\mathbf{S}^{-}$
otherwise. Similarly, $\hat{Q}_{X|W}$ are assigned into $\mathbf{S}^{+},\mathbf{S}^{-}$
respectively according to the sign of $\hat{p}$. For $q\in(-\infty,1)\backslash\{0\}$,
define 
\begin{equation}
\overline{\Gamma}_{p,q,\hat{p}}^{*}(\alpha):=\begin{cases}
\sup_{Q_{W},\mathbf{S}^{+}}\inf_{Q_{Y|W},\mathbf{S}^{-}}\min\{\theta_{p,q'}(Q_{W},Q_{X|W},Q_{Y|W}),\theta_{\hat{p},q'}(Q_{W},\hat{Q}_{X|W},Q_{Y|W})\} & 0<q<1\\
\sup_{Q_{WY},\mathbf{S}^{+}}\inf_{\mathbf{S}^{-}}\min\{\theta_{p,q'}(Q_{W},Q_{X|W},Q_{Y|W}),\theta_{\hat{p},q'}(Q_{W},\hat{Q}_{X|W},Q_{Y|W})\} & q<0
\end{cases}\label{eq:GammaReverse}
\end{equation}
where 
\begin{align*}
\theta_{p,q'}(Q_{W},Q_{X|W},Q_{Y|W}) & :=\mathbb{D}(Q_{X|W},Q_{Y|W}\|P_{XY}|Q_{W})-\frac{D(Q_{Y|W}\|P_{Y}|Q_{W})}{q'}+\frac{\alpha-D(Q_{X|W}\|P_{X}|Q_{W})}{p}.
\end{align*}
and the infimization $\inf_{\mathbf{S}^{-}}$ is taken over all the
distributions in $\mathbf{S}^{-}$, and the supremization $\sup_{\mathbf{S}^{+}}$
is taken over all the distributions in $\mathbf{S}^{+}$ under the
constraints $\alpha_{\hat{p}/p}^{-}\le D(Q_{X|W}\|P_{X}|Q_{W})\le\alpha_{\hat{p}/p}^{+}$
and $\alpha_{p/\hat{p}}^{-}\le D(\hat{Q}_{X|W}\|P_{X}|Q_{W})\le\alpha_{p/\hat{p}}^{+}$
if $p,\hat{p}\ge0$. By Carathéodory's theorem, without loss of optimality,
it suffices to restrict $|\mathcal{W}|\le5$.

We obtain the following single-function version of strong BL inequalities. 
\begin{thm}[Brascamp--Lieb Exponents (Single-Function Version)]
\label{thm:BLExponentSingleFunc} Let $n\ge1$ and $\alpha\in\mathcal{E}_{p,\hat{p}}^{(n)}(\mathcal{X})$.
Let $p\in\mathbb{R}\backslash\{0\},\hat{p}\in\overline{\mathbb{R}}$.
If $q\in[1,\infty)$, then we have 
\begin{align}
\underline{\Gamma}_{p,q,\hat{p}}^{(n)}(\alpha) & \geq\underline{\Gamma}_{p,q,\hat{p}}^{*}(\alpha),\label{eq:FBLE_single}
\end{align}
where $\mathcal{E}_{p,\hat{p}}(\mathcal{X})$ was defined in \eqref{eq:E}.
If $q\in(-\infty,1)\backslash\{0\}$, then we have 
\begin{align}
\overline{\Gamma}_{p,q,\hat{p}}^{(n)}(\alpha) & \leq\overline{\Gamma}_{p,q,\hat{p}}^{*}(\alpha).\label{eq:RBLE_single}
\end{align}
Moreover, for finite $\mathcal{X},\mathcal{Y}$, these two inequalities
are asymptotically tight as $n\to\infty$ for given $(p,q,\hat{p},\alpha)$
as described above but with $\hat{p}\neq0$ and $\alpha\neq0,\alpha_{\max}$. 
\end{thm}

\begin{proof}
By the forward and reverse versions of the Hölder inequality, for
$\hat{g}:\mathcal{Y}^{n}\to[0,\infty),$ 
\begin{equation}
\Vert\hat{g}\Vert_{q}=\begin{cases}
\sup_{g\ge0}\frac{\langle\hat{g},g\rangle}{\Vert g\Vert_{q'}} & q\ge1\\
\inf_{g\ge0}\frac{\langle\hat{g},g\rangle}{\Vert g\Vert_{q'}} & q<1
\end{cases}\label{eq:equivalence}
\end{equation}
where the supremization and infimization are taken over all $g:\mathcal{Y}^{n}\to[0,\infty)$.
Setting $\hat{g}\leftarrow P_{X|Y}^{\otimes n}(f)$, we obtain the
following equivalences: 
\begin{align}
\sup_{f\in\mathcal{F}_{\alpha}}\frac{\Vert P_{X|Y}^{\otimes n}(f)\Vert_{q}}{\Vert f\Vert_{\hat{p}}} & =\sup_{f\in\mathcal{F}_{\alpha}}\sup_{g}\frac{\langle f,g\rangle}{\Vert f\Vert_{\hat{p}}\Vert g\Vert_{q'}},\quad\mbox{for }q\ge1,\label{eq:-38}\\
\inf_{f\in\mathcal{F}_{\alpha}}\frac{\Vert P_{X|Y}^{\otimes n}(f)\Vert_{q}}{\Vert f\Vert_{\hat{p}}} & =\inf_{f\in\mathcal{F}_{\alpha}}\inf_{g}\frac{\langle f,g\rangle}{\Vert f\Vert_{\hat{p}}\Vert g\Vert_{q'}},\quad\mbox{for }q<1.\label{eq:-39}
\end{align}
Hence, combined with these equivalences, Theorem \ref{thm:BLexponent}
implies \eqref{eq:FBLE_single} and \eqref{eq:RBLE_single}. Note
that here we substitute $q\leftarrow q'$ in Theorem \ref{thm:BLexponent},
and to ensure that $\Ent_{q',\hat{q}}(g)$ is finite (which is true
if $\|g\|_{q'},\|g\|_{\hat{q}}$ are finite), we set $\hat{q}$ to
a value in $(0,q')$ for the case of $q'>0$ and $\hat{q}$ to a value
in $(q',0)$ for the case of $q'<0$ in Theorem \ref{thm:BLexponent}. 
\end{proof}
The theorem above is equivalent to the following single-function version
of strong BL inequalities. 
\begin{cor}[Strong Brascamp--Lieb Inequalities (Single-Function Version)]
\label{cor:BL_single} Let $p\in\mathbb{R}\backslash\{0\},\hat{p}\in\overline{\mathbb{R}}$.
For any $n\ge1$, 
\begin{align}
\Vert P_{X|Y}^{\otimes n}(f)\Vert_{q} & \leq e^{-n(\underline{\Gamma}_{p,q,\hat{p}}^{*}(\alpha)-\frac{\alpha}{p})}\Vert f\Vert_{p}\textrm{ for }q\in[1,\infty),\label{eq:-9}\\
\Vert P_{X|Y}^{\otimes n}(f)\Vert_{q} & \geq e^{-n(\overline{\Gamma}_{p,q,\hat{p}}^{*}(\alpha)-\frac{\alpha}{p})}\Vert f\Vert_{p}\textrm{ for }q\in(-\infty,1)\backslash\{0\},\label{eq:-10}
\end{align}
for all $f:\mathcal{X}^{n}\to[0,\infty)$ such that $\alpha=\frac{1}{n}\Ent_{p,\hat{p}}(f)\in\mathbb{R}$. 
\end{cor}

\begin{rem}
By symmetry, the inequalities \eqref{eq:-9} and \eqref{eq:-10}  still
hold if $p,\hat{p}$ are swapped. 
\end{rem}

Recall that $\Theta_{q}(\alpha)$ is defined in \eqref{eq:ThetaOverline_q}
for $q<0$. Here we extend it to the case $r>0$ by defining 
\begin{equation}
\varphi_{r}(s):=\begin{cases}
\inf_{t\ge0}\underline{\varphi}(s,t)-\frac{t}{r} & r\ge1\\
\sup_{t\ge0}\overline{\varphi}(s,t)-\frac{t}{r} & 0<r<1
\end{cases}\label{eq:phi_r}
\end{equation}
and 
\begin{equation}
\Theta_{r}(\alpha):=\begin{cases}
\inf_{s\geq\alpha}\breve{\varphi}_{r}(s) & r\ge1\\
\sup_{s\le\alpha}\invbreve\varphi_{r}(s) & 0<r<1
\end{cases}.\label{eq:Theta_q}
\end{equation}
It is easy to verify that $\Theta_{r}(\alpha)\ge0$ for all $r\neq0$.
Define several conditions, which are similar to those in the two-function
case: 
\begin{enumerate}
\item Condition 0+: $\frac{1}{p^{*}}$ is a subgradient of $\Theta_{q'}$
at the point $\alpha$. 
\item Condition 0-: $\frac{1}{p^{*}}$ is a supergradient of $\Theta_{q'}$
at the point $\alpha$. 
\item Condition 1+: $\frac{1}{n}\Ent_{p,p^{*}}(f)\in[\alpha,\infty)$ if
$p\ge p^{*}$; $\frac{1}{n}\Ent_{p,p^{*}}(f)\in(-\infty,\alpha]$
if $0\le p\le p^{*}$. 
\item Condition 1-: $\frac{1}{n}\Ent_{p,p^{*}}(f)\in(-\infty,\alpha]$ if
$p\ge p^{*}$; $\frac{1}{n}\Ent_{p,p^{*}}(f)\in[\alpha,\infty)$ if
$0\le p\le p^{*}$. 
\end{enumerate}
By Corollary \ref{cor:BL_single}, we obtain the following simpler
forms of strong (forward and reverse) BL inequalities. 
\begin{thm}[Strong Brascamp--Lieb Inequalities (Single-Function Version)]
\label{thm:strongBL_single} Let $\alpha\in[0,\alpha_{\max}]$. Then
the following hold. 
\begin{enumerate}
\item Let $p^{*}\in[0,\infty],q\in[1,\infty)$ satisfy Condition 0+. Let
$p\in(0,\infty)$. Then 
\begin{align}
\Vert P_{X|Y}^{\otimes n}(f)\Vert_{q} & \leq e^{-n(\Theta_{q'}(\alpha)-\frac{\alpha}{p})}\Vert f\Vert_{p}\label{eq:FBL-3-2}
\end{align}
for all $f:\mathcal{X}^{n}\to[0,\infty)$ satisfying Condition 1+. 
\item Let $p^{*}\in[0,\infty],q\in(-\infty,1)\backslash\{0\}$ satisfy Condition
0-. Let $p\in(0,\infty)$. Then 
\begin{align}
\Vert P_{X|Y}^{\otimes n}(f)\Vert_{q} & \geq e^{-n(\Theta_{q'}(\alpha)-\frac{\alpha}{p})}\Vert f\Vert_{p}\label{eq:RBL-3-3}
\end{align}
for all $f:\mathcal{X}^{n}\to[0,\infty)$ satisfying Condition 1-. 
\item For arbitrary Polish spaces $\mathcal{X},\mathcal{Y}$, \eqref{eq:FBL-3-2}
and \eqref{eq:RBL-3-3} are exponentially sharp as $n\to\infty$ for
given $(p,q,p^{*},\alpha)$ with $\alpha\neq0,\alpha_{\max}$. 
\end{enumerate}
\end{thm}

For $\alpha\ge0$, define 
\begin{align*}
\mathcal{R}_{\mathtt{FH},q'}^{(\alpha)}(P_{XY}) & :=\{p\in(0,\infty):\Theta_{q'}(s)\geq\frac{s}{p},\forall s\in[\alpha,\alpha_{\max}]\}\\
\mathcal{R}_{\mathtt{RH},q'}^{(\alpha)}(P_{XY}) & :=\{p\in(0,\infty):\Theta_{q'}(s)\le\frac{s}{p},\forall s\in[\alpha,\alpha_{\max}]\}.
\end{align*}
Note that here $\mathcal{R}_{\mathtt{RH},q'}^{(\alpha)}(P_{XY})$
is consistent with the one defined in \eqref{eq:RH}. By Theorem \ref{thm:strongBL_single},
it is easy to obtain the following single function version of strong
hypercontractivity. 
\begin{cor}[Strong Hypercontractivity Inequalities]
\label{thm:stronghypercontractivity-1-1} Let $\alpha\in[0,\alpha_{\max}]$.
Then the following hold. 
\begin{enumerate}
\item Let $p^{*}\in[0,\infty],q\in[1,\infty)$ satisfy Condition 0+. Let
$p\in\mathcal{R}_{\mathtt{FH},q'}^{(\alpha)}(P_{XY})$. Then 
\begin{align}
\Vert P_{X|Y}^{\otimes n}(f)\Vert_{q} & \leq\Vert f\Vert_{p}\label{eq:sFBL}
\end{align}
for all $f:\mathcal{X}^{n}\to[0,\infty)$ satisfying Condition 1+. 
\item Let $p^{*}\in[0,\infty],q\in(-\infty,1)\backslash\{0\}$ satisfy Condition
0-. Let $p\in\mathcal{R}_{\mathtt{RH},q'}^{(\alpha)}(P_{XY})$. Then
\begin{align}
\Vert P_{X|Y}^{\otimes n}(f)\Vert_{q} & \geq\Vert f\Vert_{p}\label{eq:sRBL}
\end{align}
for all $f:\mathcal{X}^{n}\to[0,\infty)$ satisfying Condition 1-. 
\end{enumerate}
\end{cor}

\begin{rem}
Given $\alpha>0$, the inequality \eqref{eq:sFBL} is exponentially
sharp, in the sense that for any $p\in(0,\infty)\backslash\mathrm{cl}\mathcal{R}_{\mathtt{FH},q'}^{(\alpha)}(P_{XY})$,
there exists a positive integer $n$ and a nonnegative function $f$
on $\mathcal{X}^{n}$ that satisfy the assumption in Theorem \ref{thm:BLexponent}
but violates \eqref{eq:sFBL}. Given $\alpha>0$, the inequality \eqref{eq:sRBL}
is exponentially sharp in a similar sense. 
\end{rem}

\section{\label{sec:Applications}Applications}

\subsection{\label{subsec:R=0000E9nyi-Concentration-Functions}Rényi Concentration
Functions}

Define the \emph{(forward and reverse) $(p,q)$-Rényi concentration
functions }of $(P_{X},P_{Y|X})$ as 
\begin{align*}
\overline{\eta}_{p\to q}(\alpha|P_{X},P_{Y|X}) & :=\sup_{Q_{X}:D_{p}(Q_{X}\|P_{X})=\alpha}D_{q}(Q_{Y}\|P_{Y})\\
\underline{\eta}_{p\to q}(\alpha|P_{X},P_{Y|X}) & :=\inf_{Q_{X}:D_{p}(Q_{X}\|P_{X})=\alpha}D_{q}(Q_{Y}\|P_{Y})
\end{align*}
where $P_{Y}:=P_{X}\circ P_{Y|X}$ and $Q_{Y}:=Q_{X}\circ P_{Y|X}$.
Note that $\lim_{\alpha\downarrow0}\frac{1}{\alpha}\overline{\eta}_{p\to q}(\alpha|P_{X},P_{Y|X})$
corresponds to the Rényi hypercontractivity constants defined by Raginsky
\cite{raginsky2013logarithmic}. As a special case, the $(1,1)$-Rényi
concentration functions (which we term the \emph{Shannon concentration
functions})\emph{ }of $(P_{X},P_{Y|X})$ is 
\begin{align*}
\overline{\eta}_{1\to1}(\alpha|P_{X},P_{Y|X}) & =\sup_{Q_{X}:D(Q_{X}\|P_{X})=\alpha}D(Q_{Y}\|P_{Y})\\
\underline{\eta}_{1\to1}(\alpha|P_{X},P_{Y|X}) & =\inf_{Q_{X}:D(Q_{X}\|P_{X})=\alpha}D(Q_{Y}\|P_{Y}).
\end{align*}
Define the $n$-dimensional versions: 
\begin{align*}
\overline{\eta}_{p\to q}^{(n)}(\alpha|P_{X},P_{Y|X}) & :=\frac{1}{n}\overline{\eta}_{p\to q}(n\alpha|P_{X}^{\otimes n},P_{X}^{\otimes n})\\
\underline{\eta}_{p\to q}^{(n)}(\alpha|P_{X},P_{Y|X}) & :=\frac{1}{n}\underline{\eta}_{p\to q}(n\alpha|P_{X}^{\otimes n},P_{X}^{\otimes n}),
\end{align*}
and their limits as $\overline{\eta}_{p\to q}^{(\infty)}$ and $\underline{\eta}_{p\to q}^{(\infty)}$
as $n\to\infty$. Since $P_{X},P_{Y|X}$ are fixed, in the following,
we sometimes omit ``$|P_{X},P_{Y|X}$'' in the notations above.

By using the data-processing inequality, it is not difficult to see
that \cite{polyanskiy2019improved} for $\alpha\in\mathcal{E}_{p,1}(\mathcal{X})$,
\begin{align*}
\overline{\eta}_{1\to1}^{(n)}(\alpha) & \leq\overline{\eta}_{1\to1}^{*}(\alpha):=\sup_{Q_{XW}:D(Q_{X|W}\|P_{X}|Q_{W})=\alpha}D(Q_{Y|W}\|P_{Y}|Q_{W})
\end{align*}
where the most RHS is equal to the upper concave envelope of $\overline{\eta}_{1\to1}$.
It is obvious that this upper bound is asymptotically sharp as $n\to\infty$,
i.e., $\overline{\eta}_{1\to1}^{(\infty)}(\alpha)=\overline{\eta}_{1\to1}^{*}(\alpha).$
We now characterize the Rényi concentration functions in terms of
the single-function version of BL exponents, as shown in the following
lemma. The proof of Lemma \ref{lem:RenyiSDPI} is provided in Appendix
\ref{sec:Proof-of-Lemma-RenyiSDPI}.
\begin{lem}[Equivalence]
\label{lem:RenyiSDPI} For $P_{XY}$, $p\in\overline{\mathbb{R}},q\in\mathbb{R}\backslash\{0,1\}$
and $\alpha\in\mathcal{E}_{p,1}(\mathcal{X})$, we have 
\begin{equation}
\overline{\eta}_{p\to q}(\alpha)=\begin{cases}
q'(\alpha-\underline{\Gamma}_{p,q,1}(\alpha)) & q>1\textrm{ or }q<0\\
q'(\alpha-\overline{\Gamma}_{p,q,1}(\alpha)) & 0<q<1
\end{cases}\label{eq:-88}
\end{equation}
and 
\begin{equation}
\underline{\eta}_{p\to q}(\alpha)=\begin{cases}
q'(\alpha-\overline{\Gamma}_{p,q,1}(\alpha)) & q>1\textrm{ or }q<0\\
q'(\alpha-\underline{\Gamma}_{p,q,1}(\alpha)) & 0<q<1
\end{cases}.\label{eq:-89}
\end{equation}
\end{lem}

Combining Lemma \ref{lem:RenyiSDPI} and Theorem \ref{thm:BLExponentSingleFunc}
yields the following result. 
\begin{thm}[Generalized Mrs. Gerber's Lemma]
\label{thm:Gerber} For $p\in\overline{\mathbb{R}},q\in\mathbb{R}\backslash\{0,1\}$,
and $\alpha\in\mathcal{E}_{p,1}(\mathcal{X})$, we have for $q\in(0,\infty)\backslash\{1\}$,
\begin{align*}
\overline{\eta}_{p\to q}^{(n)}(\alpha) & \leq\overline{\eta}_{p\to q}^{*}(\alpha):=\begin{cases}
q'(\alpha-\underline{\Gamma}_{p,q,1}^{*}(\alpha)) & q>1,\\
q'(\alpha-\overline{\Gamma}_{p,q,1}^{*}(\alpha)) & 0<q<1,
\end{cases}
\end{align*}
and for $q\in(-\infty,0)$, 
\[
\underline{\eta}_{p\to q}^{(n)}(\alpha)\geq\underline{\eta}_{p\to q}^{*}(\alpha):=q'(\alpha-\overline{\Gamma}_{p,q,1}^{*}(\alpha)),
\]
where $\underline{\Gamma}_{p,q,1}^{*},\overline{\Gamma}_{p,q,1}^{*}$
are respectively defined in \eqref{eq:GammaForward} and \eqref{eq:GammaReverse}.
Moreover, for finite alphabets $\mathcal{X},\mathcal{Y}$, these two
inequalities are asymptotically tight as $n\to\infty$ for given $(p,q,\alpha)$
with $\alpha\neq0,\alpha_{\max}$, i.e., $\overline{\eta}_{p\to q}^{(\infty)}(\alpha)=\overline{\eta}_{p\to q}^{*}(\alpha)$
and $\underline{\eta}_{p\to q}^{(\infty)}(\alpha)=\underline{\eta}_{p\to q}^{*}(\alpha)$.
\end{thm}

Mrs. Gerber's lemma in \cite{wyner1973theorem} only focuses on the
KL divergence and the doubly symmetric binary distribution. In contrast,
Theorem \ref{thm:Gerber} corresponds to a general version of Mrs.
Gerber's lemma for the Rényi divergence and arbitrary distributions
on Polish spaces. We provide tight bounds in Theorem \ref{thm:Gerber}
only for the forward part with $q>0$ and the reverse part with $q<0$.
It is interesting to provide tight bounds for remaining cases. In
fact, the forward Rényi concentration function for $q<0$ and the
reverse Rényi concentration function for $q>0$ are closely related
to the Rényi-resolvability (or Rényi-covering) problem \cite{yu2018renyi}
in which additionally, the input distribution $Q_{X^{n}}$ is restricted
to be uniform. By such a connection, the Rényi-resolvability results
in \cite{yu2018renyi} can be used to derive some interesting bounds
on the Rényi concentration functions. For example, for finite $\mathcal{X},\mathcal{Y}$
and $p\in(0,\infty],q\in(0,2]\cup\{\infty\}$, it holds that $\underline{\eta}_{p\to q}^{(\infty)}(\alpha|P_{X},P_{Y|X})=0$
for all $\alpha\in[0,H(P_{X})-R_{q}]$, where 
\begin{align*}
R_{q} & =\begin{cases}
\mathbb{E}_{X\sim P_{X}}[D_{q}(P_{Y|X}(\cdot|X)\|P_{Y})], & q\in(1,2]\cup\{\infty\}\\
D(P_{Y|X}\|R_{Y}|P_{X}), & q\in(0,1]
\end{cases}.
\end{align*}

\subsection{\label{subsec:Strong-Small-Set-Expansion}Noise Stability}

Consider the following \emph{noise stability} or \emph{noninteractive
correlation distillation} problem \cite{witsenhausen1975sequences,mossel2006non}.
For two events $A\in\mathbb{B}_{\mathcal{X}}^{\otimes n},B\in\mathbb{B}_{\mathcal{Y}}^{\otimes n}$,
if the marginal probabilities $P_{X}^{\otimes n}(A)$ and $P_{Y}^{\otimes n}(B)$
are given, then how large and how small can the joint probability
$P_{XY}^{\otimes n}(A\times B)$ be? First consider a relaxed version,
in which $P_{X}^{\otimes n}(A)$ and $P_{Y}^{\otimes n}(B)$ are bounded,
instead exactly given. Define 
\begin{align}
\underline{\Theta}^{(n)}(\alpha,\beta) & :=-\frac{1}{n}\log\sup_{\substack{A\in\mathbb{B}_{\mathcal{X}}^{\otimes n},B\in\mathbb{B}_{\mathcal{Y}}^{\otimes n}:\\
P_{X}^{\otimes n}(A)\leq e^{-n\alpha},\\
P_{Y}^{\otimes n}(B)\leq e^{-n\beta}
}
}P_{XY}^{\otimes n}(A\times B),\label{eq:-72}\\
\overline{\Theta}^{(n)}(\alpha,\beta) & :=-\frac{1}{n}\log\inf_{\substack{A\in\mathbb{B}_{\mathcal{X}}^{\otimes n},B\in\mathbb{B}_{\mathcal{Y}}^{\otimes n}:\\
P_{X}^{\otimes n}(A)\geq e^{-n\alpha},\\
P_{Y}^{\otimes n}(B)\geq e^{-n\beta}
}
}P_{XY}^{\otimes n}(A\times B),\label{eq:-71}
\end{align}
and $\underline{\Theta}^{(\infty)}$ and $\overline{\Theta}^{(\infty)}$
as their limits as $n\to\infty$. Then, the small-set expansion theorem
\cite{O'Donnell14analysisof} provides some bounds for these two quantities,
but those bounds are not asymptotically tight. Our strong BL inequalities
in Theorem \ref{thm:strongBL} imply the following strong version
of the small-set expansion theorem. The proof is provided in Appendix
\ref{sec:Proof-of-Theorem-strongsse}. 
\begin{thm}[Strong Small-Set Expansion Theorem]
\label{thm:strongsse} For any $n\ge1$ and $\ensuremath{\alpha\in[0,\alpha_{\max}],\beta\in[0,\beta_{\max}]}$,
\begin{align}
\underline{\Theta}^{(n)}(\alpha,\beta) & \ge\underline{\Theta}(\alpha,\beta),\label{eq:-15}\\
\overline{\Theta}^{(n)}(\alpha,\beta) & \leq\begin{cases}
\overline{\Theta}(\alpha,\beta) & \alpha,\beta>0\\
\alpha & \beta=0\\
\beta & \alpha=0
\end{cases},\label{eq:-64}
\end{align}
where $\underline{\Theta}(\alpha,\beta),\overline{\Theta}(\alpha,\beta)$
are respectively defined in \eqref{eq:ThetaUnderline} and \eqref{eq:ThetaOverline}.
Moreover, the inequalities in \eqref{eq:-15} and \eqref{eq:-64}
are asymptotically tight as $n\to\infty$, i.e., for $\alpha\in[0,\alpha_{\max}],\beta\in[0,\beta_{\max}]$,
\begin{align}
\underline{\Theta}^{(\infty)}(\alpha,\beta) & =\underline{\Theta}(\alpha,\beta)\label{eq:-78}\\
\overline{\Theta}^{(\infty)}(\alpha,\beta) & =\begin{cases}
\overline{\Theta}(\alpha,\beta) & \alpha,\beta>0\\
\alpha & \beta=0\\
\beta & \alpha=0
\end{cases}.\label{eq:-79}
\end{align}
\end{thm}

\begin{rem}
\label{rem:equality} Observe that $P_{XY}^{\otimes n}(A\times B)$
is nondecreasing in $(A,B)$ in the sense that $P_{XY}^{\otimes n}(A\times B)\le P_{XY}^{\otimes n}(A'\times B')$
for any $A\subseteq A',B\subseteq B'$. By this property, for the
forward noise stability, given $\ensuremath{\alpha\in[0,\alpha_{\max}],\beta\in[0,\beta_{\max}]}$,
there exists a sequence $\left\{ (A_{n},B_{n})\right\} $ such that
$P_{X}^{\otimes n}(A_{n})\le e^{-n\alpha},P_{Y}^{\otimes n}(B_{n})\le e^{-n\beta}$
for all $n$, and moreover, $-\frac{1}{n}\log P_{X}^{\otimes n}(A_{n})\downarrow\alpha,\,-\frac{1}{n}\log P_{Y}^{\otimes n}(B_{n})\downarrow\beta$,
and $-\frac{1}{n}\log P_{XY}^{\otimes n}(A_{n}\times B_{n})\downarrow\underline{\Theta}(\alpha,\beta)$
as $n\to\infty$. In fact, each $(A_{n},B_{n})$ can be chosen in
the following way: First, choose them as an optimal pair attaining
$\underline{\Theta}^{(n)}(\alpha,\beta)$, and then enlarge them as
long as possible under the condition $P_{X}^{\otimes n}(A)\leq e^{-n\alpha},\,P_{Y}^{\otimes n}(B)\leq e^{-n\beta}$.
 Similarly, for the reverse noise stability, given $(\alpha,\beta)$
in the effective region of $\overline{\Theta}$, there exists a sequence
of $(A_{n},B_{n})$ such that $P_{X}^{\otimes n}(A_{n})\ge e^{-n\alpha},P_{Y}^{\otimes n}(B_{n})\ge e^{-n\beta}$
for all $n$, and moreover, $-\frac{1}{n}\log P_{X}^{\otimes n}({A}_{n})\uparrow\alpha,-\frac{1}{n}\log P_{Y}^{\otimes n}({B}_{n})\uparrow\beta$,
and $-\frac{1}{n}\log P_{XY}^{\otimes n}({A}_{n}\times{B}_{n})\to\overline{\Theta}(\alpha,\beta).$ 
\end{rem}

This theorem for finite alphabets $\mathcal{X},\mathcal{Y}$ was first
proven by the present author together with Anantharam and Chen \cite{yu2021Graphs}.
Our theorem here is a generalization of the one in \cite{yu2021Graphs}
to arbitrary distributions on Polish spaces. As discussed in \cite{yu2021Graphs},
our strong small-set expansion theorem is stronger than O'Donnell's
forward small-set expansion theorem \cite{O'Donnell14analysisof}
and Mossel et al's reverse small-set expansion theorem \cite{mossel2006non}.
Moreover, for the limiting case of $n\to\infty$, our theorems reduce
to theirs for a sequence of pairs $(\alpha_{n},\beta_{n})$ such that
$\alpha_{n}=\kappa_{n}\alpha_{0},\beta_{n}=\kappa_{n}\beta_{0},\lim_{n\to\infty}\kappa_{n}=0$,
e.g., $e^{-n\alpha_{n}}=a,e^{-n\beta_{n}}=b$ for some fixed $0<a,b<1$.
Similarly to the observation in Section \ref{subsec:Strong-Hypercontractivity-Inequa},
conceptually, our small-set expansion theorem here can be seen as
the large deviation theorem of the noise stability problem; while
the forward and reverse small-set expansion theorems in \cite{O'Donnell14analysisof}
and \cite{mossel2006non} correspond to the moderate deviation theorem
of the small-set expansion problem. The moderate deviation theorem
can be recovered from the large deviation theorem by letting $\alpha,\beta$
go to zero in a certain speed, if $P_{XY}$ satisfies the following
condition: $\underline{\Theta}(s,t)\ge\frac{s}{p}+\frac{t}{q}$ and
$\overline{\Theta}(s,t)\le\frac{s}{p}+\frac{t}{q}$ hold for all $s,t\ge0$
if and only if they hold for a neighborhood of the origin. This condition
is satisfied if $\underline{\Theta}$ is convex and $\overline{\Theta}$
is concave. It is known that this condition is satisfied by the \emph{doubly
symmetric binary distribution} given in \eqref{eq:DSBS} in Section
\ref{sec:Examples} \cite{yu2021common}. For the doubly symmetric
binary distribution, the small-set expansion theorems in \cite{O'Donnell14analysisof}
and \cite{mossel2006non} are sharp in the moderate deviation regime
\cite{yu2021common}. Furthermore, strengthening the small-set expansion
theorem for doubly symmetric binary distributions was recently studied
by Ordentlich, Polyanskiy, and Shayevitz \cite{ordentlich2020note},
but they only solved the limiting cases as the correlation coefficient
$\rho\to0,1$. The symmetric case $\alpha=\beta$ was solved by Kirshner
and Samorodnitsky \cite{kirshner2019moment}. Our theorem here is
a generalization of theirs.

We can further strengthen the strong small-set theorem once the exact
values of the marginal probabilities are given. 
\begin{thm}[Strong Small-Set Expansion Theorem]
\label{thm:strongsse-2} For any $n\ge1$ and any subsets $A\in\mathbb{B}_{\mathcal{X}}^{\otimes n},B\in\mathbb{B}_{\mathcal{Y}}^{\otimes n}$,
we have 
\begin{equation}
\underline{\Theta}(\alpha,\beta)\le-\frac{1}{n}\log P_{XY}^{\otimes n}(A\times B)\le\overline{\Theta}^{*}(\alpha,\beta):=\begin{cases}
\invbreve{\varphi}(\alpha,\beta) & \alpha,\beta>0\\
\alpha & \beta=0\\
\beta & \alpha=0
\end{cases}.\label{eq:strongsse}
\end{equation}
where $\alpha:=-\frac{1}{n}\log P_{X}^{\otimes n}(A),\,\beta:=-\frac{1}{n}\log P_{Y}^{\otimes n}(B)$,
and $\overline{\varphi}$ is defined in \eqref{eq:phiOverline}. Moreover,
both the lower and upper bounds above are asymptotically sharp as
$n\to\infty$. That is, for any $\alpha\in[0,\alpha_{\max}],\beta\in[0,\beta_{\max}]$,
there exists a sequence of $(A_{n},B_{n})$ such that $-\frac{1}{n}\log P_{X}^{\otimes n}({A}_{n})\downarrow\alpha,-\frac{1}{n}\log P_{Y}^{\otimes n}({B}_{n})\downarrow\beta$,
and $-\frac{1}{n}\log P_{XY}^{\otimes n}({A}_{n}\times{B}_{n})\to\underline{\Theta}(\alpha,\beta)$
as $n\to\infty$, and there also exists another sequence of $(A_{n},B_{n})$
such that $-\frac{1}{n}\log P_{X}^{\otimes n}({A}_{n})\uparrow\alpha,-\frac{1}{n}\log P_{Y}^{\otimes n}({B}_{n})\uparrow\beta$,
and $-\frac{1}{n}\log P_{XY}^{\otimes n}({A}_{n}\times{B}_{n})\to\overline{\Theta}^{*}(\alpha,\beta)$
as $n\to\infty$. 
\end{thm}

\begin{proof}
The lower bound in \eqref{eq:strongsse} follows by Theorem \ref{thm:strongsse},
and its exponential sharpness follows by Remark \ref{rem:equality}.
It is easy to verify the upper bound for $\alpha=0$ or $\beta=0$.
We now prove the upper bound for $\alpha,\beta>0$ follows by the
reverse strong BL inequality in Corollary \ref{cor:BL}. Specifically,
we set $f=1_{A},g=1_{B}$, $p>\hat{p}>0,q>\hat{q}>0$, and get 
\begin{align*}
 & -\frac{1}{n}\log P_{XY}^{\otimes n}(A\times B)\le\overline{\Lambda}_{p,q,\hat{p},\hat{q}}^{*}(\alpha,\beta)\\
 & \le\sup_{\substack{Q_{W},Q_{X|W},Q_{Y|W}:\\
D(Q_{X|W}\|P_{X}|Q_{W})\ge\alpha,\\
D(Q_{Y|W}\|P_{Y}|Q_{W})\ge\beta
}
}\mathbb{D}(Q_{X|W},Q_{Y|W}\|P_{XY}|Q_{W})+\frac{\alpha-D(Q_{X|W}\|P_{X}|Q_{W})}{p}+\frac{\beta-D(Q_{Y|W}\|P_{Y}|Q_{W})}{q}\\
 & =\sup_{s\ge\alpha,t\ge\beta}\invbreve{\varphi}(s,t)+\frac{\alpha-s}{p}+\frac{\beta-t}{q}
\end{align*}
for any $p,q>0$. We first suppose that $\invbreve{\varphi}(\alpha,\beta)<\infty$
for any finite $(\alpha,\beta)$. Since $\invbreve{\varphi}$ is concave,
any supergradient of $\invbreve{\varphi}$ at $(\alpha,\beta)$ for
$\alpha,\beta>0$ must be finite. Let $(u,v)$ be  a supergradient
of $\invbreve{\varphi}$ at $(\alpha,\beta)$. Then, if we choose
$p,q$ sufficiently small such that $\frac{1}{p}\ge u,\frac{1}{q}\ge v$,
we have 
\begin{align*}
-\frac{1}{n}\log P_{XY}^{\otimes n}(A\times B) & \le\sup_{s\ge\alpha,t\ge\beta}\invbreve(\overline{\varphi}(s,t)-us-vt)+(u-\frac{1}{p})s+(v-\frac{1}{q})t+\frac{\alpha}{p}+\frac{\beta}{q}\\
 & =\invbreve{\varphi}(\alpha,\beta)-u\alpha-v\beta+(u-\frac{1}{p})\alpha+(v-\frac{1}{q})\beta+\frac{\alpha}{p}+\frac{\beta}{q}\\
 & =\invbreve{\varphi}(\alpha,\beta).
\end{align*}
We next suppose $\invbreve{\varphi}(\alpha,\beta)=\infty$ for some
finite $(\alpha,\beta)$. By the concavity, $\invbreve{\varphi}(\alpha,\beta)=\infty$
for all $\alpha,\beta>0$. Hence, the upper bound in \eqref{eq:strongsse}
follows trivially. The asymptotic sharpness of the upper bound in
\eqref{eq:strongsse} follows similarly to the asymptotic sharpness
of \eqref{eq:-64} in Theorem \ref{thm:strongsse}. 
\end{proof}

\subsection{\label{subsec:-Stability}$q$-Stability}

In this subsection, we apply our results to the \emph{$q$-stability}
of Boolean functions. The $q$-stability problem concerns the following
question: For an event $A\in\mathbb{B}_{\mathcal{X}}^{\otimes n}$,
if the probability $P_{X}^{\otimes n}(A)$ is given, then how large
and how small could the noisy version $\Vert P_{X|Y}^{\otimes n}(A|\cdot)\Vert_{q}$
be? Similarly to the small-set expansion case, we first consider a
relaxed version, in which $P_{X}^{\otimes n}(A)$ is bounded, instead
exactly given. For $\alpha\in[0,\alpha_{\max}]$, we aim at characterizing

\begin{align}
\Theta_{q}^{(n)}(\alpha) & :=\begin{cases}
-\frac{1}{n}\log\sup_{\substack{A\in\mathbb{B}_{\mathcal{X}}^{\otimes n}:P_{X}^{\otimes n}(A)\leq e^{-n\alpha}}
}\Vert P_{X|Y}^{\otimes n}(A|\cdot)\Vert_{q} & q\ge1\\
-\frac{1}{n}\log\inf_{\substack{A\in\mathbb{B}_{\mathcal{X}}^{\otimes n}:P_{X}^{\otimes n}(A)\geq e^{-n\alpha}}
}\Vert P_{X|Y}^{\otimes n}(A|\cdot)\Vert_{q} & q\in(-\infty,1)\backslash\{0\}
\end{cases}\label{eq:-72-3}
\end{align}
and $\Theta_{q}^{(\infty)}$ as its limit as $n\to\infty$. The $q$-stability
problem in Gaussian measure spaces was previously investigated in
\cite{eldan2015two}, and the one in Hamming spaces was studied in
\cite{li2019boolean}. As a consequence of Theorem \ref{thm:strongBL_single},
we have the following result for more general distributions. The proof
is provided in Appendix \ref{sec:Proof-of-Theorem-strongsse-1}. 
\begin{thm}[Strong $q$-Stability Theorem]
\label{thm:strongqstability} For any $n\ge1$ and $\ensuremath{\alpha\in[0,\alpha_{\max}]}$,
\begin{align}
\Theta_{q}^{(n)}(\alpha) & \ge\Theta_{q'}(\alpha),\qquad q\in[1,\infty),\label{eq:-63}\\
\Theta_{q}^{(n)}(\alpha) & \leq\begin{cases}
\Theta_{q'}(\alpha) & \alpha>0\\
0 & \alpha=0
\end{cases},\qquad q\in(-\infty,1)\backslash\{0\},\label{eq:-69}
\end{align}
where $\Theta_{q'}$ is defined in \eqref{eq:Theta_q}. Moreover,
\eqref{eq:-63} and \eqref{eq:-69} are asymptotically tight for arbitrary
Polish spaces $\mathcal{X},\mathcal{Y}$. 
\end{thm}

Our results here strengthen and generalize the (single-function version
of) forward small-set expansion theorem due to Kahn, Kalai, and Linial
\cite{kahn1988influence} and the (single-function version of) reverse
small-set expansion theorem due to Mossel et al \cite{mossel2006non}.

We can further strengthen the strong $q$-stability theorem once the
exact values of the marginal probability are given. The proof is similar
to the one of Theorem \ref{thm:strongsse-2}, and hence, omitted. 
\begin{thm}[Strong $q$-Stability Theorem]
\label{thm:strongsse-2-1} For any $n\ge1$ and any subsets $A\in\mathbb{B}_{\mathcal{X}}^{\otimes n}$,
we have 
\begin{align}
-\frac{1}{n}\log\Vert P_{X|Y}^{\otimes n}(A|\cdot)\Vert_{q} & \ge\Theta_{q'}(\alpha),\qquad q\in[1,\infty),\label{eq:-70}\\
-\frac{1}{n}\log\Vert P_{X|Y}^{\otimes n}(A|\cdot)\Vert_{q} & \leq\begin{cases}
\invbreve\varphi_{q'}(\alpha) & \alpha>0\\
0 & \alpha=0
\end{cases},\qquad q\in(-\infty,1)\backslash\{0\},\label{eq:-73}
\end{align}
where $\alpha:=-\frac{1}{n}\log P_{X}^{\otimes n}(A)$, and $\varphi_{r}$
is defined in \eqref{eq:phi_r} for $r>0$ and in \eqref{eq:phi_q}
for $r<0$. Moreover, \eqref{eq:-70} and \eqref{eq:-73} are asymptotically
sharp for arbitrary Polish spaces. 
\end{thm}

Similarly to the Rényi concentration function, it is also interesting
to investigate the variant of $q$-stability $\Theta_{q}^{(n)}$ in
which the supremum in the case $q\ge1$ is replaced by an infimum,
and the infimum in the case $q<1$ is replaced by a supremum. This
variant is in fact closely related to the Rényi-resolvability problem,
and hence, the results on the Rényi-resolvability problem in \cite{yu2018renyi}
can be used to derive some interesting bounds for it.

\section{\label{sec:Examples}Examples: Binary Case}

In this section, the bases of logarithms are set to $2$. Consider
a doubly symmetric binary distribution $P_{XY}$ with correlation
coefficient $\rho\in(0,1)$, i.e., 
\begin{equation}
P_{XY}=\begin{array}{ccc}
X\backslash Y & 0 & 1\\
0 & \frac{1+\rho}{4} & \frac{1-\rho}{4}\\
1 & \frac{1-\rho}{4} & \frac{1+\rho}{4}
\end{array}.\label{eq:DSBS}
\end{equation}
Define $\kappa=(\frac{1+\rho}{1-\rho})^{2}$. Define 
\begin{align*}
D(a) & :=D(a\|\frac{1}{2})=1-H_{2}(a),\\
D_{a,b}(p) & :=D((p,a-p,b-p,1+p-a-b)\|(\frac{1+\rho}{4},\frac{1-\rho}{4},\frac{1-\rho}{4},\frac{1+\rho}{4})),\\
\mathbb{D}(a,b) & :=\min_{0,\,a+b-1\le p\le a,\,b}D_{a,b}(p)\,=D_{a,b}(p_{a,b}^{*}),
\end{align*}
where $H_{2}:t\in[0,1]\mapsto-t\log_{2}t-(1-t)\log_{2}(1-t)$ is the
binary entropy function, and 
\[
p_{a,b}^{*}=\frac{(\kappa-1)(a+b)+1-\sqrt{((\kappa-1)(a+b)+1)^{2}-4\kappa(\kappa-1)ab}}{2(\kappa-1)}.
\]
For the doubly symmetric binary distribution $P_{XY}$, 
\begin{align}
\underline{\varphi}(s,t) & =\mathbb{D}(H_{2}^{-1}(1-s),H_{2}^{-1}(1-t))\label{eq:phi-1}\\
\overline{\varphi}(s,t) & =\mathbb{D}(H_{2}^{-1}(1-s),1-H_{2}^{-1}(1-t))\nonumber \\
\varphi_{r}(s) & =\begin{cases}
\min_{0\le t\le1}\underline{\varphi}(s,t)-\frac{t}{r} & r\ge1\mbox{ or }r<0\\
\max_{0\le t\le1}\overline{\varphi}(s,t)-\frac{t}{r} & 0<r<1
\end{cases}\nonumber 
\end{align}
where $H_{2}^{-1}$ is the inverse of the restriction of the binary
entropy function $H_{2}$ to the set $[0,\frac{1}{2}]$. Define the
corresponding lower and upper increasing envelopes of the functions
above as 
\begin{align}
\underline{\psi}(\alpha,\beta) & =\min_{s\ge\alpha,t\ge\beta}\underline{\varphi}(s,t),\label{eq:lce-1}\\
\overline{\psi}(\alpha,\beta) & =\max_{s\leq\alpha,t\leq\beta}\overline{\varphi}(s,t),\label{eq:uce-1}\\
\psi_{r}(\alpha) & =\begin{cases}
\min_{s\ge\alpha}\varphi_{r}(s) & r\ge1\\
\max_{s\leq\alpha}\varphi_{r}(s) & r<1,r\neq0
\end{cases}.
\end{align}
It is not difficult to see that $\overline{\psi}=\overline{\varphi}$,
and $\psi_{r}=\varphi_{r}$ for all $r\neq0$; see \cite{yu2021Graphs}.
By definition, $\underline{\Theta}$ is the lower convex envelope
of $\underline{\psi}$, $\underline{\Theta}$ is the upper concave
envelope of $\overline{\psi}$, and $\Theta_{r}$ is the lower convex
envelope of $\psi_{r}$ for $r\ge1$ and the upper concave envelope
of $\psi_{r}$ for $r<1,r\neq0$. All these functions for $\rho=0.9$
are plotted in Fig. \ref{fig:upsilon}. 
\begin{figure}
\centering %
\begin{tabular}{cc}
\includegraphics[width=0.5\columnwidth]{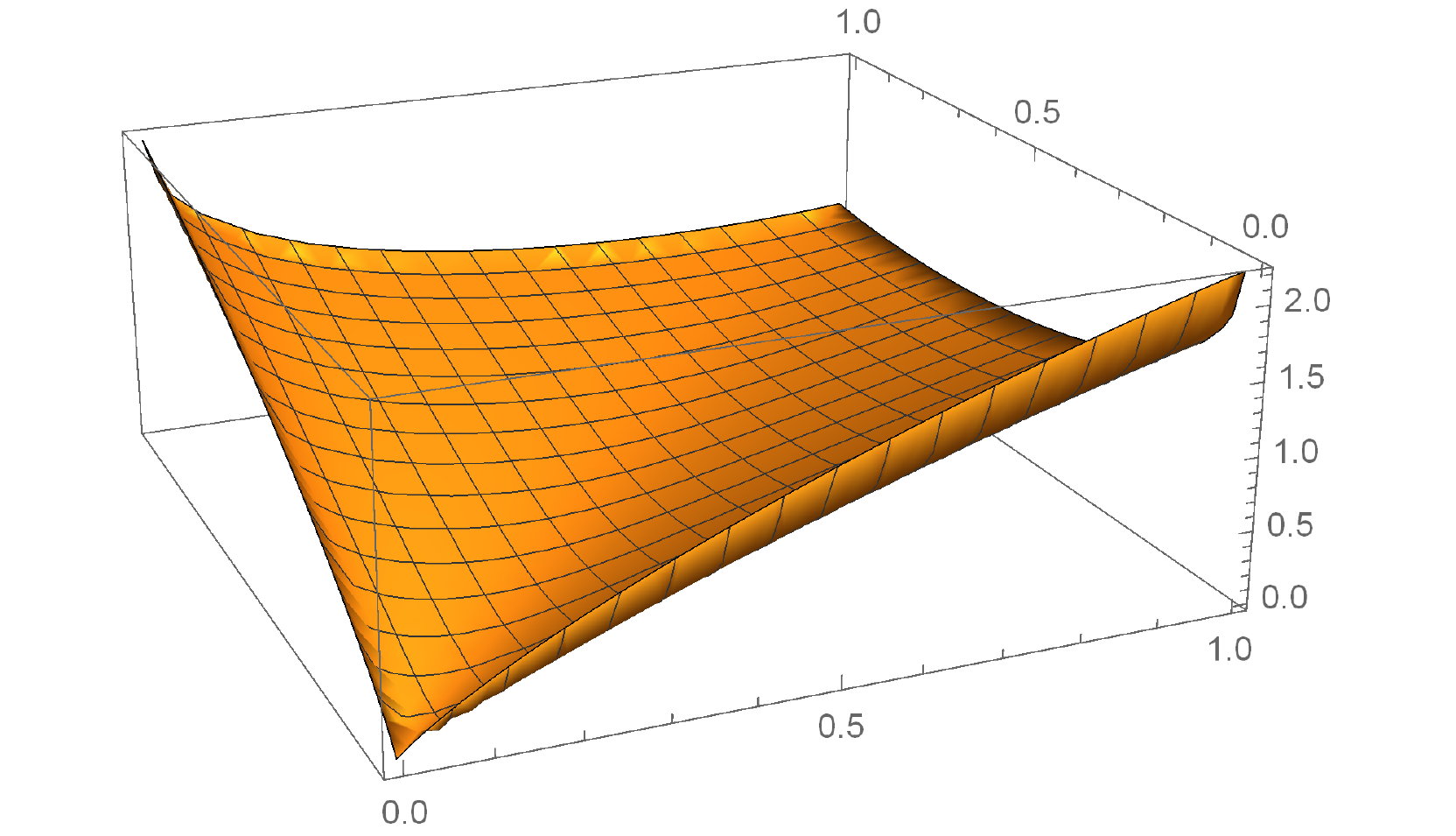}  & \includegraphics[width=0.4\columnwidth,height=4cm]{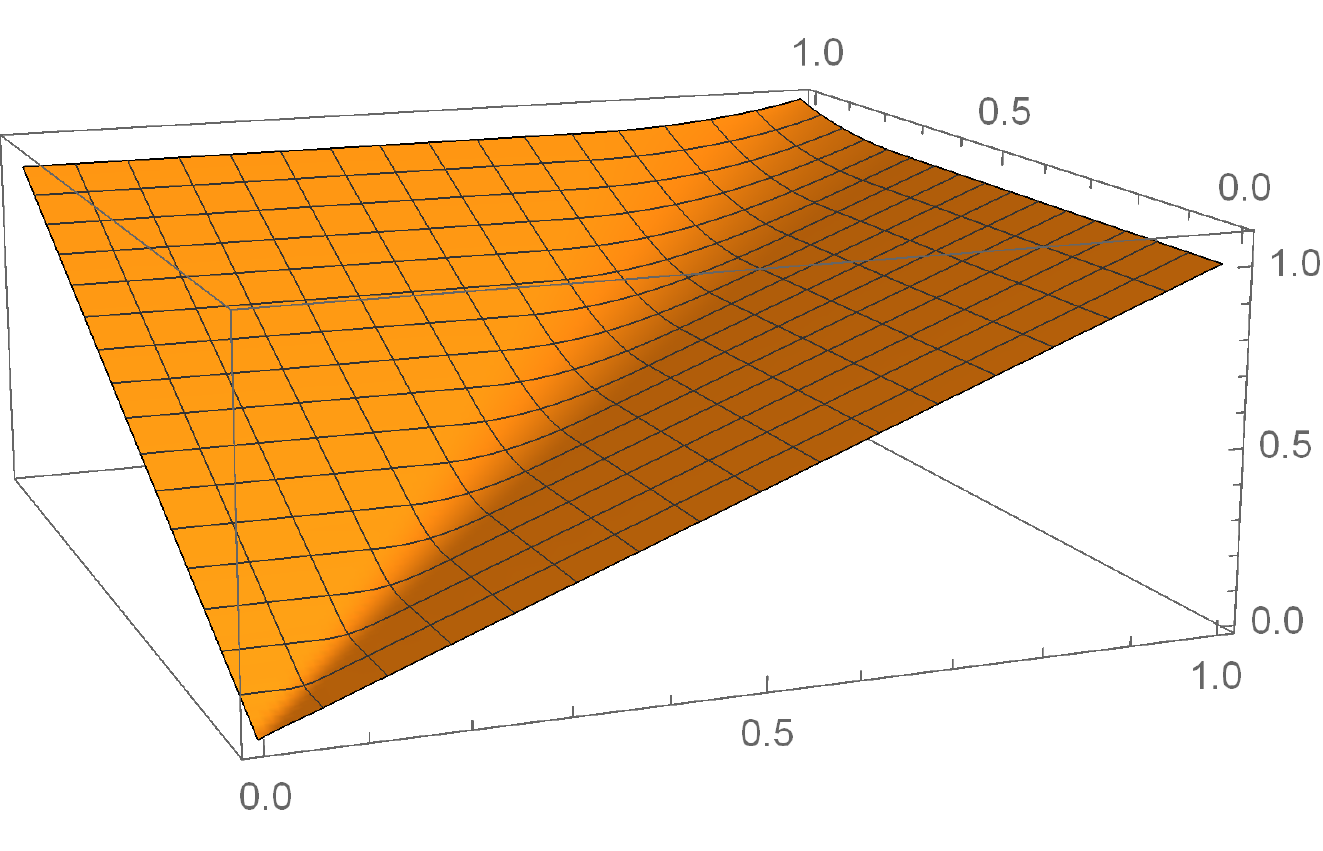}\tabularnewline
{\footnotesize{}{}{}{}{}$\underline{\varphi}$}  & {\footnotesize{}{}{}{}{}$\underline{\Theta}=\underline{\psi}$}\tabularnewline
\end{tabular}

\begin{tabular}{cc}
\includegraphics[width=0.5\columnwidth]{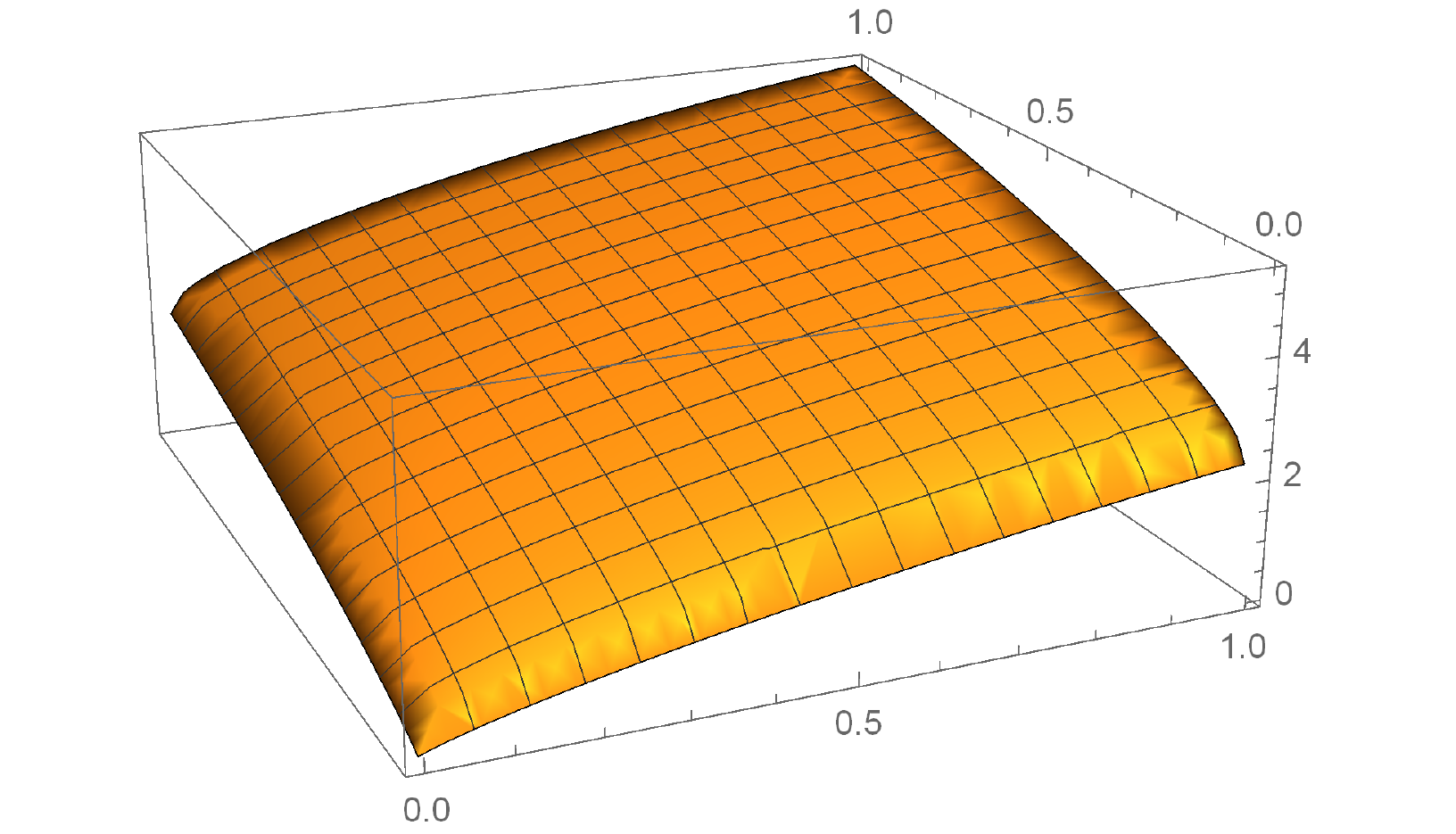}  & \includegraphics[width=0.4\columnwidth]{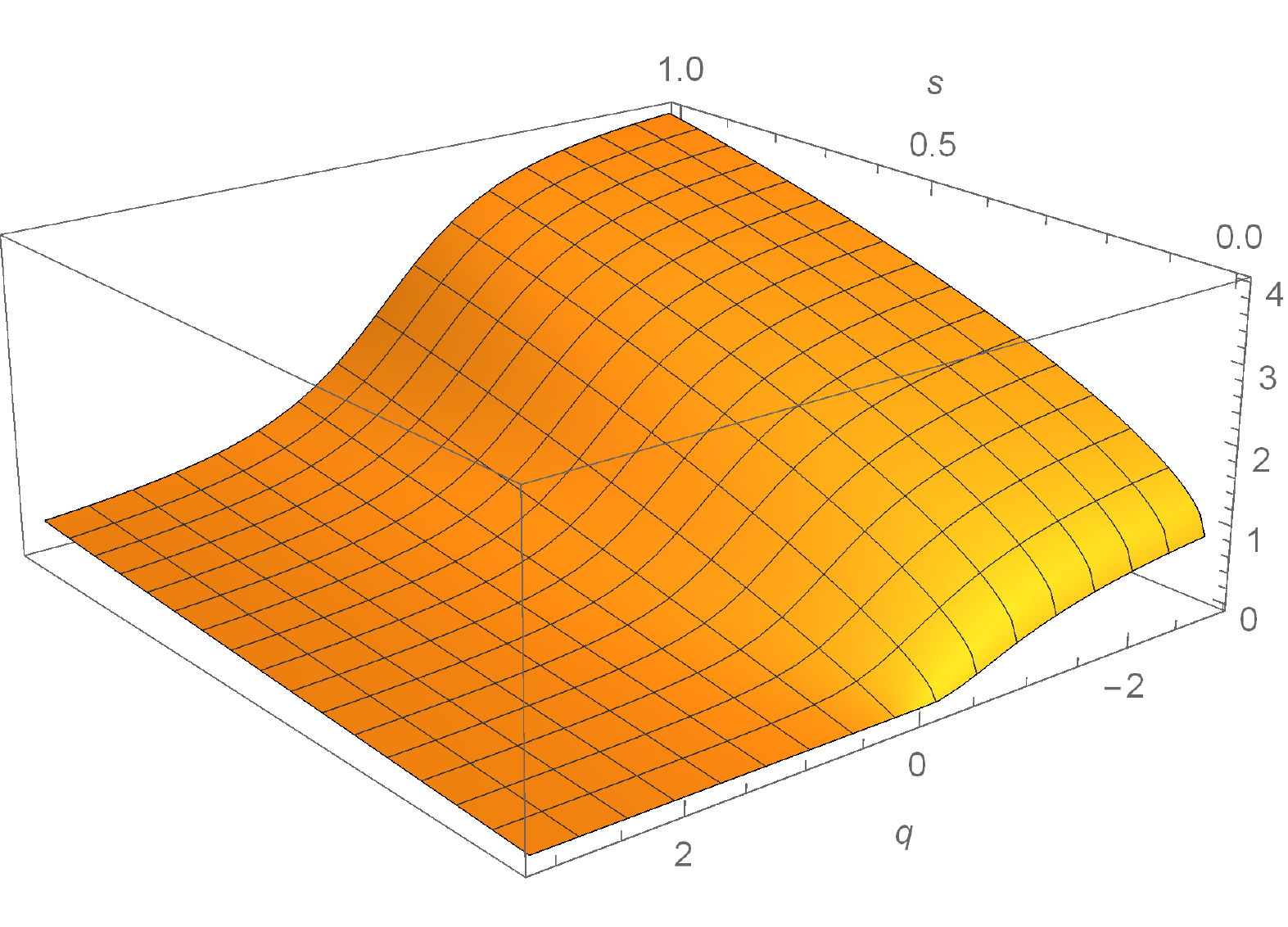}\tabularnewline
{\footnotesize{}{}{}{}{}$\overline{\Theta}=\overline{\psi}=\overline{\varphi}$}  & {\footnotesize{}{}{}{}{}$\Theta_{q'}=\psi_{q'}=\varphi_{q'}$}\tabularnewline
\end{tabular}

\caption{\label{fig:upsilon}Illustration of $\underline{\varphi}$, $\underline{\psi}=\underline{\Theta}$,
$\overline{\Theta}=\overline{\psi}=\overline{\varphi}$, and $\Theta_{q'}=\psi_{q'}=\varphi_{q'}$
for $\rho=0.9$. It can be observed from these figures that $\underline{\Theta}=\underline{\psi}$
is convex, $\overline{\Theta}=\overline{\psi}=\overline{\varphi}$
is concave, and $\Theta_{q'}=\psi_{q'}=\varphi_{q'}$ is convex for
$q\ge1$ and convex for $q<1$. The equalities and convexity and concavity
mentioned here have been shown in \cite{yu2021convexity}.}
\end{figure}

By observing that $\underline{\Theta},\overline{\Theta},\Theta_{q}$
are respectively the optimal exponents for the noise stability and
$q$-stability  (see Sections \ref{subsec:Strong-Small-Set-Expansion}
and \ref{subsec:-Stability}), it is not difficult to show the following
properties of them. The proof is given in Appendix \ref{sec:Proof-of-Lemma-Theta}. 
\begin{lem}
\label{lem:Theta} For a distribution $P_{XY}$ on the product of
two finite alphabets, define $\alpha_{\min}:=-\log\max_{x}P_{X}(x)$
and $\beta_{\min}:=-\log\max_{y}P_{Y}(y)$. Then for $\alpha,\beta,s,t\ge0$
such that $\alpha+s\leq\alpha_{\min},\beta+t\leq\beta_{\min}$, we
have 
\begin{align}
\underline{\Theta}(\alpha+s,\beta+t)-\underline{\Theta}(\alpha,\beta) & \leq s+t\label{eq:sublinear}\\
\overline{\Theta}(\alpha+s,\beta+t)-\overline{\Theta}(\alpha,\beta) & \geq s+t\label{eq:sublinear-1}\\
\Theta_{q'}(\alpha+s)-\Theta_{q'}(\alpha) & \le s\;\textrm{ for }q\ge1\label{eq:sublinear-2}\\
\Theta_{q'}(\alpha+s)-\Theta_{q'}(\alpha) & \geq s\;\textrm{ for }q\in(0,1).\label{eq:sublinear-3}
\end{align}
In particular, for the doubly symmetric binary distribution, \eqref{eq:sublinear}-\eqref{eq:sublinear-3}
hold for all $\alpha,\beta\in[0,1]$ and $0\le s\leq1-\alpha,0\le t\leq1-\beta$. 
\end{lem}

\begin{lem}
\label{lem:Theta-1} For the doubly symmetric binary distribution,
we have $\Theta_{r}=\breve{\varphi}_{r}$ for $r\ge1$, and $\Theta_{r}=\invbreve\varphi_{r}$
for $r<0$. 
\end{lem}

\begin{proof}
Here we only prove $\Theta_{r}=\invbreve\varphi_{r}$ for $r<0$.
The other equality follows similarly. By Lemma \ref{lem:exchange},
for the doubly symmetric binary distribution, for $r<0$, 
\begin{equation}
\Theta_{r}(\alpha)=\max_{s\leq\alpha}\invbreve\varphi_{r}(s).\label{eq:-93}
\end{equation}
By Lemma \ref{lem:Theta}, the supergradient of $\Theta_{r}$ is not
smaller than $1$. If $\invbreve\varphi_{r}$ is decreasing on an
interval, then by \eqref{eq:-93}, $\Theta_{r}$ is constant on that
interval, which contradicts with that the supergradient of $\Theta_{r}$
is not smaller than $1$. Hence, $\invbreve\varphi_{r}$ is nondecreasing.
This implies that $\Theta_{r}=\invbreve\varphi_{r}.$ 
\end{proof}
For the doubly symmetric binary distribution, by Theorems \ref{thm:BLexponent}
and \ref{thm:BLExponentSingleFunc} and Lemma \ref{lem:Theta}, we
obtain the following corollaries. 
\begin{cor}[Strong Brascamp--Lieb Inequalities for Binary Distributions with
$\hat{p}=\hat{q}=1$ (Two-Function Version)]
\label{thm:strongBL-1-1} Let $\alpha,\beta\in[0,1]$. Consider the
doubly symmetric binary distribution in \eqref{eq:DSBS}. Then the
following hold. 
\begin{enumerate}
\item For $p,q\in[1,\infty]$, we have 
\begin{align}
\langle f,g\rangle & \leq e^{-n(\min_{\alpha\le s\le1,\beta\le t\le1}\breve{\varphi}(s,t)-\frac{s}{p}-\frac{t}{q})}\Vert f\Vert_{p}\Vert g\Vert_{q}\label{eq:FBL-4}
\end{align}
for all $f:\mathcal{X}^{n}\to[0,\infty),g:\mathcal{Y}^{n}\to[0,\infty)$
satisfying $\frac{1}{n}\Ent_{p,1}(f)\ge\alpha$ and $\frac{1}{n}\Ent_{q,1}(g)\geq\beta$. 
\item For $p,q\in(0,1]$, we have 
\begin{align}
\langle f,g\rangle & \ge e^{-n(\max_{0\le s\le\alpha,0\le t\le\beta}\invbreve\varphi(s,t)-\frac{s}{p}-\frac{t}{q})}\Vert f\Vert_{p}\Vert g\Vert_{q}\label{eq:RBL-5}
\end{align}
for all $f:\mathcal{X}^{n}\to[0,\infty),g:\mathcal{Y}^{n}\to[0,\infty)$
satisfying $\frac{1}{n}\Ent_{p,1}(f)\le\alpha$ and $\frac{1}{n}\Ent_{q,1}(g)\le\beta$. 
\item Moreover, for given $(p,q,\alpha,\beta)$ with $\alpha,\beta\in(0,1)$,
the two inequalities above are exponentially sharp. 
\end{enumerate}
\end{cor}

\begin{proof}
By Lemma \ref{lem:Theta}, all the subgradients $(\lambda,\mu)$ of
$\underline{\Theta}$ satisfy $\lambda,\mu\le1$. By definition, $\underline{\Theta}(\alpha,\beta)=\min_{\alpha\le s\le1,\beta\le t\le1}\breve{\varphi}(s,t)$,
which implies that all the subgradients $(\lambda,\mu)$ of $\breve{\varphi}$
also satisfy $\lambda,\mu\le1$. By this property and Theorem \ref{thm:BLexponent}
with $p,q$ set to $1$ (so that $\hat{p},\hat{q}$ can be equal to
$\pm\infty$), we have \eqref{eq:FBL-4}. (Note that $p,q$ in \eqref{eq:FBL-4}
correspond to $\hat{p},\hat{q}$ in Theorem \ref{thm:BLexponent}.)
By Theorem \ref{thm:strongsse}, the inequality in \eqref{eq:FBL-4}
is exponentially sharp.

Inequality in \eqref{eq:RBL-5} follows directly from Theorem \ref{thm:BLExponentSingleFunc}.
The exponential sharpness follows by Theorem \ref{thm:strongsse-2}. 
\end{proof}
Similarly to the corollary above, we have the following corollary. 
\begin{cor}[Strong Brascamp--Lieb Inequality for Binary Distributions with $\hat{p}=1$
(Single-Function Version)]
\label{cor:-Let-}\label{thm:strongBL_single-1-1} Let $\alpha\in[0,1]$.
Consider the doubly symmetric binary distribution in \eqref{eq:DSBS}.
Then the following hold. 
\begin{enumerate}
\item For $p\in[1,\infty],q\in[1,\infty)$, we have 
\begin{align}
\Vert P_{X|Y}^{\otimes n}(f)\Vert_{q} & \leq e^{-n(\min_{\alpha\le s\le1}\breve{\varphi}_{q'}(s)-\frac{s}{p})}\Vert f\Vert_{p}\label{eq:FBL-3}
\end{align}
for all $f:\mathcal{X}^{n}\to[0,\infty)$ satisfying $\frac{1}{n}\Ent_{p,1}(f)\ge\alpha$. 
\item For $p\in(0,1],q\in(-\infty,1)\backslash\{0\}$, we have 
\begin{align}
\Vert P_{X|Y}^{\otimes n}(f)\Vert_{q} & \ge e^{-n(\max_{0\le s\le\alpha}\invbreve\varphi_{q'}(s)-\frac{s}{p})}\Vert f\Vert_{p}\label{eq:RBL-6}
\end{align}
for all $f:\mathcal{X}^{n}\to[0,\infty)$ satisfying $\frac{1}{n}\Ent_{p,1}(f)\le\alpha$. 
\item Moreover, for given $(p,q,\alpha)$ with $\alpha\in(0,1)$, the two
inequalities above are exponentially sharp. 
\end{enumerate}
\end{cor}

If $p\in[p^{*},\infty],q\in[1,\infty)$ where $\frac{1}{p^{*}}$ denotes
a subgradient of $\breve{\varphi}_{q'}$ at the point $\alpha$, then
the minimum in the RHS of \eqref{eq:FBL-3} is attained at $s=\alpha$.
 Similarly, if $p\in(0,p^{*}],q\in(-\infty,1)\backslash\{0\}$ where
$\frac{1}{p^{*}}$ denotes a supergradient of $\invbreve\varphi_{q'}$
at the point $\alpha$, then the maximum in the RHS of \eqref{eq:RBL-6}
is attained at $s=\alpha$. 

In a recent work \cite{yu2021convexity}, we have shown that for the
doubly symmetric binary distribution, $\underline{\psi}$ and $\varphi_{r}$
with $r\ge1$ are nondecreasing and convex, and $\overline{\varphi}$,
$\varphi_{r}$ with $r<1,r\neq0$ are nondecreasing and concave, which
imply that 
\begin{align*}
 & \underline{\Theta}=\underline{\psi}\leq\underline{\varphi},\quad\overline{\Theta}=\overline{\psi}=\overline{\varphi},\quad\Theta_{r}=\psi_{r}=\varphi_{r}.
\end{align*}
Hence, the functions $\invbreve\varphi,\breve{\varphi}_{q'},\invbreve\varphi_{q'}$
appearing in \eqref{eq:RBL-5}-\eqref{eq:RBL-6} can be respectively
replaced by $\overline{\varphi},\varphi_{q'}$ and $\varphi_{q'}$.
As for \eqref{eq:FBL-4}, since $p,q\ge1$ is assumed, $\breve{\varphi}$
in \eqref{eq:FBL-4} can be replaced by $\underline{\Theta}$, and
hence can be also replaced by $\underline{\varphi}$. In other words,
the time-sharing random variables can be removed for the exponents
at the RHSs of \eqref{eq:FBL-4}-\eqref{eq:RBL-6}. These further
imply that the exponential sharpness of \eqref{eq:FBL-4}-\eqref{eq:RBL-6}
is attained by the indicators of Hamming spheres (without time-sharing).
Specifically, the exponents of the two sides of \eqref{eq:FBL-4}
coincide asymptotically as $n\to\infty$ if we choose $f,g$ as a
pair of certain \emph{concentric} Hamming spheres, and the exponents
of the two sides of \eqref{eq:RBL-5} coincide asymptotically if we
choose $f,g$ as a pair of certain \emph{anti-concentric} Hamming
spheres. The exponents of the two sides of \eqref{eq:FBL-3} with
$q\in[1,\infty)$, as well as the exponents of the two sides of \eqref{eq:RBL-6}
with $q\in(-\infty,1)\backslash\{0\}$, respectively coincide asymptotically
if we choose $f$ as certain Hamming spheres. These results in fact
correspond to a conjecture of Polyansky on the forward strong Brascamp--Lieb
inequality \cite{kirshner2019moment} and a conjecture on the reverse
version. As mentioned in \cite{kirshner2019moment}, Polyansky's conjecture
was already solved by himself in an unpublished paper \cite{polyanskiy2019hypercontractivity2}.
Here, we independently resolve Polyansky's conjecture and its counterpart
on the reverse strong Brascamp--Lieb inequality. 

The inequality in \eqref{eq:FBL-3} immediately implies the following
forward version of strong hypercontractivity inequality. Given $\alpha\in[0,1]$
and $q>1$, denote $p^{*}$ as the minimum $p$ such that $\min_{\alpha\le s\le1}\varphi_{q'}(s)-\frac{s}{p}=0$.
Since $\varphi_{q'}$ is a convex function and $\varphi_{q'}(0)=0$,
indeed, it holds that $1/p^{*}=\varphi_{q'}(\alpha)/\alpha$. 
\begin{cor}
For $\alpha\in(0,1)$, $q\in[1,\infty)$, and any $p\ge p^{*}$ with
$1/p^{*}=\varphi_{q'}(\alpha)/\alpha$, we have 
\begin{align}
\Vert P_{X|Y}^{\otimes n}(f)\Vert_{q} & \leq\Vert f\Vert_{p}\label{eq:FBL-3-1-1}
\end{align}
for all $f:\mathcal{X}^{n}\to[0,\infty)$ satisfying $\frac{1}{n}\Ent_{p,1}(f)\ge\alpha$. 
\end{cor}

\section{\label{sec:Generalization-to-Original}Generalization to Original
BL Inequalities}

Consider a tuple $(\pi_{XY},\mu_{X},\mu_{Y})$, 
where $\pi_{XY}$ is a $\sigma$-finite nonnegative measure  on $(\mathcal{X}\times\mathcal{Y},\mathbb{B}_{\mathcal{X}}\otimes\mathbb{B}_{\mathcal{Y}})$,
and $\mu_{X},\mu_{Y}$ are two $\sigma$-finite nonnegative measures
respectively on $(\mathcal{X},\mathbb{B}_{\mathcal{X}})$ and $(\mathcal{Y},\mathbb{B}_{\mathcal{Y}})$.
Here $\mu_{X},\mu_{Y}$ are not necessarily to be the marginals of
$\pi_{XY}.$ We redefine $\langle f,g\rangle$, $\Vert f\Vert_{p}$,
and $\Vert g\Vert_{q}$ respectively with respect to $\pi_{XY},\mu_{X},\mu_{Y}$.
 Then, the forward BL inequality corresponds to the original BL inequality
studied in \cite{brascamp1976best}, but the reverse BL inequality
in \eqref{eq:RBL} is different from the one introduced in \cite{barthe1998reverse}.
In fact, both of them generalize the case in which $\mu_{X},\mu_{Y}$
are the marginals of $\pi_{XY}.$ 

For simplicity, the condition $\pi_{X}\ll\mu_{X},\pi_{Y}\ll\mu_{Y}$
is assumed. This does not lose any generality, due to the following
argument. If $\pi_{X}\nll\mu_{X}$, then we can write $\pi_{X}=\bar{\pi}_{X}+\hat{\pi}_{X}$
where $\bar{\pi}_{X}\ll\mu_{X}$ and $\hat{\pi}_{X}\perp\mu_{X}$.
By definition, there is a measurable set $A$ such that $\mu_{X}(A^{c})=0$
and $\hat{\pi}_{X}(A)=0$. So, we can also write $f$ as  $f=f1_{A}+f1_{A^{c}}$.
Note that $\Ent_{p,\hat{p}}(f)$ only depends on the part $f1_{A}$
of $f$, and is independent of the other part. Hence, $\Ent_{p,\hat{p}}(f)$
remains unchanged if we rechoose $f=f1_{A}+c1_{A^{c}}.$ For this
new function, $\langle f,g\rangle=\langle f1_{A},g\rangle+c\langle1_{A^{c}},g\rangle$,
which tends to infinity as $c\to\infty$. This implies $\underline{\Lambda}_{p,q,\hat{p},\hat{q}}(\alpha,\beta|\pi_{XY},\mu_{X},\mu_{Y})=-\infty$.
As for the reverse part $\overline{\Lambda}_{p,q,\hat{p},\hat{q}}(\alpha,\beta)$,
if $\pi_{X}\nll\mu_{X}$, by the argument above, without loss of optimality,
one can restrict $f$ to be $f1_{A}$. If $\pi_{X}(A)=0$, then $\langle f,g\rangle=0$,
which implies $\overline{\Lambda}_{p,q,\hat{p},\hat{q}}(\alpha,\beta|\pi_{XY},\mu_{X},\mu_{Y})=\infty$.
Otherwise, we can denote the conditional distribution $\pi_{XY|A}$
of $\pi_{XY}$ on $A\times\mathcal{Y}$ as 
\[
\pi_{XY|A}:B\in\mathbb{B}_{\mathcal{X}}\otimes\mathbb{B}_{\mathcal{Y}}\mapsto\frac{\pi_{XY}(B\cap(A\times\mathcal{Y}))}{\pi_{X}(A)}.
\]
Obviously, the marginal of $\pi_{XY|A}$ on $\mathcal{X}$ is $\pi_{X|A}:B\in\mathbb{B}_{\mathcal{X}}\mapsto\frac{\pi_{X}(B\cap A)}{\pi_{X}(A)}=\frac{\bar{\pi}_{X}(B\cap A)}{\bar{\pi}_{X}(A)}$.
Hence, $\pi_{X|A}\ll\mu_{X}$. On the other hand, 
\[
\overline{\Lambda}_{p,q,\hat{p},\hat{q}}(\alpha,\beta|\pi_{XY},\mu_{X},\mu_{Y})=\overline{\Lambda}_{p,q,\hat{p},\hat{q}}(\alpha,\beta|\pi_{XY|A},\mu_{X},\mu_{Y})-\log\pi_{X}(A),
\]
which implies that determining $\overline{\Lambda}_{p,q,\hat{p},\hat{q}}$
for $(\pi_{XY},\mu_{X},\mu_{Y})$ is equivalent to determining $\overline{\Lambda}_{p,q,\hat{p},\hat{q}}$
for $(\pi_{XY|A},\mu_{X},\mu_{Y})$. Hence, $\pi_{X}\ll\mu_{X},\pi_{Y}\ll\mu_{Y}$
can be assumed. 

By checking our proofs, one can verify that for the case of $\pi_{X}\ll\mu_{X},\pi_{Y}\ll\mu_{Y}$,
all the inequalities derived in Sections \ref{sec:Strong-BL-and}-\ref{sec:Applications}
still hold for this general setting, but they require to be slightly
modified. For this general case, $\Ent_{p,\hat{p}}(f)=D_{\hat{p}/p}(Q_{X}\|\mu_{X})$
still holds, but the Rényi divergence here is required to extend to
the one of a probability from a $\sigma$-finite nonnegative measure.
More specifically, the Rényi divergence of a probability measure $Q$
from a $\sigma$-finite nonnegative measure $\mu$ on the same measurable
space is still defined as $D_{s}(Q\|\mu):=\frac{1}{s-1}\log\int(\frac{\mathrm{d}Q}{\mathrm{d}\nu})^{s}(\frac{\mathrm{d}\mu}{\mathrm{d}\nu})^{1-s}\mathrm{d}\nu$
for $s\in\mathbb{R}\backslash\{0,1\}$ and a $\sigma$-finite nonnegative
measure $\nu$ such that $Q,\mu\ll\nu$  (in fact, the value of $D_{s}(Q\|\mu)$
is independent of the choice of $\nu$). For $s\in\{0,1,\pm\infty\}$,
the Rényi divergence of order $s$ is still defined by the continuous
extension. In particular, the formula for the relative entropy (i.e.,
the case $s=1$) can be found in Remark \ref{rem:The-probability-measures}
in Appendix \ref{sec:Proof-of-Proposition-ITChara}. In fact, this
general version of the Rényi divergence  can be seen as the negative
Rényi entropy of $Q$ with respect to the reference measure $\mu$.
It still admits the following variational formula (see e.g. Lemma
\ref{lem:dual} in Appendix \ref{sec:Proof-of-Proposition-ITChara}):
\begin{equation}
D_{s}(Q\|\mu)=\frac{1}{1-s}\inf_{R}\{sD(R\|Q)+(1-s)D(R\|\mu)\},\label{eq:-51-2}
\end{equation}
where the infimum is still taken over probability measures $R$. By
this formula, given $Q,\mu$, the Rényi divergence $D_{s}(Q\|\mu)$
is still nondecreasing in its order $s\in[-\infty,\infty]$. Similarly
to Lemma \ref{lem:norm}, one can also prove the continuity of $D_{s}(Q\|\mu)$
in $s\in\{s:|D_{s}(Q\|\mu)|<\infty\}$.  However, unlike the case
of probability measure $\mu$, in general, for a $\sigma$-finite
nonnegative measure $\mu$, $D_{s}(Q\|\mu)$ is \emph{not} nonnegative
anymore for $s\ge0$, and also \emph{not} nonpositive anymore for
$s<0$.  Based on this general definitions of the Rényi divergence
and the relative entropy, all the inequalities derived in Sections
\ref{sec:Strong-BL-and}-\ref{sec:Applications} still hold if we
replace $(P_{XY},P_{X},P_{Y})$ with the tuple $(\pi_{XY},\mu_{X},\mu_{Y})$.
Note that all the optimizations involved in these results are still
taken over probability measures. Furthermore, it should be noted that
for the single-function version of BL inequalities and the theorems
on the $q$-stability, we should additionally assume $\mu_{Y}=\pi_{Y}$,
so that the equivalences in \eqref{eq:-38} and \eqref{eq:-39} are
still valid for the new setting. 

In addition, the proofs of the exponential tightness of the inequalities
derived in Sections \ref{sec:Strong-BL-and}-\ref{sec:Applications}
requires the large deviation theory, more specifically, Sanov's theorem.
If $\pi_{XY},\mu_{X},\mu_{Y}$ are finite measures, then we can normalize
them to probability measures. In this case, Sanov's theorem for the
probability measures still works. However, if $\pi_{XY},\mu_{X},\mu_{Y}$
are infinite (but $\sigma$-finite) measures, our proofs require a
version of Sanov's theorem for infinite measures. In fact, the generalization
of Sanov's theorem to infinite measures was already investigated in
\cite{bakhtin2016kullback}, but for a so-called \emph{fine topology},
instead of the weak topology. The exponential tightness of the inequalities
derived in Sections \ref{sec:Strong-BL-and}-\ref{sec:Applications}
for infinite measures remains to be investigated in the future. 

\appendix

\section{\label{sec:Proof-of-Proposition-ITChara}Proof of Proposition \ref{prop:ITcharacterization}}

To prove Proposition \ref{prop:ITcharacterization}, we need the following
lemma. 
\begin{lem}
\label{lem:dual}Let $P_{i},i\in[n]$ be $n$ probability measures
on $\mathcal{X}$ and $\mu$ be a $\sigma$-finite nonnegative measure
on $\mathcal{X}$ such that all $P_{i}\ll\mu,i\in[n]$. Let $s_{i},i\in[n]$
be $n$ real numbers such that $\sum_{i=1}^{n}s_{i}=1$. Let $c:\mathcal{X}\to[-\infty,\infty]$
be a measurable function. Denote $\beta:=\int e^{-c}\prod_{i=1}^{n}(\frac{\mathrm{d}P_{i}}{\mathrm{d}\mu})^{s_{i}}\mathrm{d}\mu$.
Then we have 
\begin{equation}
-\log\beta=\inf_{Q}\sum_{i=1}^{n}s_{i}D(Q\|P_{i})+\int c\,\mathrm{d}Q,\label{eq:duality}
\end{equation}
where the conventions that $0\cdot\infty=0$, and $0^{s}=\infty$
for $s<0$, and $1$ for $s=0$ are adopted, according to Convention
\ref{Convention-:-When}, the infimization in the RHS of \eqref{eq:duality}
is over all probability measures $Q$ on $\mathcal{X}$ such that
$D(Q\|P_{i})<\infty,\forall i\in[n]$ and $\int|c|\,\mathrm{d}Q<\infty$.
Moreover, if $0<\beta<\infty$, $D(Q^{*}\|P_{i})<\infty,\forall i\in[n]$,
and $\int|c|\,\mathrm{d}Q^{*}<\infty$, where $Q^{*}$ is a probability
measure with the density 
\begin{equation}
\frac{\mathrm{d}Q^{*}}{\mathrm{d}\mu}=\frac{e^{-c}}{\beta}\prod_{i=1}^{n}(\frac{\mathrm{d}P_{i}}{\mathrm{d}\mu})^{s_{i}},\label{eq:-62}
\end{equation}
then the infimization in the RHS of \eqref{eq:duality} is uniquely
attained by $Q^{*}$. 
\end{lem}

\begin{rem}
In fact, from the definition, $\beta$ is independent of the reference
measure $\mu$, which can be also seen from \eqref{eq:duality}. 
\end{rem}

\begin{rem}
\label{rem:The-probability-measures}The probability measures $P_{i},i\in[n]$
can be relaxed to any $\sigma$-finite nonnegative measures. For this
case, the relative entropy of a probability measure $Q$ from a $\sigma$-finite
nonnegative measure $\pi$ on the same measurable space is still defined
as $D(Q\|\pi):=\int\log(\frac{\mathrm{d}Q}{\mathrm{d}\pi})\mathrm{d}Q$
if $Q\ll\pi$ and the integral exists, and infinity otherwise. Moreover,
the infimum in \eqref{eq:duality}  is still taken over probability
measures. 
\end{rem}

\begin{rem}
Special cases of this lemma were provided in \cite{van2010data,shayevitz2011renyi,Erven}.
The finite alphabet version can be proven by the Lagrange multiplier
method. The proof here is based on the non-negativity of the relative
entropy, which is essentially due to Csiszár and Matus \cite{csiszar2003information}
and also used in \cite{Erven,liu2018information}. 
\end{rem}

\begin{proof}
We first consider the case $0<\beta<\infty$. For $0<a<b<\infty$,
define 
\[
\mathcal{A}_{a,b}:=\Big\{ x:\frac{\mathrm{d}P_{i}}{\mathrm{d}\mu}(x),c(x)\in(a,b),\forall i\in[n]\Big\}.
\]
Observe that for any $Q\ll\mu$ concentrated on $\mathcal{A}_{a,b}$,
we have $D(Q\|P_{i})<\infty,\forall i\in[n]$, $\int|c|\,\mathrm{d}Q<\infty$,
and moreover, 
\begin{align*}
\sum_{i=1}^{n}s_{i}D(Q\|P_{i})+\int c\,\mathrm{d}Q+\log\beta_{a,b} & =D(Q\|Q_{a,b}^{*})\geq0,
\end{align*}
where $\beta_{a,b}:=\int_{\mathcal{A}_{a,b}}\prod_{i=1}^{n}(\frac{\mathrm{d}P_{i}}{\mathrm{d}\mu})^{s_{i}}\mathrm{d}\mu$,
and $Q_{a,b}^{*}$ is a probability measure with the density 
\[
\frac{\mathrm{d}Q_{a,b}^{*}}{\mathrm{d}\mu}:=\frac{1_{\mathcal{A}_{a,b}}e^{-c}\prod_{i=1}^{n}(\frac{\mathrm{d}P_{i}}{\mathrm{d}\mu})^{s_{i}}}{\beta_{a,b}}.
\]
The equality above holds if $Q=Q_{a,b}^{*}$. Hence, 
\[
-\log\beta_{a,b}=\inf_{Q:\,Q(\mathcal{A}_{a,b}^{c})=0}\sum_{i=1}^{n}s_{i}D(Q\|P_{i})+\int c\,\mathrm{d}Q,
\]
by Convention \ref{Convention-:-When}, the infimum is taken over
all $Q$ such that $Q(\mathcal{A}_{a,b}^{c})=0$, $D(Q\|P_{i})<\infty,\forall i\in[n],\int|c|\,\mathrm{d}Q<\infty$.
Taking infimization over $0<a<b<\infty$ for both sides above and
applying the monotone convergence theorem, we have 
\[
-\log\beta=\inf_{0<a<b<\infty}\inf_{Q:\,Q(\mathcal{A}_{a,b}^{c})=0}\sum_{i=1}^{n}s_{i}D(Q\|P_{i})+\int c\,\mathrm{d}Q,
\]
which implies 
\begin{equation}
-\log\beta\geq\inf_{Q}\sum_{i=1}^{n}s_{i}D(Q\|P_{i})+\int c\mathrm{d}Q.\label{eq:-54}
\end{equation}

On the other hand, for any $Q$ such that $D(Q\|P_{i})<\infty,\forall i\in[n]$
and $\int|c|\mathrm{d}Q<\infty$, we have 
\begin{align}
\sum_{i=1}^{n}s_{i}D(Q\|P_{i})+\int c\mathrm{d}Q+\log\beta & =D(Q\|Q^{*})\geq0,\label{eq:-56}
\end{align}
where $Q^{*}$ is defined in \eqref{eq:-62} and its existence follows
by the assumption $0<\beta<\infty$. It is easy to see that the equality
for the inequality above holds if $Q=Q^{*}$, $D(Q^{*}\|P_{i})<\infty,\forall i\in[n]$,
and $\int|c|\mathrm{d}Q^{*}<\infty$. Therefore, combining \eqref{eq:-54}
and \eqref{eq:-56} yields \eqref{eq:duality}.

We next consider the case $\beta=\infty$. For this case, it is easy
to check that \eqref{eq:-54} still holds. Hence, both sides of \eqref{eq:duality}
are $-\infty$.

We lastly consider the case $\beta=0$. For this case, $e^{-c}\prod_{i=1}^{n}(\frac{\mathrm{d}P_{i}}{\mathrm{d}\mu})^{s_{i}}=0$
holds $\mu$-a.e. Hence, for any $Q$, it holds that $\int|c|\mathrm{d}Q^{*}=\infty$
or there exists at least one $i\in[n]$ such that $s_{i}>0$ and $D(Q\|P_{i})=\infty$.
Hence, the infimization at the RHS of \eqref{eq:duality} is taken
over the empty set. Hence, both sides of \eqref{eq:duality} are $\infty$.
This completes the proof of Lemma \ref{lem:dual}. 
\end{proof}
We may assume, by homogeneilty, that $\Vert f\Vert_{p}=\Vert g\Vert_{q}=1$.
Observe that all the integrals involved in $\Vert f\Vert_{p},\Vert g\Vert_{q}$,
$\Ent_{p,\hat{p}}(f)$, $\Ent_{q,\hat{q}}(g),\langle f,g\rangle$
can be written as the ones with respect to $P_{XY}$. Hence, we can
choose $P_{XY}$ as a reference measure. Then, without loss of generality,
we can write 
\[
f^{p}=\frac{\mathrm{d}Q_{X}}{\mathrm{d}P_{X}}=\frac{\mathrm{d}(Q_{X}P_{Y|X})}{\mathrm{d}P_{XY}}\qquad g^{q}=\frac{\mathrm{d}Q_{Y}}{\mathrm{d}P_{Y}}=\frac{\mathrm{d}(Q_{Y}P_{X|Y})}{\mathrm{d}P_{XY}},
\]
for some probability measures $Q_{X}\ll P_{X},Q_{Y}\ll P_{Y}$. Moreover,
we require $f<\infty$. Hence, $Q_{X}\ll\gg P_{X}$ if $p<0$. Similarly,
$Q_{Y}\ll\gg P_{Y}$ if $q<0$.

By this choice of $f,g$, we have 
\begin{align}
\Ent_{p,\hat{p}}(f) & =-\frac{p}{p-\hat{p}}\log\int(\frac{\mathrm{d}Q_{X}}{\mathrm{d}P_{X}})^{\hat{p}/p}\mathrm{d}P_{X}=D_{\hat{p}/p}(Q_{X}\|P_{X}).\label{eq:-3}
\end{align}
To ensure that $D_{{\hat{p}}/{p}}(Q_{X}\|P_{X})$ is finite, for the
case $p>0>\hat{p}$, we still require $Q_{X}\ll\gg P_{X}$. Hence,
$Q_{X}\ll P_{X}$ if $p>0,\hat{p}\ge0$; otherwise, $Q_{X}\ll\gg P_{X}$.
Similarly, $Q_{Y}\ll P_{Y}$ if $q>0,\hat{q}\ge0$; otherwise, $Q_{Y}\ll\gg P_{Y}$.

In addition, we also have that 
\begin{align}
-\log\langle f,g\rangle & =-\log\int(\frac{\mathrm{d}(Q_{X}P_{Y|X})}{\mathrm{d}P_{XY}})^{1/p}(\frac{\mathrm{d}(Q_{Y}P_{X|Y})}{\mathrm{d}P_{XY}})^{1/q}\mathrm{d}P_{XY}\\
 & =\inf_{R_{XY}}\{D(R_{XY}\|P_{XY})+\frac{1}{p}D(R_{X}\|Q_{X})-\frac{1}{p}D(R_{X}\|P_{X})+\frac{1}{q}D(R_{Y}\|Q_{Y})-\frac{1}{q}D(R_{Y}\|P_{Y})\}\label{eq:-55}\\
 & =\phi(Q_{X},Q_{Y}|P_{XY}),\label{eq:-4}
\end{align}
where \eqref{eq:-55} follows by Lemma \ref{lem:dual}, and moreover,
in \eqref{eq:-55}, the infimization is taken over all $R_{XY}$ such
that all the relative entropies appearing in the objective function
are finite.

Substituting \eqref{eq:-3} and \eqref{eq:-4} into the definitions
of $\underline{\Lambda}_{p,q,\hat{p},\hat{q}}(\alpha,\beta)$ and
$\overline{\Lambda}_{p,q,\hat{p},\hat{q}}(\alpha,\beta)$ yields Theorem
\ref{prop:ITcharacterization}.

\section{\label{sec:Proof-of-Theorem-BLexponent}Proof of Theorem \ref{thm:BLexponent}}

\subsection{Forward Case}

Our proof consists of two parts: one-shot bound and singleletterization.

\subsubsection{One-shot Bound}

We first consider the case $n=1$. Denote $\lambda=\hat{p}/p,\mu=\hat{q}/q$.
Define 
\begin{align*}
\psi(R_{XY}) & =D(R_{XY}\|P_{XY})-\frac{1}{p}D(R_{X}\|P_{X})-\frac{1}{q}D(R_{Y}\|P_{Y})\\
 & \qquad+\inf_{\substack{Q_{X}:D_{\lambda}(Q_{X}\|P_{X})=\alpha}
}\frac{1}{p}D(R_{X}\|Q_{X})+\inf_{\substack{Q_{Y}:D_{\mu}(Q_{Y}\|P_{Y})=\beta}
}\frac{1}{q}D(R_{Y}\|Q_{Y}).
\end{align*}
Then $\underline{\Lambda}_{p,q,\hat{p},\hat{q}}(\alpha,\beta)=\inf_{R_{XY}}\psi(R_{XY})+\frac{\alpha}{p}+\frac{\beta}{q}.$
We will prove 
\begin{align}
\xi:=\frac{1}{p}(\alpha-D(R_{X}\|P_{X}))+\inf_{\substack{Q_{X}:D_{\lambda}(Q_{X}\|P_{X})=\alpha}
}\frac{1}{p}D(R_{X}\|Q_{X}) & \ge\eta_{p,\hat{p}}(\alpha,D(R_{X}\|P_{X})),\label{eq:-46}
\end{align}
which, by symmetry, implies that a similar inequality holds for $P_{Y}$,
and further implies that 
\begin{align}
\underline{\Lambda}_{p,q,\hat{p},\hat{q}}(\alpha,\beta) & \geq\inf_{R_{XY}}D(R_{XY}\|P_{XY})+\eta_{p,\hat{p}}(\alpha,D(R_{X}\|P_{X}))+\eta_{q,\hat{q}}(\beta,D(R_{Y}\|P_{Y})).\label{eq:-52-3}
\end{align}

I. We first assume $p>0,\hat{p}\ge0$. 

1) We first consider the case $\lambda\in(0,\infty]\backslash\{1\}$.
By Lemma \ref{lem:dual}, 
\begin{equation}
D_{\lambda}(Q_{X}\|P_{X})=\frac{1}{1-\lambda}\inf_{R_{X}}\{\lambda D(R_{X}\|Q_{X})+(1-\lambda)D(R_{X}\|P_{X})\}.\label{eq:-51}
\end{equation}
Hence, for any $R_{X}$, 
\begin{equation}
\inf_{\substack{Q_{X}:D_{\lambda}(Q_{X}\|P_{X})=\alpha}
}\frac{1}{p}D(R_{X}\|Q_{X})\ge\frac{1-\lambda}{\hat{p}}(\alpha-D(R_{X}\|P_{X})).\label{eq:-10-2}
\end{equation}
Combining this with the nonnegtivity of relative entropies yields
that the left-hand side (LHS) above is lower bounded by $[\frac{1-\lambda}{\hat{p}}(\alpha-D(R_{X}\|P_{X}))]^{+}$,
where $[x]^{+}:=x\vee0$. Hence, 
\begin{align}
\xi & \ge\frac{1}{p}(\alpha-D(R_{X}\|P_{X}))+[\frac{1-\lambda}{\hat{p}}(\alpha-D(R_{X}\|P_{X}))]^{+}=\eta_{p,\hat{p}}(\alpha,D(R_{X}\|P_{X})).\label{eq:-52}
\end{align}

2) We next consider the case $\lambda=0$. Observe that for any feasible
$(R_{XY},Q_{X})$, we have $R_{X}\ll Q_{X}\ll P_{X}$ since, otherwise,
$\inf_{\substack{Q_{X}:D_{0}(Q_{X}\|P_{X})=\alpha}
}\frac{1}{p}D(R_{X}\|Q_{X})=\infty$. This means, $\alpha=D_{0}(Q_{X}\|P_{X})\le D_{0}(R_{X}\|P_{X})\le D(R_{X}\|P_{X})$.
Therefore, by the nonnegtivity of relative entropies, we have 
\begin{align*}
\xi & \geq\begin{cases}
\frac{1}{p}(\alpha-D(R_{X}\|P_{X})), & D(R_{X}\|P_{X})\geq\alpha\\
\infty, & D(R_{X}\|P_{X})<\alpha
\end{cases}\;=\eta_{p,\hat{p}}(\alpha,D(R_{X}\|P_{X})).
\end{align*}

3) We next consider the case $\lambda=1$. By the nonnegtivity of
relative entropies, we have $\xi\geq\frac{1}{p}(\alpha-D(R_{X}\|P_{X}))=\eta_{p,\hat{p}}(\alpha,D(R_{X}\|P_{X})).$

II. For $p<0<\hat{p}$, \eqref{eq:-10-2} still holds. Hence, \eqref{eq:-46}
holds. 

III.  For $\hat{p}<0<p$, by the nonnegtivity of relative entropies,
we have $\xi\ge\frac{1}{p}(\alpha-D(R_{X}\|P_{X}))=\eta_{p,\hat{p}}(\alpha,D(R_{X}\|P_{X})).$ 

IV. For $p<0,\hat{p}\le0$, \eqref{eq:-46} holds trivially since
the RHS is equal to $-\infty$. 

\subsubsection{Singleletterization}

We next consider the singleletterization. Substituting $(\alpha,\beta,P_{XY})\leftarrow(n\alpha,n\beta,P_{XY}^{\otimes n}),$
we obtain the $n$-dimensional version: 
\begin{align}
\underline{\Lambda}_{p,q,\hat{p},\hat{q}}^{(n)}(\alpha,\beta) & \geq\frac{1}{n}\inf_{R_{X^{n}Y^{n}}}D(R_{X^{n}Y^{n}}\|P_{XY}^{\otimes n})+\eta_{p,\hat{p}}(n\alpha,D(R_{X^{n}}\|P_{X}^{\otimes n}))\nonumber \\
 & \qquad+\eta_{q,\hat{q}}(n\beta,D(R_{Y^{n}}\|P_{Y}^{\otimes n})).\label{eq:-13}
\end{align}
By the chain rule, we have 
\begin{align}
D(R_{X^{n}Y^{n}}\|P_{XY}^{\otimes n}) & =\sum_{i=1}^{n}D(R_{X_{i}Y_{i}|X^{i-1}Y^{i-1}}\|P_{XY}|R_{X^{i-1}Y^{i-1}})\label{eq:-32}\\
D(R_{X^{n}}\|P_{X}^{\otimes n}) & =\sum_{i=1}^{n}D(R_{X_{i}|X^{i-1}}\|P_{X}|R_{X^{i-1}})\nonumber \\
D(R_{Y^{n}}\|P_{Y}^{\otimes n}) & =\sum_{i=1}^{n}D(R_{Y_{i}|Y^{i-1}}\|P_{Y}|R_{Y^{i-1}}).\nonumber 
\end{align}
Consider the joint distribution $R_{X^{n}Y^{n}K}:=R_{X^{n}Y^{n}}\otimes\mathrm{Unif}[n]$,
i.e., under this distribution, $K\sim\mathrm{Unif}[n]$ (called a
\emph{random index}) is independent of $(X^{n},Y^{n})$. Denote $X:=X_{K},Y:=Y_{K},U:=(X^{K-1},K),V:=(Y^{K-1},K),W:=(U,V)$.
Then 
\begin{align*}
D(R_{X^{n}Y^{n}}\|P_{XY}^{\otimes n}) & =nD(R_{XY|W}\|P_{XY}|R_{W})\\
D(R_{X^{n}}\|P_{X}^{\otimes n}) & =nD(R_{X|U}\|P_{X}|R_{U})\leq nD(R_{X|W}\|P_{X}|R_{W})\\
D(R_{Y^{n}}\|P_{Y}^{\otimes n}) & =nD(R_{Y|V}\|P_{Y}|R_{V})\leq nD(R_{Y|W}\|P_{Y}|R_{W}),
\end{align*}
where the inequalities in the last two lines follow by the fact that
conditioning increasing the relative entropy. Substituting these into
\eqref{eq:-13} and utilizing the monotonicity of $\eta_{p,\hat{p}}(\alpha,s)$
in $s$ yields that 
\begin{align}
\underline{\Lambda}_{p,q,\hat{p},\hat{q}}^{(n)}(\alpha,\beta) & \ge\inf_{R_{XYW}}D(R_{XY|W}\|P_{XY}|R_{W})\nonumber \\
 & \qquad+\eta_{p,\hat{p}}(\alpha,D(R_{X|W}\|P_{X}|R_{W}))+\eta_{q,\hat{q}}(\beta,D(R_{Y|W}\|P_{Y}|R_{W}))\nonumber \\
 & =\underline{\Lambda}_{p,q,\hat{p},\hat{q}}^{*}(\alpha,\beta).\label{eq:-34}
\end{align}

\subsection{\label{subsec:overline}Reverse Case}

Define $\overline{\Lambda}_{p,q,\hat{p},\hat{q}}^{**}(\alpha,\beta)$
as a variant of $\overline{\Lambda}_{p,q,\hat{p},\hat{q}}^{*}(\alpha,\beta)$,
by removing $\sup_{Q_{W}}$ from the definition of $\overline{\Lambda}_{p,q,\hat{p},\hat{q}}^{*}(\alpha,\beta)$,
and moreover, taking $Q_{W}$ to be a Dirac measure. In the following,
we first consider the case $n=1$, and prove 
\begin{equation}
\overline{\Lambda}_{p,q,\hat{p},\hat{q}}(\alpha,\beta)\le\overline{\Lambda}_{p,q,\hat{p},\hat{q}}^{**}(\alpha,\beta).\label{eq:-18}
\end{equation}
We then prove the $n$-dimensional version 
\begin{equation}
\overline{\Lambda}_{p,q,\hat{p},\hat{q}}^{**}(n\alpha,n\beta|P_{XY}^{\otimes n})\le n\overline{\Lambda}_{p,q,\hat{p},\hat{q}}^{*}(\alpha,\beta).\label{eq:-19}
\end{equation}
These yield the desired result 
\begin{equation}
\overline{\Lambda}_{p,q,\hat{p},\hat{q}}^{(n)}(\alpha,\beta)\le\overline{\Lambda}_{p,q,\hat{p},\hat{q}}^{*}(\alpha,\beta).\label{eq:-26}
\end{equation}

Denote $\lambda=\hat{p}/p,\mu=\hat{q}/q$. Define 
\begin{align*}
\psi(Q_{X},Q_{Y},R_{X},R_{Y}) & =\mathbb{D}(R_{X},R_{Y}\|P_{XY})-\frac{1}{p}D(R_{X}\|P_{X})-\frac{1}{q}D(R_{Y}\|P_{Y})\\
 & \qquad+\frac{1}{p}D(R_{X}\|Q_{X})+\frac{1}{q}D(R_{Y}\|Q_{Y})+\frac{\alpha}{p}+\frac{\beta}{q}.
\end{align*}
Then, $\phi(Q_{X},Q_{Y})+\frac{\alpha}{p}+\frac{\beta}{q}=\inf_{R_{X},R_{Y}}\psi(Q_{X},Q_{Y},R_{X},R_{Y}).$

\subsubsection{\label{subsec:One-shot-Bound}One-shot Bound}

We now prove \eqref{eq:-18}. We divide our proof into several parts
according to different cases: 

I. $p,q>0,\hat{p},\hat{q}\ge0$: 1) $\lambda,\mu\in(1,\infty)$ or
$\lambda,\mu\in(0,1)$; 2) $\lambda=\mu=0$; 3) $\lambda=\mu=1$;
4) $\lambda=\mu=\infty$.

II. $p,q<0,\hat{p},\hat{q}\ge0$: 1) $p,q<0<\hat{p},\hat{q}\le\infty$;
2) $p,q<0=\hat{p}=\hat{q}$.

III. $p,q>0,\hat{p},\hat{q}<0$: 1) $p,q>0>\hat{p},\hat{q}>-\infty$;
2) $p,q>0,\hat{p},\hat{q}=-\infty$.

IV. $p,q,\hat{p},\hat{q}<0$.

V. Other cases. 

The cases I-IV are ``symmetric'' with respect to $(p,q)$ and also
with respect to $(\hat{p},\hat{q})$. For brevity, we only provide
the proof for these ``symmetric'' cases. For the ``asymmetric''
case (i.e., the case V), one can prove \eqref{eq:-18} by ``mixing''
the proofs for ``symmetric'' cases, since the ``asymmetric'' case
can be seen as a mixture of the ``symmetric'' cases. 

I. We first assume $p,q>0,\hat{p},\hat{q}\ge0$, and prove \eqref{eq:-18}
for this case. 

1) We first consider the case $\lambda,\mu\in(1,\infty)$. Define
$R_{X}^{*}$ as the distribution with density 
\begin{equation}
\frac{\mathrm{d}R_{X}^{*}}{\mathrm{d}P_{X}}=\frac{(\frac{\mathrm{d}Q_{X}}{\mathrm{d}P_{X}})^{\lambda}}{\int(\frac{\mathrm{d}Q_{X}}{\mathrm{d}P_{X}})^{\lambda}\mathrm{d}P_{X}}.\label{eq:-R*}
\end{equation}
The existence of $R_{X}^{*}$ follows by the requirement $D_{\lambda}(Q_{X}\|P_{X})=\alpha<\infty$.
Define $R_{Y}^{*}$ similarly. We first assume that $D(R_{X}^{*}\|P_{X}),D(R_{Y}^{*}\|Q_{Y})<\infty$,
and later will consider the case $D(R_{X}^{*}\|P_{X})=\infty$ or
$D(R_{Y}^{*}\|Q_{Y})=\infty$. It is easy to check that for $D(R_{X}^{*}\|P_{X})<\infty$,
we have 
\begin{equation}
D(R_{X}^{*}\|P_{X})+\frac{\lambda}{1-\lambda}D(R_{X}^{*}\|Q_{X})=D_{\lambda}(Q_{X}\|P_{X}),\label{eq:-58}
\end{equation}
which, combined with the monotonicity of the Rényi divergence in its
order, implies that for $\lambda>1$, 
\begin{equation}
D(Q_{X}\|P_{X})\leq D_{\lambda}(Q_{X}\|P_{X})\le D(R_{X}^{*}\|P_{X}).\label{eq:-22}
\end{equation}
Denote $s=D(Q_{X}\|P_{X}),\hat{s}=D(R_{X}^{*}\|P_{X})$ and $t=D(Q_{Y}\|P_{Y}),\hat{t}=D(R_{Y}^{*}\|P_{Y})$.
Then under the constraints $D_{\lambda}(Q_{X}\|P_{X})=\alpha,D_{\mu}(Q_{Y}\|P_{Y})=\beta$,
we have 
\begin{equation}
s\le\alpha\le\hat{s},\quad t\leq\beta\le\hat{t}.\label{eq:-35}
\end{equation}
Hence, we have 
\begin{align}
\phi(Q_{X},Q_{Y})+\frac{\alpha}{p}+\frac{\beta}{q} & \le\min_{R_{X}\in\{R_{X}^{*},Q_{X}\},R_{Y}\in\{R_{Y}^{*},Q_{Y}\}}\psi(Q_{X},Q_{Y},R_{X},R_{Y})\label{eq:-23-1}\\
 & \leq\hat{\phi}(Q_{X},Q_{Y},R_{X}^{*},R_{Y}^{*}),\label{eq:-20-1}
\end{align}
where 
\begin{align*}
\hat{\phi}(Q_{X},Q_{Y},R_{X},R_{Y}) & :=\min\biggl\{\mathbb{D}(R_{X},R_{Y}\|P_{XY})+\frac{\alpha-\hat{s}}{\hat{p}}+\frac{\beta-\hat{t}}{\hat{q}},\;\mathbb{D}(R_{X},Q_{Y}\|P_{XY})+\frac{\alpha-\hat{s}}{\hat{p}}+\frac{\beta-t}{q},\\
 & \qquad\mathbb{D}(Q_{X},R_{Y}\|P_{XY})+\frac{\alpha-s}{p}+\frac{\beta-\hat{t}}{\hat{q}},\:\mathbb{D}(Q_{X},Q_{Y}\|P_{XY})+\frac{\alpha-s}{p}+\frac{\beta-t}{q}\biggr\},
\end{align*}
and in \eqref{eq:-21}, \eqref{eq:-58} was used to eliminate the
terms $D(R_{X}^{*}\|Q_{X}),D(R_{Y}^{*}\|Q_{Y})$. Then, relaxing $Q_{X},Q_{Y},R_{X}^{*},R_{Y}^{*}$
to arbitrary distributions satisfying \eqref{eq:-35}, we obtain 
\begin{align}
\overline{\Lambda}_{p,q,\hat{p},\hat{q}}(\alpha,\beta) & \le\sup_{\substack{s\le\alpha\le\hat{s},\,t\leq\beta\le\hat{t}}
}\sup_{\substack{Q_{X},Q_{Y},R_{X},R_{Y}:\\
D(Q_{X}\|P_{X})=s,D(Q_{Y}\|P_{Y})=t\\
D(R_{X}\|P_{X})=\hat{s},D(R_{Y}\|P_{Y})=\hat{t}
}
}\hat{\phi}(Q_{X},Q_{Y},R_{X},R_{Y})=\overline{\Lambda}_{p,q,\hat{p},\hat{q}}^{**}(\alpha,\beta).\label{eq:-59}
\end{align}

We now consider the case $D(R_{X}^{*}\|P_{X})=\infty$ or $D(R_{Y}^{*}\|Q_{Y})=\infty$.
Here we only provide a proof for the case of $D(R_{X}^{*}\|P_{X})=D(R_{Y}^{*}\|Q_{Y})=\infty$.
The case of $D(R_{X}^{*}\|P_{X})<D(R_{Y}^{*}\|Q_{Y})=\infty$ and
the case of $D(R_{Y}^{*}\|Q_{Y})<D(R_{X}^{*}\|P_{X})=\infty$ can
be proven similarly. If $D(R_{X}^{*}\|P_{X})=\infty$, we replace
$R_{X}^{*}$ above as the distribution $R_{X}^{(r)}$ with density
w.r.t. $P_{X}$ being 
\begin{equation}
\frac{\mathrm{d}R_{X}^{(r)}}{\mathrm{d}P_{X}}=\frac{(\frac{\mathrm{d}Q_{X}}{\mathrm{d}P_{X}})^{\lambda}1_{\mathcal{A}_{r}}}{\int(\frac{\mathrm{d}Q_{X}}{\mathrm{d}P_{X}})^{\lambda}1_{\mathcal{A}_{r}}\mathrm{d}P_{X}}\label{eq:-Rr}
\end{equation}
where $\mathcal{A}_{r}:=\{x:\frac{\mathrm{d}Q_{X}}{\mathrm{d}P_{X}}(x)<r\}$
for $r>0$. Then it is easy to verify that $D(R_{X}^{(r)}\|P_{X})<\infty$
and $D(R_{X}^{(r)}\|P_{X})\to D(R_{X}^{*}\|P_{X})=\infty$ as $r\to\infty$.
Moreover, analogue to \eqref{eq:-58}, we have 
\begin{equation}
(1-\lambda)D(R_{X}^{(r)}\|P_{X})+\lambda D(R_{X}^{(r)}\|Q_{X})=(1-\lambda)D_{\lambda}(Q_{X}\|P_{X})+\varepsilon_{r},\label{eq:-21}
\end{equation}
where $\varepsilon_{r}\ge0$ vanishes as $r\to\infty$. The equality
above implies that $D(R_{X}^{(r)}\|Q_{X})\to\infty$ as $r\to\infty$,
under the constraint $D_{\lambda}(Q_{X}\|P_{X})=\alpha$. Hence, for
sufficiently large $r$, $D(R_{X}^{(r)}\|P_{X})\ge\alpha$. By redefining
$\hat{s}=D(R_{X}^{(r)}\|P_{X})$, we have that \eqref{eq:-35} still
holds. Using \eqref{eq:-21} to replace \eqref{eq:-58}, we still
have \eqref{eq:-20-1}, but with $\hat{\phi}$ replaced by the following
$\hat{\phi}_{r}$. 
\begin{align*}
\hat{\phi}_{r}(Q_{X},Q_{Y},R_{X},R_{Y}) & :=\min\biggl\{\mathbb{D}(R_{X},R_{Y}\|P_{XY})+\frac{\alpha-\hat{s}+\varepsilon_{r}}{\hat{p}}+\frac{\beta-\hat{t}+\varepsilon_{r}}{\hat{q}},\\
 & \qquad\mathbb{D}(R_{X},Q_{Y}\|P_{XY})+\frac{\alpha-\hat{s}+\varepsilon_{r}}{\hat{p}}+\frac{\beta-t}{q},\\
 & \qquad\mathbb{D}(Q_{X},R_{Y}\|P_{XY})+\frac{\alpha-s}{p}+\frac{\beta-\hat{t}+\varepsilon_{r}}{\hat{q}},\\
 & \qquad\mathbb{D}(Q_{X},Q_{Y}\|P_{XY})+\frac{\alpha-s}{p}+\frac{\beta-t}{q}\biggr\}.
\end{align*}
Then, letting $r\to\infty$, we have \eqref{eq:-59}. This completes
the proof of \eqref{eq:-18} for the case $\lambda,\mu\in(1,\infty)$.

Since $\overline{\Lambda}_{p,q,\hat{p},\hat{q}}(\alpha,\beta)$ is
symmetric w.r.t. $p,\hat{p}$ and also $q,\hat{q}$, \eqref{eq:-18}
holds for the case $\lambda,\mu\in(0,1)$.

2) We next consider the case $\lambda=\mu=0$. By the monotonicity
of the Rényi divergence in its order, $D(Q_{X}\|P_{X})\ge D_{0}(Q_{X}\|P_{X}).$
Denote $s=D(Q_{X}\|P_{X})$ and $t=D(Q_{Y}\|P_{Y})$. Then under the
constraints $D_{0}(Q_{X}\|P_{X})=\alpha,D_{0}(Q_{Y}\|P_{Y})=\beta$,
we have $s\ge\alpha,t\ge\beta.$ Hence, we have 
\begin{align}
 & \phi(Q_{X},Q_{Y})+\frac{\alpha}{p}+\frac{\beta}{q}\le\mathbb{D}(Q_{X},Q_{Y}\|P_{XY})+\frac{\alpha-s}{p}+\frac{\beta-t}{q}.\label{eq:-23-1-3}
\end{align}
Then, relaxing $Q_{X},Q_{Y}$ to be arbitrary distributions satisfying
$s\ge\alpha,t\ge\beta$, we obtain 
\begin{align}
\overline{\Lambda}_{p,q,\hat{p},\hat{q}}(\alpha,\beta) & \le\sup_{s\ge\alpha,t\ge\beta}\sup_{\substack{Q_{X},Q_{Y}:\\
D(Q_{X}\|P_{X})=s,D(Q_{Y}\|P_{Y})=t
}
}\mathbb{D}(Q_{X},Q_{Y}\|P_{XY})+\frac{\alpha-s}{p}+\frac{\beta-t}{q}\label{eq:-59-4}\\
 & =\overline{\Lambda}_{p,q,\hat{p},\hat{q}}^{**}(\alpha,\beta).
\end{align}

3) We next consider the case $\lambda=\mu=1$. For this case, \eqref{eq:-23-1-3}
still holds. 
Then, we obtain 
\begin{align}
\overline{\Lambda}_{p,q,\hat{p},\hat{q}}(\alpha,\beta) & \le\sup_{\substack{Q_{X},Q_{Y}:\\
D(Q_{X}\|P_{X})=\alpha,D(Q_{Y}\|P_{Y})=\beta
}
}\mathbb{D}(Q_{X},Q_{Y}\|P_{XY})+\frac{\alpha-s}{p}+\frac{\beta-t}{q}\label{eq:-59-2}\\
 & =\overline{\Lambda}_{p,q,\hat{p},\hat{q}}^{**}(\alpha,\beta).
\end{align}

4) We then consider the case $\lambda=\mu=\infty$. Suppose that $P_{X}(A)>0$
with $A:=\{x:\frac{\mathrm{d}Q_{X}}{\mathrm{d}P_{X}}(x)=\|\frac{\mathrm{d}Q_{X}}{\mathrm{d}P_{X}}\|_{\infty}\}$.
Otherwise, we can set $A_{\delta}:=\{x:\frac{\mathrm{d}Q_{X}}{\mathrm{d}P_{X}}(x)\ge\|\frac{\mathrm{d}Q_{X}}{\mathrm{d}P_{X}}\|_{\infty}-\delta\}$
for small $\delta>0$. 
Here, we only consider the case of $P_{X}(A)>0$. Define $R_{X}^{*}$
as the distribution with density $\frac{\mathrm{d}R_{X}^{*}}{\mathrm{d}P_{X}}=\frac{1_{A}}{P_{X}(A)}.$
Then 
\[
\frac{\mathrm{d}R_{X}^{*}}{\mathrm{d}Q_{X}}=\frac{\frac{1_{A}}{P_{X}(A)}}{\frac{\mathrm{d}Q_{X}}{\mathrm{d}P_{X}}}=\frac{1_{A}}{P_{X}(A)\|\frac{\mathrm{d}Q_{X}}{\mathrm{d}P_{X}}\|_{\infty}},\quad P_{X}\textrm{-almost everywhere}.
\]
Define $R_{Y}^{*}$ similarly. Observe that $D(R_{X}^{*}\|P_{X})=-\log P_{X}(A)$
and \eqref{eq:-58} still holds, i.e., 
\begin{align}
D(R_{X}^{*}\|Q_{X}) & =\int\frac{1_{A}}{P_{X}(A)}\log\frac{1_{A}}{P_{X}(A)\|\frac{\mathrm{d}Q_{X}}{\mathrm{d}P_{X}}\|_{\infty}}\mathrm{d}P_{X}=D(R_{X}^{*}\|P_{X})-D_{\infty}(Q_{X}\|P_{X}).\label{eq:-58-2}
\end{align}
Combined with the monotonicity of the Rényi divergence in its order
and also the non-negativity, \eqref{eq:-58-2} implies that for $\lambda>1$,
\begin{equation}
D(Q_{X}\|P_{X})\leq D_{\infty}(Q_{X}\|P_{X})\le D(R_{X}^{*}\|P_{X}).\label{eq:-22-4}
\end{equation}
Denote $s,t,\hat{s},\hat{t}$ same as the ones above \eqref{eq:-35}.
Then \eqref{eq:-35} and \eqref{eq:-20-1} still hold, but with $\hat{p},\hat{q}$
in the definition of $\hat{\phi}$ taking value of $\infty$. Furthermore,
\eqref{eq:-59} also holds, completing the proof for this case. 

II. We next consider the case $p,q<0,\hat{p},\hat{q}\ge0$. For this
case, we adopt a method similar to the above. For $p,q<0<\hat{p},\hat{q}\le\infty$,
we still define $R_{X}^{*},R_{Y}^{*}$ as above. We first assume that
$D(R_{X}^{*}\|P_{X}),D(R_{Y}^{*}\|Q_{Y})<\infty$. For this case,
similarly to the above, 
\begin{align}
 & \phi(Q_{X},Q_{Y})+\frac{\alpha}{p}+\frac{\beta}{q}\nonumber \\
 & \le\min\{\psi(Q_{X},Q_{Y},R_{X}^{*},R_{Y}^{*}),\inf_{R_{Y}}\psi(Q_{X},Q_{Y},R_{X}^{*},R_{Y}),\inf_{R_{X}}\psi(Q_{X},Q_{Y},R_{X},R_{Y}^{*}),\inf_{R_{X},R_{Y}}\psi(Q_{X},Q_{Y},R_{X},R_{Y})\}\nonumber \\
 & =\hat{\phi}(R_{X}^{*},R_{Y}^{*}),\label{eq:-16}
\end{align}
where 
\begin{align*}
\hat{\phi}(S_{X},S_{Y}) & :=\inf_{R_{X},R_{Y}}\min\biggl\{\mathbb{D}(S_{X},S_{Y}\|P_{XY})+\frac{\alpha-\hat{s}}{\hat{p}}+\frac{\beta-\hat{t}}{\hat{q}},\\
 & \qquad\mathbb{D}(S_{X},R_{Y}\|P_{XY})+\frac{\alpha-\hat{s}}{\hat{p}}+\frac{\beta-D(R_{Y}\|P_{Y})}{q},\\
 & \qquad\mathbb{D}(R_{X},S_{Y}\|P_{XY})+\frac{\alpha-D(R_{X}\|P_{X})}{p}+\frac{\beta-\hat{t}}{\hat{q}},\\
 & \qquad\mathbb{D}(R_{X},R_{Y}\|P_{XY})+\frac{\alpha-D(R_{X}\|P_{X})}{p}+\frac{\beta-D(R_{Y}\|P_{Y})}{q}\biggr\}.
\end{align*}
Then, $\overline{\Lambda}_{p,q,\hat{p},\hat{q}}(\alpha,\beta)\le\sup_{S_{X},S_{Y}}\hat{\phi}(S_{X},S_{Y})=\overline{\Lambda}_{p,q,\hat{p},\hat{q}}^{**}(\alpha,\beta),$
i.e., \eqref{eq:-18}. The cases of $D(R_{X}^{*}\|P_{X})=\infty$
or $D(R_{Y}^{*}\|Q_{Y})=\infty$ can be proven similarly as the above.
We omit the proofs for these cases.

For $p,q<0=\hat{p}=\hat{q}$, \eqref{eq:-23-1-3} still holds. Hence,
\begin{align*}
\overline{\Lambda}_{p,q,\hat{p},\hat{q}}(\alpha,\beta) & \le\sup_{Q_{X},Q_{Y}}\mathbb{D}(Q_{X},Q_{Y}\|P_{XY})+\frac{\alpha-s}{p}+\frac{\beta-t}{q}=\overline{\Lambda}_{p,q,\hat{p},\hat{q}}^{**}(\alpha,\beta).
\end{align*}

III. We then consider the case $p,q>0,\hat{p},\hat{q}<0$. For $p,q>0>\hat{p},\hat{q}>-\infty$,
by symmetry, \eqref{eq:-18} also holds. For $p,q>0,\hat{p},\hat{q}=-\infty$,
we have 
\begin{equation}
D_{-\infty}(Q_{X}\|P_{X})=\inf_{R_{X}}\{D(R_{X}\|P_{X})-D(R_{X}\|Q_{X})\}.\label{eq:-51-1}
\end{equation}
which implies that for any $R_{X},R_{Y}$, 
\begin{align*}
\phi(Q_{X},Q_{Y})+\frac{\alpha}{p}+\frac{\beta}{q} & \leq\min\biggl\{\mathbb{D}(R_{X},R_{Y}\|P_{XY}),\\
 & \qquad\mathbb{D}(R_{X},Q_{Y}\|P_{XY})+\frac{\beta-D(Q_{Y}\|P_{Y})}{q},\\
 & \qquad\mathbb{D}(Q_{X},R_{Y}\|P_{XY})+\frac{\alpha-D(Q_{X}\|P_{X})}{p},\\
 & \qquad\mathbb{D}(Q_{X},Q_{Y}\|P_{XY})+\frac{\alpha-D(Q_{X}\|P_{X})}{p}+\frac{\beta-D(Q_{Y}\|P_{Y})}{q}\biggr\}.
\end{align*}
Taking $\inf_{R_{X},R_{Y}}$ and then taking $\sup_{Q_{X},Q_{Y}}$,
we obtain \eqref{eq:-18}. 

IV. We then consider the case $p,q,\hat{p},\hat{q}<0$. For this case,
\begin{align}
\phi(Q_{X},Q_{Y})+\frac{\alpha}{p}+\frac{\beta}{q} & \le\inf_{R_{X},R_{Y},\text{\ensuremath{Q_{X}},\ensuremath{Q_{Y}}}}\min\biggl\{\mathbb{D}(R_{X},R_{Y}\|P_{XY})+\frac{\alpha-D(R_{X}\|P_{X})}{\hat{p}}+\frac{\beta-D(R_{Y}\|P_{Y})}{\hat{q}},\label{eq:-23-1-1}\\
 & \qquad\mathbb{D}(R_{X},Q_{Y}\|P_{XY})+\frac{\alpha-D(R_{X}\|P_{X})}{\hat{p}}+\frac{\beta-D(Q_{Y}\|Q_{Y})}{q},\\
 & \qquad\mathbb{D}(Q_{X},R_{Y}\|P_{XY})+\frac{\alpha-D(Q_{X}\|P_{X})}{p}+\frac{\beta-D(R_{Y}\|P_{Y})}{\hat{q}},\\
 & \qquad\mathbb{D}(Q_{X},Q_{Y}\|P_{XY})+\frac{\alpha-D(Q_{X}\|P_{X})}{p}+\frac{\beta-D(Q_{Y}\|Q_{Y})}{q}\biggr\},\\
 & =\min\{\frac{\alpha}{p}+\frac{\beta}{q},\frac{\alpha}{p}+\frac{\beta}{\hat{q}},\frac{\alpha}{\hat{p}}+\frac{\beta}{q},\frac{\alpha}{\hat{p}}+\frac{\beta}{\hat{q}}\}=\overline{\Lambda}_{p,q,\hat{p},\hat{q}}^{**}(\alpha,\beta).
\end{align}
Therefore, $\overline{\Lambda}_{p,q,\hat{p},\hat{q}}(\alpha,\beta)\le\overline{\Lambda}_{p,q,\hat{p},\hat{q}}^{**}(\alpha,\beta).$

\subsubsection{Singleletterization}

Based on the inequality \eqref{eq:-18}, in order to prove the desired
inequality \eqref{eq:-26}, it suffices to prove \eqref{eq:-19},
i.e., the following lemma. 
\begin{lem}
\label{lem:tensorization-1} For $p,q\in\overline{\mathbb{R}}\backslash\{0\},\hat{p},\hat{q}\in\overline{\mathbb{R}}$,
\begin{equation}
\overline{\Lambda}_{p,q,\hat{p},\hat{q}}^{**}(n\alpha,n\beta|P_{XY}^{\otimes n})\leq n\overline{\Lambda}_{p,q,\hat{p},\hat{q}}^{*}(\alpha,\beta|P_{XY}).\label{eq:-77}
\end{equation}
\end{lem}

We next prove Lemma \ref{lem:tensorization-1}. 

Similarly to the proof of the one-shot bound above, we only provide
the proofs for the ``symmetric'' cases. For the ``asymmetric''
case, one can prove Lemma \ref{lem:tensorization-1} by ``mixing''
the proofs for ``symmetric'' cases, since the ``asymmetric'' case
can be seen as a mixture of the ``symmetric'' cases. 

I. We now prove the case $p,q>0,\hat{p},\hat{q}\ge0$. We first consider
the case $p,q,\hat{p},\hat{q}>0$.  For any reals $b_{i,j}$ with
$i,j\in[2]$, define 
\begin{align*}
\eta(R_{X_{1}},R_{Y_{1}},R_{X_{2}},R_{Y_{2}}) & :=\min_{i,j\in[2]}\{\mathbb{D}(R_{X_{i}},R_{Y_{j}}\|P_{XY})+b_{i,j}\},
\end{align*}
and 
\begin{align}
\hat{\eta}(R_{X_{1}},R_{Y_{1}},R_{X_{2}},R_{Y_{2}}) & :=\inf_{R_{X_{1}Y_{1}X_{2}Y_{2}}\in\mathcal{C}(R_{X_{1}},R_{Y_{1}},R_{X_{2}},R_{Y_{2}})}\min_{i,j\in[2]}\{D(R_{X_{i}Y_{j}}\|P_{XY})+b_{i,j}\}.\label{eq:-14}
\end{align}
Then, we have the following lemma. 
\begin{lem}
\label{lem:eta}For any distributions $R_{X_{1}},R_{Y_{1}},R_{X_{2}},R_{Y_{2}}$,
\[
\eta(R_{X_{1}},R_{Y_{1}},R_{X_{2}},R_{Y_{2}})=\hat{\eta}(R_{X_{1}},R_{Y_{1}},R_{X_{2}},R_{Y_{2}}).
\]
\end{lem}

\begin{proof}
This lemma follows by swapping the infimization and minimization in
\eqref{eq:-14}. 
\end{proof}
By definition, 
\begin{align}
 & \frac{1}{n}\overline{\Lambda}_{p,q,\hat{p},\hat{q}}^{**}(n\alpha,n\beta|P_{XY}^{\otimes n})\nonumber \\
 & =\sup_{\substack{\alpha_{\lambda}^{-}\le s_{1}\le\alpha_{\lambda}^{+},\beta_{\mu}^{-}\le t_{1}\leq\beta_{\mu}^{+},\\
\alpha_{1/\lambda}^{-}\le s_{2}\le\alpha_{1/\lambda}^{+},\beta_{1/\mu}^{-}\le t_{2}\leq\beta_{1/\mu}^{+}
}
}\sup_{\substack{R_{X_{1}^{n}},R_{Y_{1}^{n}},R_{X_{2}^{n}},R_{Y_{2}^{n}}:\\
\frac{1}{n}D(R_{X_{i}^{n}}\|P_{X}^{\otimes n})=s_{i},\\
\frac{1}{n}D(R_{Y_{j}^{n}}\|P_{Y}^{\otimes n})=t_{j},\forall i,j\in[2]
}
}\chi(R_{X_{1}^{n}},R_{Y_{1}^{n}},R_{X_{2}^{n}},R_{Y_{2}^{n}}),\label{eq:-29}
\end{align}
where 
\[
\chi(R_{X_{1}^{n}},R_{Y_{1}^{n}},R_{X_{2}^{n}},R_{Y_{2}^{n}}):=\min_{i,j\in[2]}\{\frac{1}{n}\mathbb{D}(R_{X_{i}^{n}},R_{Y_{j}^{n}}\|P_{XY}^{\otimes n})+a_{i,j}\}
\]
and $a_{1,1}:=\frac{\alpha-s_{1}}{p}+\frac{\beta-t_{1}}{q},\,a_{1,2}:=\frac{\alpha-s_{1}}{p}+\frac{\beta-t_{2}}{\hat{q}},\,a_{2,1}:=\frac{\alpha-s_{2}}{\hat{p}}+\frac{\beta-t_{1}}{q},\,a_{2,2}:=\frac{\alpha-s_{2}}{\hat{p}}+\frac{\beta-t_{2}}{\hat{q}}.$
For ease of presentation, here we use $(R_{X_{1}},R_{Y_{1}},R_{X_{2}},R_{Y_{2}})$
to denote $(Q_{X},Q_{Y},R_{X},R_{Y})$ in the expression of $\overline{\Lambda}_{p,q,\hat{p},\hat{q}}^{**}(\alpha,\beta)$
in \eqref{eq:-59}, and use $(s_{1},t_{1},s_{2},t_{2})$ to denote
$(s,t,\hat{s},\hat{t})$. By the lemma above, we can rewrite 
\begin{align*}
 & \chi(R_{X_{1}^{n}},R_{Y_{1}^{n}},R_{X_{2}^{n}},R_{Y_{2}^{n}})\\
 & =\inf_{R_{X_{1}^{n}Y_{1}^{n}X_{2}^{n}Y_{2}^{n}}\in\mathcal{C}(R_{X_{1}^{n}},R_{Y_{1}^{n}},R_{X_{2}^{n}},R_{Y_{2}^{n}})}\min_{i,j\in[2]}\{\frac{1}{n}D(R_{X_{i}^{n}Y_{j}^{n}}\|P_{XY}^{\otimes n})+a_{i,j}\}.
\end{align*}
Denote $R_{IJ}$ as a distribution on $[2]^{2}$. Then we can rewrite
\begin{align*}
\chi(R_{X_{1}^{n}},R_{Y_{1}^{n}},R_{X_{2}^{n}},R_{Y_{2}^{n}}) & =\inf_{R_{X_{1}^{n}Y_{1}^{n}X_{2}^{n}Y_{2}^{n}}\in\mathcal{C}(R_{X_{1}^{n}},R_{Y_{1}^{n}},R_{X_{2}^{n}},R_{Y_{2}^{n}})}\min_{R_{IJ}}\{\frac{1}{n}D(R_{X_{I}^{n}Y_{J}^{n}}\|P_{XY}^{\otimes n}|R_{IJ})+\mathbb{E}_{R_{I,J}}[a_{I,J}]\}\\
 & =\min_{R_{IJ}}\{\omega(R_{X_{1}^{n}},R_{Y_{1}^{n}},R_{X_{2}^{n}},R_{Y_{2}^{n}})+\mathbb{E}_{R_{I,J}}[a_{I,J}]\},
\end{align*}
where 
\[
\omega(R_{X_{1}^{n}},R_{Y_{1}^{n}},R_{X_{2}^{n}},R_{Y_{2}^{n}}):=\inf_{R_{X_{1}^{n}Y_{1}^{n}X_{2}^{n}Y_{2}^{n}}\in\mathcal{C}(R_{X_{1}^{n}},R_{Y_{1}^{n}},R_{X_{2}^{n}},R_{Y_{2}^{n}})}\frac{1}{n}D(R_{X_{I}^{n}Y_{J}^{n}}\|P_{XY}^{\otimes n}|R_{IJ}).
\]
To single-letterize $\chi(R_{X_{1}^{n}},R_{Y_{1}^{n}},R_{X_{2}^{n}},R_{Y_{2}^{n}})$,
we need the following ``chain rule'' for coupling sets. 
\begin{lem}[``Chain Rule'' for Coupling Sets]
\cite[Lemma 9]{yu2020exact} \label{lem:coupling} For a pair of
regular conditional distributions $(P_{X^{n}|W},P_{Y^{n}|W})$, we
have 
\[
\prod_{i=1}^{n}\mathcal{C}(P_{X_{i}|X^{i-1}W},P_{Y_{i}|Y^{i-1}W})\subseteq\mathcal{C}(P_{X^{n}|W},P_{Y^{n}|W}),
\]
where for $i\in[n]$, 
\begin{align*}
\mathcal{C}(P_{X_{i}|X^{i-1}W},P_{Y_{i}|Y^{i-1}W}) & :=\{Q_{X_{i}Y_{i}|X^{i-1}Y^{i-1}W}:\\
 & \qquad Q_{X_{i}|X^{i-1}Y^{i-1}W}=P_{X_{i}|X^{i-1}W},Q_{Y_{i}|X^{i-1}Y^{i-1}W}=P_{Y_{i}|Y^{i-1}W}\},
\end{align*}
and 
\[
\prod_{i=1}^{n}\mathcal{C}(P_{X_{i}|X^{i-1}W},P_{Y_{i}|Y^{i-1}W}):=\{\prod_{i=1}^{n}Q_{X_{i}Y_{i}|X^{i-1}Y^{i-1}W}:\,Q_{X_{i}Y_{i}|X^{i-1}Y^{i-1}W}\in\mathcal{C}(P_{X_{i}|X^{i-1}W},P_{Y_{i}|Y^{i-1}W}),i\in[n]\}.
\]
\end{lem}

This lemma also holds for the coupling set of multiple marginal distributions.
Denote 
\[
D_{k}:=D(R_{X_{I,k}Y_{J,k}|X_{1}^{k-1}Y_{1}^{k-1}X_{2}^{k-1}Y_{2}^{k-1}}\|P_{XY}|R_{X_{1}^{k-1}Y_{1}^{k-1}X_{2}^{k-1}Y_{2}^{k-1}}R_{IJ}).
\]
Then, by the chain rule, $D(R_{X_{I}^{n}Y_{J}^{n}}\|P_{XY}^{\otimes n}|R_{IJ})=\sum_{k=1}^{n}D_{k}$.
Applying the lemma above, we obtain that 
\begin{align}
 & n\omega(R_{X_{1}^{n}},R_{Y_{1}^{n}},R_{X_{2}^{n}},R_{Y_{2}^{n}})\leq\inf_{\substack{R_{X_{1,k}Y_{1,k}X_{2,k}Y_{2,k}|X_{1}^{k-1}Y_{1}^{k-1}X_{2}^{k-1}Y_{2}^{k-1}}\\
\in\mathcal{C}(R_{X_{1,k}|X_{1}^{k-1}},R_{Y_{1,k}|Y_{1}^{k-1}},R_{X_{2,k}|X_{2}^{k-1}},R_{Y_{2,k}|Y_{2}^{k-1}}),k\in[n]
}
}\sum_{k=1}^{n}D_{k}.\label{eq:-11}
\end{align}
Obviously, the RHS above can be rewritten as 
\begin{align}
 & \inf_{\substack{R_{X_{1,1}Y_{1,1}X_{2,1}Y_{2,1}}\in\\
\mathcal{C}(R_{X_{1,1}},R_{Y_{1,1}},R_{X_{2,1}},R_{Y_{2,1}})
}
}\Biggl(D_{1}+...+\inf_{\substack{R_{X_{1,n-1}Y_{1,n-1}X_{2,n-1}Y_{2,n-1}|X_{1}^{n-2}Y_{1}^{n-2}X_{2}^{n-2}Y_{2}^{n-2}}\in\\
\mathcal{C}(R_{X_{1,n-1}|X_{1}^{n-2}},R_{Y_{1,n-1}|Y_{1}^{n-2}},R_{X_{2,n-1}|X_{2}^{n-2}},R_{Y_{2,n-1}|Y_{2}^{n-2}})
}
}\nonumber \\
 & \qquad\biggl(D_{n-1}+\inf_{\substack{R_{X_{1,n}Y_{1,n}X_{2,n}Y_{2,n}|X_{1}^{n-1}Y_{1}^{n-1}X_{2}^{n-1}Y_{2}^{n-1}}\in\\
\mathcal{C}(R_{X_{1,n}|X_{1}^{n-1}},R_{Y_{1,n}|Y_{1}^{n-1}},R_{X_{2,n}|X_{2}^{n-1}},R_{Y_{2,n}|Y_{2}^{n-1}})
}
}D_{n}\biggr)\Biggr).\label{eq:nestedinf}
\end{align}
In order to simplify the expression above, we take a supremization
for each infimization in the nested optimization above. Specifically,
for the $k$-th infimization, we take the supremization over $R_{X_{1}^{k-1}Y_{1}^{k-1}X_{2}^{k-1}Y_{2}^{k-1}}\in\mathcal{C}(R_{X_{1}^{k-1}},R_{Y_{1}^{k-1}},R_{X_{2}^{k-1}},R_{Y_{2}^{k-1}})$.
Then, we obtain the following upper bound for \eqref{eq:nestedinf}:
\begin{align}
\sum_{k=1}^{n}\sup_{\substack{R_{X_{1}^{k-1}Y_{1}^{k-1}X_{2}^{k-1}Y_{2}^{k-1}}\in\\
\mathcal{C}(R_{X_{1}^{k-1}},R_{Y_{1}^{k-1}},R_{X_{2}^{k-1}},R_{Y_{2}^{k-1}})
}
}\inf_{\substack{R_{X_{1,k}Y_{1,k}X_{2,k}Y_{2,k}|X_{1}^{k-1}Y_{1}^{k-1}X_{2}^{k-1}Y_{2}^{k-1}}\in\\
\mathcal{C}(R_{X_{1,k}|X_{1}^{k-1}},R_{Y_{1,k}|Y_{1}^{k-1}},R_{X_{2,k}|X_{2}^{k-1}},R_{Y_{2,k}|Y_{2}^{k-1}})
}
}D_{k}.\label{eq:sumDk}
\end{align}
We denote $K\sim R_{K}:=\mathrm{Unif}[n]$ as a random index which
is assumed to be independent of all other random variables, and also
denote $X_{i}:=X_{i,K},Y_{j}:=Y_{j,K},U_{i}:=X_{i}^{K-1},V_{j}:=Y_{j}^{K-1},i,j\in[2]$,
and $W:=(U_{1},V_{1},U_{2},V_{2})$. Then, \eqref{eq:sumDk} can be
rewritten as 
\begin{align}
 & \sum_{k=1}^{n}\sup_{\substack{R_{W|K=k}\in\\
\mathcal{C}(R_{U_{1}|K=k},R_{V_{1}|K=k},R_{U_{2}|K=k},R_{V_{2}|K=k})
}
}\nonumber \\
 & \qquad\inf_{\substack{R_{X_{1}Y_{1}X_{2}Y_{2}|W,K=k}\in\\
\mathcal{C}(R_{X_{1}|U_{1},K=k},R_{Y_{1}|V_{1},K=k},R_{X_{2}|U_{2},K=k},R_{Y_{2}|V_{2},K=k})
}
}D(R_{X_{I}Y_{J}|W,K=k}\|P_{XY}|R_{W|K=k}R_{IJ})\nonumber \\
 & =\sup_{\substack{R_{W|K}\in\\
\mathcal{C}(R_{U_{1}|K},R_{V_{1}|K},R_{U_{2}|K},R_{V_{2}|K})
}
}\inf_{\substack{R_{X_{1}Y_{1}X_{2}Y_{2}|WK}\in\\
\mathcal{C}(R_{X_{1}|U_{1}K},R_{Y_{1}|V_{1}K},R_{X_{2}|U_{2}K},R_{Y_{2}|V_{2}K})
}
}\sum_{k=1}^{n}D(R_{X_{I}Y_{J}|W,K=k}\|P_{XY}|R_{W|K=k}R_{IJ}).\label{eq:-41}
\end{align}
Here we swap the summation with the supremization and infimization
which is feasible since the supremization and infimization is taken
for each term in the summation independently. The summation in the
last line above is in fact equal to $nD(R_{X_{I}Y_{J}|WK}\|P_{XY}|R_{WK}R_{IJ})$.
Hence, the infimization in \eqref{eq:-41} can be rewritten as 
\begin{align*}
 & \inf_{\substack{R_{X_{1}Y_{1}X_{2}Y_{2}|WK}\in\\
\mathcal{C}(R_{X_{1}|U_{1}K},R_{Y_{1}|V_{1}K},R_{X_{2}|U_{2}K},R_{Y_{2}|V_{2}K})
}
}D(R_{X_{I}Y_{J}|WK}\|P_{XY}|R_{WK}R_{IJ})\\
 & =\mathbb{E}_{R_{WK}}[g(W,K,R_{IJ})]
\end{align*}
with 
\[
g(w,k,R_{IJ}):=\inf_{\substack{R_{X_{1}Y_{1}X_{2}Y_{2}|W=w,K=k}\in\\
\mathcal{C}(R_{X_{1}|U_{1}=u_{1},K=k},R_{Y_{1}|V_{1}=v_{1},K=k},R_{X_{2}|U_{2}=u_{2},K=k},R_{Y_{2}|V_{2}=v_{2},K=k})
}
}D(R_{X_{I}Y_{J}|W=w,K=k}\|P_{XY}|R_{IJ})
\]
and $w=(u_{1},v_{1},u_{2},v_{2})$. Then summarizing the derivations
above, we obtain that 
\begin{align}
\chi(R_{X_{1}^{n}},R_{Y_{1}^{n}},R_{X_{2}^{n}},R_{Y_{2}^{n}}) & \le\min_{R_{IJ}}\sup_{\substack{R_{W|K}\in\\
\mathcal{C}(R_{U_{1}|K},R_{V_{1}|K},R_{U_{2}|K},R_{V_{2}|K})
}
}\{\mathbb{E}_{R_{WK}}[g(W,K,R_{IJ})]-\mathbb{E}_{R_{I,J}}[a_{I,J}]\}.\label{eq:minsup}
\end{align}

We now claim that the minimization and the supremization at the RHS
above can be swapped. We now prove this claim.

\textbf{Fact 1.} The space of $R_{IJ}$ is the probability simplex,
which is a nonempty convex compact subset of $\mathbb{R}^{4}$. 

\textbf{Fact 2.} The coupling set $\mathcal{C}(R_{U_{1}|K},R_{V_{1}|K},R_{U_{2}|K},R_{V_{2}|K})$
is a nonempty convex subset of a linear space (the space of nonnegative
measures).

\textbf{Fact 3.} For each $(w,k)$, $g(w,k,R_{IJ})$ is linear in
$R_{IJ}$, i.e., $g(w,k,R_{IJ})=\sum_{i,j\in[2]}R_{IJ}(i,j)b_{i,j,w,k}$
with 
\begin{equation}
b_{i,j,w,k}:=\mathbb{D}(R_{X_{i}|U_{i}=u_{i},K=k},R_{Y_{j}|V_{j}=v_{j},K=k}\|P_{XY})\in[0,\infty],\;i,j\in[2].\label{eq:fact3}
\end{equation}


Facts 1 and 2 are obvious. We now prove Fact 3. For any $\lambda\in[0,1],R_{IJ}^{(0)},R_{IJ}^{(1)}$,
it holds that 
\begin{align}
 & D(R_{X_{I}Y_{J}|W=w,K=k}\|P_{XY}|\lambda R_{IJ}^{(0)}+(1-\lambda)R_{IJ}^{(1)})\nonumber \\
 & =\lambda D(R_{X_{I}Y_{J}|W=w,K=k}\|P_{XY}|R_{IJ}^{(0)})+(1-\lambda)D(R_{X_{I}Y_{J}|W=w,K=k}\|P_{XY}|R_{IJ}^{(1)}).
\end{align}
Taking infimum over $R_{X_{1}Y_{1}X_{2}Y_{2}|W=w,K=k}\in\mathcal{C}(R_{X_{1}|U_{1}=u_{1},K=k},R_{Y_{1}|V_{1}=v_{1},K=k},R_{X_{2}|U_{2}=u_{2},K=k},R_{Y_{2}|V_{2}=v_{2},K=k})$,
we obtain that $g(w,k,\lambda R_{IJ}^{(0)}+(1-\lambda)R_{IJ}^{(1)})\ge\lambda g(w,k,R_{IJ}^{(0)})+(1-\lambda)g(w,k,R_{IJ}^{(1)})$.

On the other hand, for any $R_{X_{I}Y_{J}|W=w,K=k}^{(0)},R_{X_{I}Y_{J}|W=w,K=k}^{(1)}$,
denote $R_{X_{I}Y_{J}IJ|W=w,K=k}^{(\lambda)}:=\lambda R_{X_{I}Y_{J}IJ|W=w,K=k}^{(0)}+(1-\lambda)R_{X_{I}Y_{J}IJ|W=w,K=k}^{(1)}$,
where $R_{X_{I}Y_{J}IJ|W=w,K=k}^{(l)}$ with $l\in[2]$ denotes the
joint distribution induced by the distribution $R_{IJ}^{(l)}$ and
the conditional distribution $R_{X_{i}Y_{j}|W=w,K=k}^{(l)}$ given
$I=i,J=j$. Denote $R_{IJ}^{(\lambda)}$ and $R_{X_{I}Y_{J}|W=w,K=k}^{(\lambda)}$
as the marginal distribution and the conditional distribution induced
by the joint distribution $R_{X_{I}Y_{J}IJ|W=w,K=k}^{(\lambda)}$.
By the convexity of the relative entropy \cite{Erven}, 
\begin{align}
 & \lambda D(R_{X_{I}Y_{J}|W=w,K=k}^{(0)}\|P_{XY}|R_{IJ}^{(0)})+(1-\lambda)D(R_{X_{I}Y_{J}|W=w,K=k}^{(1)}\|P_{XY}|R_{IJ}^{(1)})\nonumber \\
 & \ge D(R_{X_{I}Y_{J}|W=w,K=k}^{(\lambda)}\|P_{XY}|R_{IJ}^{(\lambda)}).
\end{align}
Taking infimum over 
\[
R_{X_{1}Y_{1}X_{2}Y_{2}|W=w,K=k}^{(0)},R_{X_{1}Y_{1}X_{2}Y_{2}|W=w,K=k}^{(1)}\in\mathcal{C}(R_{X_{1}|U_{1}=u_{1},K=k},R_{Y_{1}|V_{1}=v_{1},K=k},R_{X_{2}|U_{2}=u_{2},K=k},R_{Y_{2}|V_{2}=v_{2},K=k}),
\]
we obtain that $\lambda g(w,k,R_{IJ}^{(0)})+(1-\lambda)g(w,k,R_{IJ}^{(1)})\ge g(w,k,R_{IJ}^{(\lambda)})=g(w,k,\lambda R_{IJ}^{(0)}+(1-\lambda)R_{IJ}^{(1)})$.
Combining the above two points, $g(w,k,\lambda R_{IJ}^{(0)}+(1-\lambda)R_{IJ}^{(1)})=\lambda g(w,k,R_{IJ}^{(0)})+(1-\lambda)g(w,k,R_{IJ}^{(1)})$
for any $\lambda\in[0,1],R_{IJ}^{(0)},R_{IJ}^{(1)}$. That is, Fact
3 is true.

Based on Facts 1-3, to prove the feasibility of the exchange of the
minimization and the supremization in \eqref{eq:minsup}, it suffices
to show the following minimax lemma. Note that in our setting, the
coefficients $b_{i,j,w,k},\;i,j\in[2]$ could be $+\infty$.
\begin{lem}
\label{lem:minimax} Let $\mathcal{X}$ be a finite set, and $\mathcal{Y}$
a Polish space. Let $b:\mathcal{X}\times\mathcal{Y}\to[a,+\infty]$
with $a\in\mathbb{R}$ be an extended real valued measurable function,
and $B\subseteq\mathcal{P}(\mathcal{Y})$ a nonempty convex set. 
Then, for the bilinear map 
\[
g:(Q_{X},Q_{Y})\in\mathcal{P}(\mathcal{X})\times B\mapsto\mathbb{E}_{(X,Y)\sim Q_{X}\otimes Q_{Y}}[b(X,Y)],
\]
it holds that 
\[
\min_{Q_{X}}\sup_{Q_{Y}\in B}g(Q_{X},Q_{Y})=\sup_{Q_{Y}\in B}\min_{Q_{X}}g(Q_{X},Q_{Y}).
\]
\end{lem}

\begin{proof}[Proof of Lemma \ref{lem:minimax}]
 For brevity, we assume $a=0$ (by re-setting $g$ to $g-a$). By
weak duality, 
\[
\min_{Q_{X}}\sup_{Q_{Y}\in B}g(Q_{X},Q_{Y})\ge\sup_{Q_{Y}\in B}\min_{Q_{X}}g(Q_{X},Q_{Y}).
\]

Denote $\hat{\mathcal{X}}:=\{x\in\mathcal{X}:\sup_{Q_{Y}\in B}\mathbb{E}_{Y\sim Q_{Y}}[b(x,Y)]<\infty\}$.
We first assume that $\hat{\mathcal{X}}$ is nonempty.   Denote $A:=\mathcal{P}(\hat{\mathcal{X}})$.
Since $\hat{\mathcal{X}}$ is finite, the function $(x,Q_{Y})\in\hat{\mathcal{X}}\times B\mapsto\mathbb{E}_{Y\sim Q_{Y}}[b(x,Y)]$
is bounded. Hence, for any $Q_{Y}\in B$, $Q_{X}\in A\mapsto g(Q_{X},Q_{Y})$
is continuous in the relative topology of probability simplex. By
\cite[Theorem 2.10.2]{zalinescu2002convex}, 
\begin{align*}
\min_{Q_{X}}\sup_{Q_{Y}\in B}g(Q_{X},Q_{Y}) & =\inf_{Q_{X}\in A}\sup_{Q_{Y}\in B}g(Q_{X},Q_{Y})=\sup_{Q_{Y}\in B}\inf_{Q_{X}\in A}g(Q_{X},Q_{Y})=:\tau.
\end{align*}
Hence, $\inf_{Q_{X}\in A}g(Q_{X},Q_{Y})\le\tau$ for all $Q_{Y}\in B$.

Let $r$ be a real value such that $\frac{\delta}{|\hat{\mathcal{X}}^{c}|}r>(1-\delta)\tau$
and let $\delta\in(0,1)$. For each $x\in\hat{\mathcal{X}}^{c}:=\mathcal{X}\backslash\hat{\mathcal{X}}$,
let $Q_{Y}^{(x)}\in B$ be such that $\mathbb{E}_{Y\sim Q_{Y}^{(x)}}[b(x,Y)]>r$.
For each $Q_{Y}$, denote $Q_{Y}^{(\delta)}:=(1-\delta)Q_{Y}+\sum_{x\in\hat{\mathcal{X}}^{c}}\frac{\delta}{|\hat{\mathcal{X}}^{c}|}Q_{Y}^{(x)}$
as a mixture distribution. For this distribution, 
\begin{align}
g(Q_{X},Q_{Y}^{(\delta)}) & =(1-\delta)g(Q_{X},Q_{Y})+\sum_{x'\in\hat{\mathcal{X}}^{c}}\frac{\delta}{|\hat{\mathcal{X}}^{c}|}g(Q_{X},Q_{Y}^{(x')}).\label{eq:-33}
\end{align}

For the first term at the RHS above, 
\begin{align}
g(Q_{X},Q_{Y}) & \ge Q_{X}(\hat{\mathcal{X}})g(\hat{Q}_{X},Q_{Y}),\label{eq:-12}
\end{align}
where $\hat{Q}_{X}=Q_{X}1_{\hat{\mathcal{X}}}/Q_{X}(\hat{\mathcal{X}})$.
For the second term, for each $x'\in\hat{\mathcal{X}}^{c}$, 
\begin{align}
g(Q_{X},Q_{Y}^{(x')}) & =\sum_{x\in\mathcal{X}}Q_{X}(x)\mathbb{E}_{Y\sim Q_{Y}^{(x')}}[b(x,Y)]\ge Q_{X}(x')r.\label{eq:-17}
\end{align}

Substituting \eqref{eq:-12} and \eqref{eq:-17} into \eqref{eq:-33}
yields 
\begin{align*}
g(Q_{X},Q_{Y}^{(\delta)}) & \ge(1-\delta)Q_{X}(\hat{\mathcal{X}})g(\hat{Q}_{X},Q_{Y})+\frac{\delta}{|\hat{\mathcal{X}}^{c}|}Q_{X}(\hat{\mathcal{X}}^{c})r\\
 & \ge\min\Bigl\{(1-\delta)\inf_{\hat{Q}_{X}\in A}g(\hat{Q}_{X},Q_{Y}),(1-\delta)\tau\Bigr\}\\
 & =(1-\delta)\inf_{\hat{Q}_{X}\in A}g(\hat{Q}_{X},Q_{Y}).
\end{align*}
Taking $\sup_{Q_{Y}\in B}\inf_{Q_{X}}$, we obtain 
\begin{align*}
\sup_{Q_{Y}\in B}\inf_{Q_{X}}g(Q_{X},Q_{Y}^{(\delta)}) & \ge(1-\delta)\sup_{Q_{Y}\in B}\inf_{\hat{Q}_{X}\in A}g(\hat{Q}_{X},Q_{Y})=(1-\delta)\tau,
\end{align*}
which implies that $\sup_{Q_{Y}\in B}\inf_{Q_{X}}g(Q_{X},Q_{Y})\ge(1-\delta)\tau.$
Letting $\delta\downarrow0$, we have $\sup_{Q_{Y}\in B}\inf_{Q_{X}}g(Q_{X},Q_{Y})\ge\tau.$ 

If $\hat{\mathcal{X}}$ is empty, then $\min_{Q_{X}}\sup_{Q_{Y}\in B}g(Q_{X},Q_{Y})=+\infty$.
On the other hand, for $Q_{Y}^{(\delta)}$ constructed above, $g(Q_{X},Q_{Y}^{(\delta)})=+\infty$
for all $Q_{X}$ and $Q_{Y}\in B$. Hence, $\sup_{Q_{Y}\in B}\min_{Q_{X}}g(Q_{X},Q_{Y})$
is also equal to $+\infty$. 
\end{proof}
In our setting, $a_{i,j}$ with $i,j\in[2]$ are finite, and $b_{i,j,w,k}\in[0,\infty]$.
Observe that the objective function at the RHS of \eqref{eq:minsup}
is $\mathbb{E}_{R_{WK}\otimes R_{IJ}}[b_{I,J,W,K}-a_{I,J}]$. By the
minimax lemma above, the RHS of \eqref{eq:minsup} is equal to 
\[
\sup_{\substack{R_{W|K}\in\\
\mathcal{C}(R_{U_{1}|K},R_{V_{1}|K},R_{U_{2}|K},R_{V_{2}|K})
}
}\min_{i,j\in[2]}\{\mathbb{E}_{R_{WK}}[b_{i,j,W,K}]-a_{i,j}\}.
\]

Substituting this formula into \eqref{eq:-29}, we obtain that 
\begin{align*}
 & \frac{1}{n}\overline{\Lambda}_{p,q,\hat{p},\hat{q}}^{**}(n\alpha,n\beta|P_{XY}^{\otimes n})\\
 & \leq\sup_{\substack{\alpha_{\lambda}^{-}\le s_{1}\le\alpha_{\lambda}^{+},\beta_{\mu}^{-}\le t_{1}\leq\beta_{\mu}^{+},\\
\alpha_{1/\lambda}^{-}\le s_{2}\le\alpha_{1/\lambda}^{+},\beta_{1/\mu}^{-}\le t_{2}\leq\beta_{1/\mu}^{+}
}
}\sup_{\substack{R_{K},R_{X_{1}U_{1}|K},R_{Y_{1}V_{1}|K},R_{X_{2}U_{2}|K},R_{Y_{2}V_{2}|K}:\\
D(R_{X_{i}|U_{i}K}\|P_{X}|R_{U_{i}|K}R_{K})=s_{i},\\
D(R_{Y_{j}|V_{j}K}\|P_{Y}|R_{V_{j}|K}R_{K})=t_{j},\forall i,j\in[2]
}
}\\
 & \qquad\sup_{\substack{R_{W|K}\in\\
\mathcal{C}(R_{U_{1}|K},R_{V_{1}|K},R_{U_{2}|K},R_{V_{2}|K})
}
}\min_{i,j\in[2]}\{\mathbb{D}(R_{X_{i}|U_{i}K},R_{Y_{j}|V_{j}K}\|P_{XY}|R_{WK})-a_{i,j}\}.
\end{align*}
Observe that the tuple consisting of all the distributions appearing
in the second and third supremizations above and the joint distribution
$R_{K}R_{W|K}R_{X_{1}|U_{1}K}R_{Y_{1}|V_{1}K}R_{X_{2}|U_{2}K}R_{Y_{2}|V_{2}K}$
are mutually determined by each other. Hence, we can replace the second
and third supremizations above with the supremization taken over the
joint distributions 
\begin{equation}
R_{K}R_{W|K}R_{X_{1}|U_{1}K}R_{Y_{1}|V_{1}K}R_{X_{2}|U_{2}K}R_{Y_{2}|V_{2}K}\label{eq:-40}
\end{equation}
satisfying the constraints under the second supremization above.

Denote $\hat{W}:=(W,K)$. Observe that for the joint distribution
in \eqref{eq:-40}, we have $R_{X_{i}|U_{i}K}=R_{X_{i}|\hat{W}}$
and $R_{Y_{j}|V_{j}K}=R_{Y_{j}|\hat{W}}$. Hence, we further have
\begin{align}
\frac{1}{n}\overline{\Lambda}_{p,q,\hat{p},\hat{q}}^{**}(n\alpha,n\beta|P_{XY}^{\otimes n}) & \leq\sup_{\substack{\alpha_{\lambda}^{-}\le s_{1}\le\alpha_{\lambda}^{+},\beta_{\mu}^{-}\le t_{1}\leq\beta_{\mu}^{+},\\
\alpha_{1/\lambda}^{-}\le s_{2}\le\alpha_{1/\lambda}^{+},\beta_{1/\mu}^{-}\le t_{2}\leq\beta_{1/\mu}^{+}
}
}\sup_{\substack{R_{\hat{W}}R_{X_{1}|\hat{W}}R_{Y_{1}|\hat{W}}R_{X_{2}|\hat{W}}R_{Y_{2}|\hat{W}}:\\
D(R_{X_{i}|\hat{W}}\|P_{X}|R_{\hat{W}})=s_{i},\\
D(R_{Y_{j}|\hat{W}}\|P_{Y}|R_{\hat{W}})=t_{j},\forall i,j\in[2]
}
}\nonumber \\
 & \qquad\min_{i,j\in[2]}\{\mathbb{D}(R_{X_{i}|\hat{W}},R_{Y_{j}|\hat{W}}\|P_{XY}|R_{\hat{W}})-a_{i,j}\}\label{eq:-48}\\
 & =\overline{\Lambda}_{p,q,\hat{p},\hat{q}}^{*}(\alpha,\beta|P_{XY}),
\end{align}
where in \eqref{eq:-48}, \eqref{eq:-40} is relaxed to arbitrary
distributions $R_{\hat{W}}R_{X_{1}|\hat{W}}R_{Y_{1}|\hat{W}}R_{X_{2}|\hat{W}}R_{Y_{2}|\hat{W}}$.

II. We next prove the case $p,q<0,\hat{p},\hat{q}\ge0$. For this
case, by definition, 
\begin{align}
\frac{1}{n}\overline{\Lambda}_{p,q,\hat{p},\hat{q}}^{**}(n\alpha,n\beta|P_{XY}^{\otimes n}) & =\sup_{R_{X_{2}^{n}},R_{Y_{2}^{n}}}\inf_{\substack{R_{X_{2}^{n}Y_{2}^{n}}\in\mathcal{C}(R_{X_{2}^{n}},R_{Y_{2}^{n}}),\\
R_{X_{1}^{n}Y_{1}^{n}|X_{2}^{n}Y_{2}^{n}}
}
}\min_{i,j\in[2]}\frac{1}{n}D(R_{X_{i}^{n}Y_{j}^{n}}\|P_{XY}^{n})\nonumber \\
 & \qquad+\frac{\alpha-\frac{1}{n}D(R_{X_{i}^{n}}\|P_{X}^{n})}{p_{i}}+\frac{\beta-\frac{1}{n}D(R_{Y_{j}^{n}}\|P_{Y}^{n})}{q_{j}}\label{eq:-50}
\end{align}
where $p_{1}=p,p_{2}=\hat{p},q_{1}=q,q_{2}=\hat{q}$. By Lemma \ref{lem:coupling},
the RHS above is upper bounded by 
\begin{align}
 & \sup_{R_{X_{2,k}|X_{2}^{k-1}},R_{Y_{2,k}|Y_{2}^{k-1}},k\in[n]}\min_{R_{IJ}}\inf_{\substack{R_{X_{1,k}Y_{1,k}X_{2,k}Y_{2,k}|X_{1}^{k-1}Y_{1}^{k-1}X_{2}^{k-1}Y_{2}^{k-1}},k\in[n]:\\
\textrm{marginals are }R_{X_{2,k}|X_{2}^{k-1}},R_{Y_{2,k}|Y_{2}^{k-1}}
}
}\nonumber \\
 & \qquad\frac{1}{n}\sum_{k=1}^{n}\mathbb{E}_{R_{IJ}}\biggl[D(R_{X_{I,k}Y_{J,k}|X_{I}^{k-1}Y_{J}^{k-1}}\|P_{XY}|R_{X_{I}^{k-1}Y_{J}^{k-1}})+\frac{\alpha-D(R_{X_{I,k}|X_{I}^{k-1}}\|P_{X}|R_{X_{I}^{k-1}})}{p_{I}}\nonumber \\
 & \qquad+\frac{\beta-D(R_{Y_{J,k}|Y_{J}^{k-1}}\|P_{Y}|R_{Y_{J}^{k-1}})}{q_{J}}\biggr].\label{eq:-65}
\end{align}
Similarly to the derivation above, if we insert a supremization for
each infimization in the expression above, we obtain the following
upper bound:
\begin{align}
 & \sup_{R_{X_{2,k}|X_{2}^{k-1}},R_{Y_{2,k}|Y_{2}^{k-1}},k\in[n]}\min_{R_{IJ}}\frac{1}{n}\sum_{k=1}^{n}\sup_{\substack{R_{X_{1}^{k-1}Y_{1}^{k-1}X_{2}^{k-1}Y_{2}^{k-1}},k\in[n]:\\
\textrm{marginals are }R_{X_{2}^{k-1}},R_{Y_{2}^{k-1}}
}
}\inf_{\substack{R_{X_{1,k}Y_{1,k}X_{2,k}Y_{2,k}|X_{1}^{k-1}Y_{1}^{k-1}X_{2}^{k-1}Y_{2}^{k-1}},k\in[n]:\\
\textrm{marginals are }R_{X_{2,k}|X_{2}^{k-1}},R_{Y_{2,k}|Y_{2}^{k-1}}
}
}\nonumber \\
 & \qquad\mathbb{E}_{R_{IJ}}\biggl[D(R_{X_{I,k}Y_{J,k}|X_{I}^{k-1}Y_{J}^{k-1}}\|P_{XY}|R_{X_{I}^{k-1}Y_{J}^{k-1}})+\frac{\alpha-D(R_{X_{I,k}|X_{I}^{k-1}}\|P_{X}|R_{X_{I}^{k-1}})}{p_{I}}\nonumber \\
 & \qquad+\frac{\beta-D(R_{Y_{J,k}|Y_{J}^{k-1}}\|P_{Y}|R_{Y_{J}^{k-1}})}{q_{J}}\biggr].\label{eq:-74}
\end{align}
We denote $K\sim R_{K}:=\mathrm{Unif}[n]$ as a random index which
is assumed to be independent of all other random variables, and also
denote $X_{i}:=X_{i,K},Y_{j}:=Y_{j,K},U_{i}:=X_{i}^{K-1},V_{j}:=Y_{j}^{K-1},i,j\in[2]$,
and $W:=(U_{1},V_{1},U_{2},V_{2})$.  Then, the above expression
is further upper bounded by 
\begin{align}
 & \sup_{R_{K},R_{U_{2}|K},R_{V_{2}|K},R_{X_{2}|U_{2}K},R_{Y_{2}|V_{2}K}}\min_{R_{IJ}}\sup_{\substack{R_{W|K}:\\
\textrm{marginals are }R_{U_{2}|K},R_{V_{2}|K}
}
}\inf_{\substack{R_{X_{1}Y_{1}X_{2}Y_{2}|WK}:\\
\textrm{marginals are }R_{X_{2}|U_{2}K},R_{Y_{2}|V_{2}K}
}
}\nonumber \\
 & \qquad\mathbb{E}_{R_{IJ}}\biggl[D(R_{X_{I}Y_{J}|WK}\|P_{XY}|R_{WK})+\frac{\alpha-D(R_{X_{I}|WK}\|P_{X}|R_{WK})}{p_{I}}\nonumber \\
 & \qquad+\frac{\beta-D(R_{Y_{J}|WK}\|P_{Y}|R_{WK})}{q_{J}}\biggr].\label{eq:-75}
\end{align}
where for $i=1$ and $j=1$, the inequalities 
\begin{align*}
D(R_{X_{1,k}|X_{i}^{k-1}}\|P_{X}|R_{X_{1}^{k-1}}) & \le D(R_{X_{1}|WK}\|P_{X}|R_{WK})\\
D(R_{Y_{1,k}|Y_{1}^{k-1}}\|P_{Y}|R_{Y_{1}^{k-1}}) & \le D(R_{Y_{1}|WK}\|P_{Y}|R_{WK})
\end{align*}
are applied which follow by the fact that conditioning increases relative
entropy. By Fact 3 above \eqref{eq:fact3}, the last infimization
and the expectation $\mathbb{E}_{R_{IJ}}$ can be swapped. Hence,
the above expression is equal to 
\begin{align}
 & \sup_{R_{K},R_{U_{2}|K},R_{V_{2}|K},R_{X_{2}|U_{2}K},R_{Y_{2}|V_{2}K}}\min_{R_{IJ}}\sup_{\substack{R_{W|K}:\\
\textrm{marginals are }R_{U_{2}|K},R_{V_{2}|K}
}
}\nonumber \\
 & \qquad\mathbb{E}_{R_{IJ}}\biggl[\inf_{R_{X_{1}|WK},R_{Y_{1}|WK}}\mathbb{D}(R_{X_{I}|WK},R_{Y_{J}|WK}\|P_{XY}|R_{WK})+\frac{\alpha-D(R_{X_{I}|WK}\|P_{X}|R_{WK})}{p_{I}}\nonumber \\
 & \qquad+\frac{\beta-D(R_{Y_{J}|WK}\|P_{Y}|R_{WK})}{q_{J}}\biggr].\label{eq:-76}
\end{align}
By Lemma \ref{lem:minimax}, the operations $\min_{R_{IJ}}$ and $\sup_{R_{U_{1}V_{1}U_{2}V_{2}|K}:R_{U_{2}|K},R_{V_{2}|K}}$
can be swapped. Hence, we obtain 
\begin{align}
 & \sup_{R_{K},R_{U_{2}|K},R_{V_{2}|K},R_{X_{2}|U_{2}K},R_{Y_{2}|V_{2}K}}\sup_{\substack{R_{W|K}:\\
\textrm{marginals are }R_{U_{2}|K},R_{V_{2}|K}
}
}\min_{i,j\in[2]}\nonumber \\
 & \qquad\Bigl\{\inf_{R_{X_{1}|WK},R_{Y_{1}|WK}}\mathbb{D}(R_{X_{i}|WK},R_{Y_{j}|WK}\|P_{XY}|R_{WK})+\frac{\alpha-D(R_{X_{i}|WK}\|P_{X}|R_{WK})}{p_{i}}\nonumber \\
 & \qquad+\frac{\beta-D(R_{Y_{j}|WK}\|P_{Y}|R_{WK})}{q_{j}}\Bigr\}.\label{eq:-66}
\end{align}
Combining two supremizations above yields $\sup_{R_{K},R_{W|K},R_{X_{2}|U_{2}K},R_{Y_{2}|V_{2}K}}$,
which is equivalent to $\sup_{R_{KW}R_{X_{2}|U_{2}K}R_{Y_{2}|V_{2}K}}.$
By relaxing $R_{KW}R_{X_{2}|U_{2}K}R_{Y_{2}|V_{2}K}$ to arbitrary
distributions $R_{\hat{W}}R_{X_{2}|\hat{W}}R_{Y_{2}|\hat{W}}$, we
obtain the upper bound $\overline{\Lambda}_{p,q,\hat{p},\hat{q}}^{*}(\alpha,\beta|P_{XY})$
on the above expression. 

The inequality \eqref{eq:-77} for the case $p,q>0=\hat{p}=\hat{q}$
follows similarly. 

III. By symmetry, \eqref{eq:-77} also holds for the case $p,q>0,\hat{p},\hat{q}<0$. 

IV. We next prove the case $p,q,\hat{p},\hat{q}<0$. In fact, \eqref{eq:-77}
 for this case can be proven by following steps similar to those in
Case II. Specifically, we only need to do the following modifications:
1) replace the supremization and infimization  in \eqref{eq:-50}
with $\inf_{R_{X_{1}^{n}Y_{1}^{n}X_{2}^{n}Y_{2}^{n}}}$; 2) remove
the first supremizations in \eqref{eq:-65} and \eqref{eq:-74}, and
remove all the marginal constraints involved in these two expressions;
replace the first supremizations $\sup_{R_{K},R_{U_{2}|K},R_{V_{2}|K},R_{X_{2}|U_{2}K},R_{Y_{2}|V_{2}K}}$
in \eqref{eq:-75}-\eqref{eq:-66} with $\sup_{R_{K}}$, remove all
the marginal constraints involved in these expressions, and replace
$\inf_{R_{X_{1}|WK},R_{Y_{1}|WK}}$ in \eqref{eq:-76} and \eqref{eq:-66}
with $\inf_{R_{X_{1}|WK},R_{Y_{1}|WK},R_{X_{2}|WK},R_{Y_{2}|WK}}$.

\subsection{\label{subsec:Sharpness}Asymptotic Tightness: Forward Case}

We next prove the asymptotic tightness of the inequality in \eqref{eq:FBLE}
 by the method of types \cite{Csi97}. We assume $\mathcal{X},\mathcal{Y}$
are the supports of $P_{X},P_{Y}$, respectively. Before starting
the proof, we first introduce some notations on empirical measures
(or types). 

For a sequence $x^{n}\in\mathcal{X}^{n}$, we use $T_{x^{n}}$ as
the \emph{empirical measure} (or \emph{type}) of $x^{n}$. We use
$T_{X}$ and $T_{Y}$ to respectively denote empirical measures (or
types) of sequences in $\mathcal{X}^{n}$ and $\mathcal{Y}^{n}$,
and $T_{XY}$ to denote an empirical joint measure (or a joint type)
of a pair of sequences in $\mathcal{X}^{n}\times\mathcal{Y}^{n}$.
For an empirical measure $T_{X}$ (resp. an empirical joint measure
$T_{XY}$), the \emph{type class} $\mathcal{T}_{T_{X}}^{(n)}$ (resp.
the \emph{joint type class} $\mathcal{T}_{T_{XY}}^{(n)}$) is defined
as the set of sequences having the same empirical measure $T_{X}$
(resp. $T_{XY}$). In other words, if we denote the empirical measure
mapping as $t:x^{n}\mapsto T_{x^{n}}$, then the type class $\mathcal{T}_{T_{X}}^{(n)}=t^{-1}(T_{X})$.
For a sequence $w^{n}$, define the conditional type class $\mathcal{T}_{T_{X|W}}^{(n)}(w^{n})=\{x^{n}:(w^{n},x^{n})\textrm{ has joint type }T_{w^{n}}T_{X|W}\}$.
Shortly denote them as $\mathcal{T}_{T_{X}}$, $\mathcal{T}_{T_{XY}}$,
and $\mathcal{T}_{T_{X|W}}(w^{n})$. 

We now start the proof. Similarly to the proof of the one-shot bound
in Section \ref{subsec:One-shot-Bound}, we only provide the proofs
for the ``symmetric'' cases. For the ``asymmetric'' case, one
can prove the asymptotic tightness by ``mixing'' the proofs for
``symmetric'' cases, since the ``asymmetric'' case can be seen
as a mixture of the ``symmetric'' cases. 

I. We now prove the case $p,q,\hat{p},\hat{q}>0$. 

1) We first assume $p\neq\hat{p},q\neq\hat{q}$. Let $w^{n}$ be a
sequence in $\mathcal{W}^{n}$ with $n$-type $T_{W}$. We choose
\begin{align}
f=\sum_{T_{X|W}}e^{n\mu_{T_{X|W}}}1_{\mathcal{T}_{T_{X|W}}(w^{n})} & \quad\textrm{ and }\quad g=\sum_{T_{Y|W}}e^{n\nu_{T_{Y|W}}}1_{\mathcal{T}_{T_{Y|W}}(w^{n})},\label{eq:g}
\end{align}
where the summations are taken over all conditional $n$-types $T_{X|W}$
and $T_{Y|W}$ respectively. Denote $s_{T_{X|W}}:=D(T_{X|W}\|P_{X}|T_{W}).$
Then, by Sanov's theorem, 
\begin{align}
\frac{1}{n}\log\Vert f\Vert_{p} & =\frac{1}{np}\log(\sum_{T_{X|W}}P_{X}^{\otimes n}(\mathcal{T}_{T_{X|W}}(w^{n}))e^{np\mu_{T_{X|W}}})=\frac{1}{p}\max_{T_{X|W}}\{p\mu_{T_{X|W}}-s_{T_{X|W}}\}+o(1),\label{eq:-47}
\end{align}
and similarly, $\frac{1}{n}\log\Vert f\Vert_{\hat{p}}=\frac{1}{\hat{p}}\max_{T_{X|W}}\{\hat{p}\mu_{T_{X|W}}-s_{T_{X|W}}\}+o(1).$

It is well-known that for the finite alphabet case, the KL divergence
$D(Q\|P)$ is continuous in $Q$ given $P$ under the condition $Q\ll P$,
and the set of types is dense in the probability simplex (i.e., the
space of probability measures on the finite set $\mathcal{W}\times\mathcal{X}\times\mathcal{Y}$).
 By these properties, it is easily verified that if we choose 
\begin{equation}
\mu_{T_{X|W}}=\min\{\frac{1}{p}s_{T_{X|W}},\,\frac{1}{\hat{p}}s_{T_{X|W}}+(\frac{1}{p}-\frac{1}{\hat{p}})\alpha\}\textrm{ for all }T_{X|W},\label{eq:-37}
\end{equation}
then the resultant $\frac{1}{n}\log\Vert f\Vert_{p}=o(1)$, $\frac{1}{n}\log\Vert f\Vert_{\hat{p}}=(\frac{1}{p}-\frac{1}{\hat{p}})\alpha+o(1)$,
$\frac{1}{n}\Ent_{p,\hat{p}}(f)=\alpha+o(1)$, and 
\begin{align}
 & -\frac{1}{n}\log\frac{\langle f,g\rangle}{\Vert f\Vert_{p}\Vert g\Vert_{q}}+\frac{\alpha}{p}+\frac{\beta}{q}\nonumber \\
 & =-\frac{1}{n}\log\Big(\sum_{T_{X|W},T_{Y|W}}P_{XY}^{\otimes n}(\mathcal{T}_{T_{X|W}}(w^{n})\times\mathcal{T}_{T_{Y|W}}(w^{n}))e^{n\mu_{T_{X|W}}+n\nu_{T_{Y|W}}}\Big)+\frac{\alpha}{p}+\frac{\beta}{q}+o(1)\label{eq:Efg}\\
 & \leq\inf_{Q_{XYW}:Q_{W}=T_{W}}D(Q_{XY|W}\|P_{XY}|Q_{W})\nonumber \\
 & \qquad+\eta_{p,\hat{p}}(\alpha,D(Q_{X|W}\|P_{X}|Q_{W}))+\eta_{q,\hat{q}}(\beta,D(Q_{Y|W}\|P_{Y}|Q_{W}))+o(1).\label{eq:-82}
\end{align}
Optimizing \eqref{eq:-82} over all types $T_{W}$ yields that 
\begin{align}
-\frac{1}{n}\log\frac{\langle f,g\rangle}{\Vert f\Vert_{p}\Vert g\Vert_{q}}+\frac{\alpha}{p}+\frac{\beta}{q} & \le\underline{\Lambda}_{p,q,\hat{p},\hat{q}}^{*}(\alpha,\beta)+o(1).\label{eq:-1}
\end{align}

Note that the proof above is not finished, since in our problem, $\frac{1}{n}\Ent_{p,\hat{p}}(f)$
is restricted to be \emph{strictly} equal to $\alpha$, rather than
asymptotically equal to $\alpha$.  This gap can be fixed as follows.
For fixed $\alpha_{1}<\alpha<\alpha_{2}$, we choose $f_{1}$ as the
above $f$ for $\alpha_{1}$, and choose $f_{2}$ as the above $f$
for $\alpha_{2}$. Then, we construct a mixture $f_{\theta}=\theta f_{1}+\bar{\theta}f_{2}$
with $\bar{\theta}:=1-\theta$. Observe that for sufficiently large
$n$, $\alpha$ is sandwiched between $\frac{1}{n}\Ent_{p,\hat{p}}(f_{1})$
and $\frac{1}{n}\Ent_{p,\hat{p}}(f_{2})$, and moreover, given $f_{1},f_{2}$,
$\frac{1}{n}\Ent_{p,\hat{p}}(f_{\theta})$ is continuous in $\theta\in[0,1]$.
Hence, there must exist a $\theta\in[0,1]$ such that $\frac{1}{n}\Ent_{p,\hat{p}}(f_{\theta})=\alpha$.
In the following we choose $\theta$ as this value. Similarly, one
can construct a $g_{\tau}=\tau g_{1}+\bar{\tau}g_{2}$ such that $\frac{1}{n}\Ent_{q,\hat{q}}(g_{\tau})=\beta$.
We now require the following basic inequalities: $(\theta f_{1})^{p}+(\bar{\theta}f_{2})^{p}\le(\theta f_{1}+\bar{\theta}f_{2})^{p}\le\theta f_{1}^{p}+\bar{\theta}f_{2}^{p}$
holds for $p\ge1$ and $\theta\in[0,1]$, and the reverse version
with inequality signs reversed holds for $0<p<1$ and $\theta\in[0,1]$.
By these inequalities, $\frac{1}{n}\log\Vert f_{\theta}\Vert_{p}=o(1)$
still holds. Furthermore,  it holds that 
\[
\langle f_{\theta},g_{\tau}\rangle=\theta\tau\langle f_{1},g_{1}\rangle+\bar{\theta}\tau\langle f_{2},g_{1}\rangle+\theta\bar{\tau}\langle f_{1},g_{2}\rangle+\bar{\theta}\bar{\tau}\langle f_{2},g_{2}\rangle\geq\min_{i,j\in[2]}\langle f_{i},g_{j}\rangle.
\]
Combining these with \eqref{eq:-1} yields that $\limsup_{n\to\infty}-\frac{1}{n}\log\frac{\langle f_{\theta},g_{\tau}\rangle}{\Vert f_{\theta}\Vert_{p}\Vert g_{\tau}\Vert_{q}}+\frac{\alpha}{p}+\frac{\beta}{q}$
is upper bounded by $\max_{i,j\in[2]}\underline{\Lambda}_{p,q,\hat{p},\hat{q}}^{*}(\alpha_{i},\beta_{j})+\frac{\alpha-\alpha_{i}}{p}+\frac{\beta-\beta_{j}}{q}$.
 It is not difficult to verify that for finite $\mathcal{X},\mathcal{Y}$,
$\underline{\Lambda}_{p,q,\hat{p},\hat{q}}^{*}(\alpha,\beta)$ is
continuous in $(\alpha,\beta)$. By letting $\alpha_{1},\alpha_{2}\to\alpha$
and $\beta_{1},\beta_{2}\to\beta$, the upper bound above converges
to $\underline{\Lambda}_{p,q,\hat{p},\hat{q}}^{*}(\alpha,\beta)+o(1)$.
This verifies the asymptotic tightness for the case of $p\neq\hat{p},q\neq\hat{q}.$

2) We next assume $p=\hat{p}>0,q=\hat{q}>0$.  Let $\delta>0$ be
fixed. Let $\{T_{X|W}^{*}\}_{n\ge1}$ be a sequence of conditional
types such that $s_{T_{X|W}^{*}}=\alpha+o(1)$, and let $\mu_{T_{X|W}^{*}}=\frac{1}{p}s_{T_{X|W}^{*}}+\delta$.
For $T_{X|W}\neq T_{X|W}^{*}$, we choose $\mu_{T_{X|W}}=\frac{1}{p}s_{T_{X|W}}$.
Then, similarly to the above derivations, for such a choice, 
\begin{align*}
\frac{1}{n}\log\mathbb{E}[f^{p}] & =p\delta+o(1),\\
\frac{1}{n}\Ent_{p}(f) & =\mathbb{E}[\frac{f^{p}}{\mathbb{E}[f^{p}]}\log\frac{f^{p}}{\mathbb{E}[f^{p}]}]=\alpha+o(1),\\
-\frac{1}{n}\log\frac{\langle f,g\rangle}{\Vert f\Vert_{p}\Vert g\Vert_{q}}+\frac{\alpha}{p}+\frac{\beta}{q} & \le\underline{\Lambda}_{p,q,\hat{p},\hat{q}}^{*}(\alpha,\beta)+\delta+o(1).
\end{align*}
By the mixture argument as in the proof for the case of $p\neq\hat{p},q\neq\hat{q}$,
one can prove the asymptotic tightness for this case. Here, $\delta>0$
is arbitrary, and hence, can be removed. 

II. We now prove the case $p,q>0,\hat{p},\hat{q}<0$. Note that for
this case, $\alpha,\beta\le0$. We still choose $f,g$ as in \eqref{eq:g}.
  Let $\{T_{X|W}^{*}\}_{n\ge1}$ be a sequence of conditional types
such that $s_{T_{X|W}^{*}}=\alpha+o(1)$. For this sequence conditional
types, we choose $\mu_{T_{X|W}^{*}}=\frac{1}{\hat{p}}s_{T_{X|W}^{*}}+(\frac{1}{p}-\frac{1}{\hat{p}})\alpha\le0$,
and for $T_{X|W}\neq T_{X|W}^{*}$, choose $\mu_{T_{X|W}}=\frac{1}{p}s_{T_{X|W}}$.
For such a choice, 
\begin{align}
\frac{1}{n}\log\Vert f\Vert_{p} & =o(1),\label{eq:-42-1}\\
\frac{1}{n}\log\Vert f\Vert_{\hat{p}} & =\min_{T_{X|W}}\{\mu_{T_{X|W}}-\frac{1}{\hat{p}}s_{T_{X|W}}\}+o(1)=(\frac{1}{p}-\frac{1}{\hat{p}})\alpha+o(1),\label{eq:-43-1}
\end{align}
and by \eqref{eq:Efg},
\begin{align}
 & -\frac{1}{n}\log\frac{\langle f,g\rangle}{\Vert f\Vert_{p}\Vert g\Vert_{q}}+\frac{\alpha}{p}+\frac{\beta}{q}\nonumber \\
 & \leq-\frac{1}{n}\log\Big(\sum_{T_{X|W}\neq T_{X|W}^{*},T_{Y|W}\neq T_{Y|W}^{*}}P_{XY}^{\otimes n}(\mathcal{T}_{T_{X|W}}(w^{n})\times\mathcal{T}_{T_{Y|W}}(w^{n}))e^{n\mu_{T_{X|W}}+n\nu_{T_{Y|W}}}\Big)+\frac{\alpha}{p}+\frac{\beta}{q}+o(1)\nonumber \\
 & \leq\inf_{Q_{XYW}:Q_{W}=T_{W},Q_{X|W}\neq T_{X|W}^{*},Q_{Y|W}\neq T_{Y|W}^{*}}D(Q_{XY|W}\|P_{XY}|Q_{W})\nonumber \\
 & \qquad+\eta_{p,\hat{p}}(\alpha,D(Q_{X|W}\|P_{X}|Q_{W}))+\eta_{q,\hat{q}}(\beta,D(Q_{Y|W}\|P_{Y}|Q_{W}))+o(1).\label{eq:-61-1}
\end{align}
The formulas \eqref{eq:-42-1} and \eqref{eq:-43-1} imply $\frac{1}{n}\Ent_{p,\hat{p}}(f)=\alpha+o(1)$.
By the continuity of the relative entropy and the denseness of the
types, \eqref{eq:-61-1} implies \eqref{eq:-1}. By the mixture argument
as in the proof for the case of $p\neq\hat{p},q\neq\hat{q}$, we obtain
the asymptotic tightness for the case $p,q>0,\hat{p},\hat{q}<0$.
By symmetry, the asymptotic tightness also holds for the case $\hat{p},\hat{q}>0,p,q<0$. 

III. We now prove the case $p,q,\widehat{p},\widehat{q}<0$. For this
case, $\alpha,\beta\geq0$. We choose $f,g$ as in \eqref{eq:g}.
Let $\{T_{X|W}^{*}\}_{n\ge1}$ be a sequence of conditional types
such that $s_{T_{X|W}^{*}}=\alpha+o(1)$. For this sequence conditional
types, we choose $\mu_{T_{X|W}^{*}}=\frac{1}{p}s_{T_{X|W}^{*}}$,
 and for $T_{X|W}\neq T_{X|W}^{*}$, choose $\mu_{T_{X|W}}=b$ where
$b$ is sufficiently large. Choose $\nu_{T_{Y|W}}$ similarly. For
such a choice, 
\begin{align}
\frac{1}{n}\log\Vert f\Vert_{p} & =\min_{T_{X|W}}\{\mu_{T_{X|W}}-\frac{1}{p}s_{T_{X|W}}\}+o(1)=o(1),\label{eq:-42-1-2}\\
\frac{1}{n}\log\Vert f\Vert_{\hat{p}} & =\min_{T_{X|W}}\{\mu_{T_{X|W}}-\frac{1}{\hat{p}}s_{T_{X|W}}\}+o(1)=(\frac{1}{p}-\frac{1}{\hat{p}})\alpha+o(1),\label{eq:-43-1-2}\\
\frac{1}{n}\Ent_{p,\hat{p}}(f) & =\alpha+o(1)
\end{align}
and by \eqref{eq:Efg},
\begin{align}
 & -\frac{1}{n}\log\frac{\langle f,g\rangle}{\Vert f\Vert_{p}\Vert g\Vert_{q}}+\frac{\alpha}{p}+\frac{\beta}{q}\nonumber \\
 & \leq-\frac{1}{n}\log\Big(\sum_{T_{X|W}\neq T_{X|W}^{*},T_{Y|W}\neq T_{Y|W}^{*}}P_{XY}^{\otimes n}(\mathcal{T}_{T_{X|W}}(w^{n})\times\mathcal{T}_{T_{Y|W}}(w^{n}))e^{2nb}\Big)+\frac{\alpha}{p}+\frac{\beta}{q}+o(1)\nonumber \\
 & \leq-2b+\frac{\alpha}{p}+\frac{\beta}{q}+o(1).\label{eq:-84}
\end{align}
  For fixed $n$, as $b\to\infty$, the upper bound in \eqref{eq:-84}
tends to $-\infty$. By the mixture argument, this proves the asymptotic
tightness for $p,q,\widehat{p},\widehat{q}<0$.

\subsection{\label{subsec:Sharpness-1}Asymptotic Tightness: Reverse Case}

We next prove the asymptotic tightness of the inequality in \eqref{eq:RBLE}.
We assume $\mathcal{X},\mathcal{Y}$ are the supports of $P_{X},P_{Y}$,
respectively. Similarly to the proof of the one-shot bound in Section
\ref{subsec:One-shot-Bound}, we only provide the proofs for the ``symmetric''
cases. For the ``asymmetric'' case, one can prove the asymptotic
tightness by ``mixing'' the proofs for ``symmetric'' cases, since
the ``asymmetric'' case can be seen as a mixture of the ``symmetric''
cases. 

I. We now prove the case $p,q,\hat{p},\hat{q}>0$. 

1) We first assume $p\neq\hat{p},q\neq\hat{q}$. Let $T_{W},T_{X|W},\hat{T}_{X|W},T_{Y|W},\hat{T}_{Y|W}$
be unconditional and conditional types such that $\hat{s}\leq\alpha\leq s$
if $p>\hat{p}$, and $s\le\alpha\le\hat{s}$ if $0<p<\hat{p}$, as
well as, $t\le\beta\le\hat{t}$ if $q>\hat{q}$, and $t\ge\beta\ge\hat{t}$
if $0<q<\hat{q}$, where 
\begin{equation}
s:=D(T_{X|W}\|P_{X}|T_{W}),\hat{s}:=D(\hat{T}_{X|W}\|P_{X}|T_{W}),t:=D(T_{Y|W}\|P_{Y}|T_{W}),\hat{t}:=D(\hat{T}_{Y|W}\|P_{Y}|T_{W}).\label{eq:-83}
\end{equation}
Let $w^{n}$ be a sequence in $\mathcal{W}^{n}$ with $n$-type $T_{W}$.
We choose 
\begin{align*}
f & =e^{n\mu}1_{\mathcal{T}_{T_{X|W}}(w^{n})}+e^{n\hat{\mu}}1_{\mathcal{T}_{\hat{T}_{X|W}}(w^{n})},\quad g=e^{n\nu}1_{\mathcal{T}_{T_{Y|W}}(w^{n})}+e^{n\hat{\nu}}1_{\mathcal{T}_{\hat{T}_{Y|W}}(w^{n})}.
\end{align*}
Then, for such a choice of $f,g$, 
\begin{align*}
\frac{1}{n}\log\Vert f\Vert_{p} & =\frac{1}{p}\max\{p\mu-s,p\hat{\mu}-\hat{s}\}+o(1),\\
\frac{1}{n}\log\Vert f\Vert_{\hat{p}} & =\frac{1}{\hat{p}}\max\{\hat{p}\mu-s,\hat{p}\hat{\mu}-\hat{s}\}+o(1),\\
\Ent_{p,\hat{p}}(f) & =\frac{\hat{p}}{p-\hat{p}}\max\{p\mu-s,p\hat{\mu}-\hat{s}\}+\frac{p}{\hat{p}-p}\max\{\hat{p}\mu-s,\hat{p}\hat{\mu}-\hat{s}\}+o(1).
\end{align*}
Given $(s,\hat{s})$, we choose $\mu=\frac{s}{p}$ and $\hat{\mu}=\frac{\hat{s}}{\hat{p}}+(\frac{1}{p}-\frac{1}{\hat{p}})\alpha$,
which satisfy $\max\{p\mu-s,p\hat{\mu}-\hat{s}\}=0$ and $\max\{\hat{p}\mu-s,\hat{p}\hat{\mu}-\hat{s}\}=(\frac{\hat{p}}{p}-1)\alpha$.
These imply $\frac{1}{n}\log\Vert f\Vert_{p}=o(1)$ and $\Ent_{p,\hat{p}}(f)=\alpha+o(1)$.
On the other hand, 
\begin{align*}
-\frac{1}{n}\log\frac{\langle f,g\rangle}{\Vert f\Vert_{p}\Vert g\Vert_{q}}+\frac{\alpha}{p}+\frac{\beta}{q} & =\min\Big\{\mathbb{D}(T_{X|W},T_{Y|W}\|P_{XY}|T_{W})+\frac{\alpha-s}{p}+\frac{\beta-t}{q},\\
 & \qquad\mathbb{D}(T_{X|W},\hat{T}_{Y|W}\|P_{XY}|T_{W})+\frac{\alpha-s}{p}+\frac{\beta-\hat{t}}{\hat{q}},\\
 & \qquad\mathbb{D}(\hat{T}_{X|W},T_{Y|W}\|P_{XY}|T_{W})+\frac{\alpha-\hat{s}}{\hat{p}}+\frac{\beta-t}{q},\\
 & \qquad\mathbb{D}(\hat{T}_{X|W},\hat{T}_{Y|W}\|P_{XY}|T_{W})+\frac{\alpha-\hat{s}}{\hat{p}}+\frac{\beta-\hat{t}}{\hat{q}}\Big\}+o(1).
\end{align*}
Optimizing $s,t,\hat{s},\hat{t}$, by the continuity of the relative
entropy and by the fact that the set of types is dense in the probability
simplex, we get 
\begin{align*}
-\frac{1}{n}\log\frac{\langle f,g\rangle}{\Vert f\Vert_{p}\Vert g\Vert_{q}}+\frac{\alpha}{p}+\frac{\beta}{q} & \ge\overline{\Lambda}_{p,q,\hat{p},\hat{q}}^{*}(\alpha,\beta)+o(1).
\end{align*}
By the mixture argument as in the proof for the forward case, one
can prove the asymptotic tightness for this reverse case. 

2) We next assume $p=\hat{p}>0,q=\hat{q}>0$. We choose $f=e^{n\mu}1_{\mathcal{T}_{T_{X|W}}(w^{n})}$
and $g=e^{n\nu}1_{\mathcal{T}_{T_{Y|W}}(w^{n})}$ for $T_{X|W}$ and
$T_{Y|W}$ such that $s=\alpha+o(1)$ and $t=\beta+o(1)$ where $s,t$
are given in \eqref{eq:-83}. We choose $\mu=\frac{s}{p}$ and $\nu=\frac{t}{q}$.
Then, 
\begin{align}
\frac{1}{n}\log\mathbb{E}[f^{p}] & =o(1),\label{eq:-47-1-1}\\
\frac{1}{n}\Ent_{p}(f) & =\alpha+o(1),\\
-\frac{1}{n}\log\frac{\langle f,g\rangle}{\Vert f\Vert_{p}\Vert g\Vert_{q}}+\frac{\alpha}{p}+\frac{\beta}{q} & \ge\overline{\Lambda}_{p,q,\hat{p},\hat{q}}^{*}(\alpha,\beta)+o(1).
\end{align}
By the mixture argument as in the proof for the forward case, one
can prove the asymptotic tightness for this reverse case. 

II. We now prove the case $p,q>0,\hat{p},\hat{q}<0$. Note that for
this case, $\alpha,\beta\le0$. We still choose $f,g$ as in \eqref{eq:g}.
  We choose $\mu_{T_{X|W}^{*}}=\frac{1}{p}s_{T_{X|W}^{*}}$ for
some $T_{X|W}^{*}$, and for $T_{X|W}\neq T_{X|W}^{*}$, choose $\mu_{T_{X|W}}=\frac{1}{\hat{p}}s_{T_{X|W}}+(\frac{1}{p}-\frac{1}{\hat{p}})\alpha$.
For such a choice, 
\begin{align}
\frac{1}{n}\log\Vert f\Vert_{p} & =\max_{T_{X|W}}\{\mu_{T_{X|W}}-\frac{1}{p}s_{T_{X|W}}\}+o(1)=o(1),\label{eq:-42-1-1}\\
\frac{1}{n}\log\Vert f\Vert_{\hat{p}} & =\min_{T_{X|W}}\{\mu_{T_{X|W}}-\frac{1}{\hat{p}}s_{T_{X|W}}\}+o(1)=(\frac{1}{p}-\frac{1}{\hat{p}})\alpha+o(1),\label{eq:-43-1-1}\\
\frac{1}{n}\Ent_{p}(f) & =\alpha+o(1),
\end{align}
and by \eqref{eq:Efg},
\begin{align}
 & -\frac{1}{n}\log\frac{\langle f,g\rangle}{\Vert f\Vert_{p}\Vert g\Vert_{q}}+\frac{\alpha}{p}+\frac{\beta}{q}\nonumber \\
 & =\min_{T_{X|W},T_{Y|W}}\min\Big\{\mathbb{D}(T_{X|W}^{*},T_{Y|W}^{*}\|P_{XY}|T_{W})+\frac{\alpha-s}{p}+\frac{\beta-t}{q},\nonumber \\
 & \qquad\mathbb{D}(T_{X|W},T_{Y|W}^{*}\|P_{XY}|T_{W})+\frac{\alpha-\hat{s}}{\hat{p}}+\frac{\beta-t}{q},\nonumber \\
 & \qquad\mathbb{D}(T_{X|W}^{*},T_{Y|W}\|P_{XY}|T_{W})+\frac{\alpha-s}{p}+\frac{\beta-\hat{t}}{\hat{q}},\nonumber \\
 & \qquad\mathbb{D}(T_{X|W},T_{Y|W}\|P_{XY}|T_{W})+\frac{\alpha-\hat{s}}{\hat{p}}+\frac{\beta-\hat{t}}{\hat{q}}\Big\}+o(1).\label{eq:-85}
\end{align}
Since $T_{X|W}^{*},T_{Y|W}^{*}$ are arbitrary, taking supremum for
\eqref{eq:-85} over all $T_{X|W}^{*},T_{Y|W}^{*}$ and by the mixture
argument to ensure $\frac{1}{n}\Ent_{p}(f)=\alpha$ and $\frac{1}{n}\Ent_{q}(g)=\beta$,
we obtain the asymptotic tightness for the case $p,q>0,\hat{p},\hat{q}<0$.
By symmetry, the asymptotic tightness also holds for the case $\hat{p},\hat{q}>0,p,q<0$. 

III. We now prove the case $p,q,\widehat{p},\widehat{q}<0$. Without
loss of generality, we assume $p\le\widehat{p},q\le\widehat{q}$.
For this case, $\alpha,\beta\geq0$. We still choose $f,g$ as in
\eqref{eq:g}.   We choose 
\begin{equation}
\mu_{T_{X|W}}=\max\{\frac{1}{p}s_{T_{X|W}},\,\frac{1}{\hat{p}}s_{T_{X|W}}+(\frac{1}{p}-\frac{1}{\hat{p}})\alpha\}\textrm{ for all }T_{X|W}.\label{eq:-37-4}
\end{equation}
Then, 
\begin{align*}
\frac{1}{n}\log\Vert f\Vert_{p} & =\min_{T_{X|W}}\{\mu_{T_{X|W}}-\frac{1}{p}s_{T_{X|W}}\}+o(1)=o(1),\\
\frac{1}{n}\log\Vert f\Vert_{\hat{p}} & =\min_{T_{X|W}}\{\mu_{T_{X|W}}-\frac{1}{\hat{p}}s_{T_{X|W}}\}+o(1)=(\frac{1}{p}-\frac{1}{\hat{p}})\alpha+o(1),\\
\frac{1}{n}\Ent_{p,\hat{p}}(f) & =\alpha+o(1),
\end{align*}
and by \eqref{eq:Efg},
\begin{align*}
 & -\frac{1}{n}\log\frac{\langle f,g\rangle}{\Vert f\Vert_{p}\Vert g\Vert_{q}}+\frac{\alpha}{p}+\frac{\beta}{q}\\
 & =\min_{T_{X|W}}\min\{\frac{\alpha-s_{T_{X|W}}}{p},\frac{\alpha-s_{T_{X|W}}}{\hat{p}}\}+\min_{T_{Y|W}}\min\{\frac{\beta-s_{T_{Y|W}}}{q},\frac{\beta-s_{T_{Y|W}}}{\hat{q}}\}+o(1)\\
 & =\min\{\frac{\alpha}{p},\frac{\alpha}{\hat{p}}\}+\min\{\frac{\beta}{q},\frac{\beta}{\hat{q}}\}+o(1)\\
 & =\underline{\Lambda}_{p,q,\hat{p},\hat{q}}^{*}(\alpha,\beta)+o(1).
\end{align*}
By the mixture argument, this proves the asymptotic tightness for
$p,q,\widehat{p},\widehat{q}<0$. 

\section{Alternative Proof of Theorem \ref{thm:strongBL}}

Respectively define the {\em forward } and {\em reverse BL exponents}
as 
\begin{align}
\underline{\Lambda}_{p,q}(X;Y) & :=-\log\sup_{f:\mathcal{X}\to[0,\infty),g:\mathcal{Y}\to[0,\infty):\Vert f\Vert_{p}\Vert g\Vert_{q}>0}\frac{\langle f,g\rangle}{\Vert f\Vert_{p}\Vert g\Vert_{q}}\quad\mbox{and}\label{eq:FI}\\
\overline{\Lambda}_{p,q}(X;Y) & :=-\log\inf_{f:\mathcal{X}\to[0,\infty),g:\mathcal{Y}\to[0,\infty):\Vert f\Vert_{p}\Vert g\Vert_{q}>0}\frac{\langle f,g\rangle}{\Vert f\Vert_{p}\Vert g\Vert_{q}}.\label{eq:FI-1}
\end{align}

It is well-known that $\underline{\Lambda}_{p,q},\overline{\Lambda}_{p,q}$
have the following information-theoretic characterizations. 
\begin{thm}
\label{thm:ITcharacterization-1} Let $\mathcal{X}$ and $\mathcal{Y}$
be two Polish spaces. For $p,q\in\mathbb{R}\backslash\{0\}$, if $(X,Y)\sim P_{XY}$,
then 
\begin{equation}
\underline{\Lambda}_{p,q}(X;Y)=\begin{cases}
{\displaystyle \inf_{s,t\ge0}\underline{\varphi}(s,t)-\frac{s}{p}-\frac{t}{q}} & p,q>0\vspace{.03in}\\
{\displaystyle -\infty} & p<0\textrm{ or }q<0
\end{cases}\label{eq:FIInfChLambdaUnderline}
\end{equation}
and 
\begin{equation}
\overline{\Lambda}_{p,q}(X;Y)=\begin{cases}
{\displaystyle \sup_{s,t\ge0}\overline{\varphi}(s,t)-\frac{s}{p}-\frac{t}{q}} & p,q>0\vspace{.03in}\\
{\displaystyle \sup_{s\ge0}\varphi_{q}(s)-\frac{s}{p}} & q<0<p\vspace{.03in}\\
{\displaystyle \sup_{t\ge0}\varphi_{p}(s)-\frac{t}{q}} & p<0<q\vspace{.03in}\\
{\displaystyle 0} & p,q\!<\!0
\end{cases}.\label{eq:FIInfChLambdaOverline}
\end{equation}
\end{thm}

For Euclidean spaces, the forward part of this theorem, i.e., \eqref{eq:FIInfChLambdaUnderline},
was derived in \cite{carlen2009subadditivity}. The reverse part of
this theorem, i.e., \eqref{eq:FIInfChLambdaOverline}, for finite
alphabets was derived in \cite{beigi2016equivalent} for all $p,q\neq0$,
and also in \cite{liu2016brascamp} for $p,q>0$. The extension of
these characterizations to Polish spaces was studied in \cite{liu2018information},
but some compactness conditions are required especially for the reverse
part in \eqref{eq:FIInfChLambdaOverline}. Theorem \ref{thm:ITcharacterization-1}
(for any Polish spaces) can be proven by using the results in this
paper. The forward part \eqref{eq:FIInfChLambdaUnderline} can be
easily proven by the duality in Proposition \ref{prop:ITcharacterization}
by swapping two infimizations, and $\le$ in \eqref{eq:FIInfChLambdaOverline}
follows by Proposition \ref{prop:ITcharacterization}, and $\ge$
in \eqref{eq:FIInfChLambdaOverline} follows by the tensorization
property and the strong small-set expansion theorem (Theorem \ref{thm:ITcharacterization-1})
and the strong $q$-stability theorem (Theorem \ref{thm:strongqstability}).

Based on the theorem above, we are ready to prove Theorem \ref{thm:strongBL}.
Combined with the tensorization property of the BL exponents, Theorem
\ref{thm:ITcharacterization-1} implies that for all functions $f:\mathcal{X}^{n}\to[0,\infty),g:\mathcal{Y}^{n}\to[0,\infty)$,
\begin{align}
 & -\frac{1}{n}\log\frac{\langle f,g\rangle}{\Vert f\Vert_{p}\Vert g\Vert_{q}}-\frac{1}{n}\log\frac{\Vert f\Vert_{p}}{\Vert f\Vert_{\hat{p}}}-\frac{1}{n}\log\frac{\Vert g\Vert_{q}}{\Vert g\Vert_{\hat{q}}}\nonumber \\
 & \quad\ge\inf_{s,t\ge0}\underline{\varphi}(s,t)-\frac{s}{\hat{p}}-\frac{t}{\hat{q}}\ge\inf_{s,t\ge0}\underline{\Theta}(s,t)-\frac{s}{\hat{p}}-\frac{t}{\hat{q}},\label{eq:FIBLFSSE-1}
\end{align}
where $\underline{\varphi}$ and $\underline{\Theta}$ are defined
in \eqref{eq:phiUnderline} and \eqref{eq:ThetaUnderline} respectively.
For any $(f,g)$, put $a=\frac{1}{n}\Ent_{p,\hat{p}}(f)$ and $b=\frac{1}{n}\Ent_{q,\hat{q}}(g)$.
Since $(\frac{1}{p^{*}},\frac{1}{q^{*}})$ is a subgradient of $\underline{\Theta}$
at the point $(\alpha,\beta)$, by the definition of subgradients,
we have 
\begin{equation}
\inf_{s,t\ge0}\underline{\Theta}(s,t)-\frac{s}{p^{*}}-\frac{t}{q^{*}}=\underline{\Theta}(\alpha,\beta)-\frac{\alpha}{p^{*}}-\frac{\beta}{q^{*}}.
\end{equation}
Substituting $\hat{p}=p^{*}$ and $\hat{q}=q^{*}$ into \eqref{eq:FIBLFSSE-1},
we have 
\begin{equation}
-\frac{1}{n}\log\frac{\langle f,g\rangle}{\Vert f\Vert_{p}\Vert g\Vert_{q}}\ge\underline{\Theta}(\alpha,\beta)-\frac{\alpha}{p^{*}}-\frac{\beta}{q^{*}}+(\frac{1}{p^{*}}-\frac{1}{p})a+(\frac{1}{q^{*}}-\frac{1}{q})b.
\end{equation}
By assumption, $(\frac{1}{p^{*}}-\frac{1}{p})(a-\alpha)\ge0$ and
$(\frac{1}{q^{*}}-\frac{1}{q})(b-\beta)\ge0$. Hence, 
\begin{equation}
-\frac{1}{n}\log\frac{\langle f,g\rangle}{\Vert f\Vert_{p}\Vert g\Vert_{q}}\ge\underline{\Theta}(\alpha,\beta)-\frac{\alpha}{p}-\frac{\beta}{q},
\end{equation}
i.e., the inequality in \eqref{eq:FBL-1} holds.

The inequalities in \eqref{eq:RBL-1} and \eqref{eq:RBL-2} can be
proven similarly. We omit the proofs.

\begin{remark} Although the method above can be used to prove the
special version of the strong BL inequalities in Theorem \ref{thm:strongBL},
this method seems to fail to prove the general version in Theorem
\ref{thm:BLexponent}. \end{remark}

\section{\label{sec:Proof-of-Lemma-RenyiSDPI}Proof of Lemma \ref{lem:RenyiSDPI}}

We may assume, by homogeneilty, that $\Vert f\Vert_{1}=1$. Hence
we can write $f=\frac{\mathrm{d}Q_{X}}{\mathrm{d}P_{X}}$. This choice
implies that 
\begin{align*}
\log\Vert f\Vert_{p} & =\frac{1}{p}\log\int(\frac{\mathrm{d}Q_{X}}{\mathrm{d}P_{X}})^{p}\mathrm{d}P_{X}=\frac{1}{p'}D_{p}(Q_{X}\|P_{X})
\end{align*}
and 
\begin{align*}
\log\Vert P_{X|Y}(f)\Vert_{q} & =\frac{1}{q}\log\int(\int\frac{\mathrm{d}Q_{X}}{\mathrm{d}P_{X}}\mathrm{d}P_{X|Y})^{q}\mathrm{d}P_{Y}\\
 & =\frac{1}{q}\log\int(\frac{\mathrm{d}Q_{Y}}{\mathrm{d}P_{Y}})^{q}\mathrm{d}P_{Y}=\frac{1}{q'}D_{q}(Q_{Y}\|P_{Y})
\end{align*}
where $Q_{Y}:=Q_{X}\circ P_{Y|X}$ and by Bayes's theorem, $\int\frac{\mathrm{d}Q_{X}}{\mathrm{d}P_{X}}\mathrm{d}P_{X|Y}=\int\frac{\mathrm{d}P_{X|Y}}{\mathrm{d}P_{X}}\mathrm{d}Q_{X}=\int\frac{\mathrm{d}P_{Y|X}}{\mathrm{d}P_{Y}}\mathrm{d}Q_{X}=\frac{\mathrm{d}Q_{Y}}{\mathrm{d}P_{Y}}$,
$P_{Y}$-a.e. Therefore, 
\begin{align}
\underline{\Gamma}_{p,q,1}(\alpha|P_{XY})-\frac{\alpha}{p} & =\inf_{Q_{X}:D_{p}(Q_{X}\|P_{X})=\alpha}-\log\frac{\Vert P_{X|Y}(f)\Vert_{q}}{\Vert f\Vert_{p}}\nonumber \\
 & =\inf_{Q_{X}:D_{p}(Q_{X}\|P_{X})=\alpha}-\frac{1}{q'}D_{q}(Q_{Y}\|P_{Y})+\frac{1}{p'}D_{p}(Q_{X}\|P_{X})\nonumber \\
 & =\inf_{Q_{X}:D_{p}(Q_{X}\|P_{X})=\alpha}-\frac{1}{q'}D_{q}(Q_{Y}\|P_{Y})+\frac{\alpha}{p'},\label{eq:-24}
\end{align}
and similarly, 
\begin{align}
\overline{\Gamma}_{p,q,1}(\alpha|P_{XY})-\frac{\alpha}{p} & =\sup_{Q_{X}:D_{p}(Q_{X}\|P_{X})=\alpha}-\frac{1}{q'}D_{q}(Q_{Y}\|P_{Y})+\frac{\alpha}{p'}.\label{eq:-25}
\end{align}
The equalities \eqref{eq:-24} and \eqref{eq:-25}  imply \eqref{eq:-88}
and \eqref{eq:-89}.

\section{\label{sec:Proof-of-Theorem-strongsse}Proof of Theorem \ref{thm:strongsse}}

The inequalities \eqref{eq:-15} for all $\alpha,\beta$ and \eqref{eq:-64}
for $\alpha,\beta>0$ follow from the strong BL inequalities in Theorem
\ref{thm:strongBL} directly. By definition, one can easily obtain
that $\overline{\Theta}^{(n)}(\alpha,\beta)=\begin{cases}
\alpha & \beta=0\\
\beta & \alpha=0
\end{cases}.$ We next prove \eqref{eq:-78} for all $\alpha,\beta$ and \eqref{eq:-79}
for $\alpha,\beta>0$.

Let $\mathtt{L}$ be a metric (e.g., the Lévy--Prokhorov metric)
on $\mathcal{P}(\mathcal{Z})$ compatible with the weak topology on
$\mathcal{P}(\mathcal{Z})$, where $\mathcal{Z}$ is $\mathcal{X}$
or $\mathcal{Y}$. In the following, we use $B_{\delta}(R):=\{Q\in\mathcal{P}(\mathcal{Z}):\mathtt{L}(R,Q)<\delta\}$
and $B_{\leq\delta}(R):=\{Q\in\mathcal{P}(\mathcal{Z}):\mathtt{L}(R,Q)\leq\delta\}$
to respectively denote an open ball and an closed ball. We use $\overline{A}$,
$A^{o}$, and $A^{c}$ to respectively denote the closure, interior,
and complement of the set $A$.

We first prove \eqref{eq:-78} for $\alpha\in[0,\alpha_{\max}),\beta\in[0,\beta_{\max})$.
For $\epsilon>0$, let 
\begin{align*}
\mathcal{Q}_{1} & =\{Q_{X}:D(Q_{X}\|P_{X})>\alpha+\epsilon\},\quad\mathcal{Q}_{2}=\{Q_{Y}:D(Q_{Y}\|P_{Y})>\beta+\epsilon\}.
\end{align*}
Let $(R_{X},R_{Y})\in\mathcal{Q}_{1}\times\mathcal{Q}_{2}$ be a pair
of distributions. Since $\mathcal{Q}_{1}$ is open under the weak
topology \cite{Dembo}, there exists $\delta_{1}>0$ such that the
closed ball $B_{\le\delta_{1}}(R_{X})\subseteq\mathcal{Q}_{1}$.
Similarly, there exists $\delta_{2}>0$ such that $B_{\le\delta_{2}}(R_{Y})\subseteq\mathcal{Q}_{2}$.
Let $\delta=\min\{\delta_{1},\delta_{2}\}$. Then, $B_{\le\delta}(R_{X})\subseteq\mathcal{Q}_{1}$
and $B_{\le\delta}(R_{Y})\subseteq\mathcal{Q}_{2}$.

Define 
\begin{align*}
A_{n} & =\bigcup_{T_{X}\in B_{\le\delta}(R_{X})}\mathcal{T}_{T_{X}}\quad\textrm{ and }\quad B_{n}=\bigcup_{T_{Y}\in B_{\le\delta}(R_{Y})}\mathcal{T}_{T_{Y}},
\end{align*}
where $\mathcal{T}_{T_{X}},\mathcal{T}_{T_{Y}}$ are type classes;
see the definition at the beginning of Appendix \ref{subsec:Sharpness}.
Then by Sanov's theorem \cite[Theorem 6.2.10]{Dembo}, we have 
\begin{align*}
\liminf_{n\to\infty}-\frac{1}{n}\log P_{X}^{\otimes n}(A_{n}) & \geq\inf_{Q_{X}\in B_{\le\delta}(R_{X})}D(Q_{X}\|P_{X})\ge\alpha+\epsilon,\\
\liminf_{n\to\infty}-\frac{1}{n}\log P_{Y}^{\otimes n}(B_{n}) & \geq\inf_{Q_{Y}\in B_{\le\delta}(R_{Y})}D(Q_{Y}\|P_{Y})\ge\beta+\epsilon,
\end{align*}
and 
\begin{align}
\limsup_{n\to\infty}-\frac{1}{n}\log P_{XY}^{\otimes n}(A_{n}\times B_{n}) & \leq\inf_{Q_{XY}\in\mathcal{C}_{\le\delta}(R_{X},R_{Y})^{o}}D(Q_{XY}\|P_{XY}),\label{eq:-86}
\end{align}
where 
\begin{equation}
\mathcal{C}_{\le\delta}(R_{X},R_{Y}):=\bigcup_{(Q_{X},Q_{Y})\in B_{\le\delta}(R_{X})\times B_{\le\delta}(R_{Y})}\mathcal{C}(Q_{X},Q_{Y}).\label{eq:-87}
\end{equation}
Obviously, $\mathcal{C}_{\le\delta}(R_{X},R_{Y})\supseteq\mathcal{C}_{\delta}(R_{X},R_{Y}):=\bigcup_{(Q_{X},Q_{Y})\in B_{\delta}(R_{X})\times B_{\delta}(R_{Y})}\mathcal{C}(Q_{X},Q_{Y})$.
Moreover, we have the following properties on the union of coupling
sets. 

1) If $E$ and $F$ are respectively closed subsets of $\mathcal{P}(\mathcal{X})$
and $\mathcal{P}(\mathcal{Y})$ (under the corresponding weak topology),
then $\mathcal{C}(E,F):=\bigcup_{(Q_{X},Q_{Y})\in E\times F}\mathcal{C}(Q_{X},Q_{Y})$
is closed in the space of $\mathcal{P}(\mathcal{X}\times\mathcal{Y})$
(under the corresponding weak topology). 

2)  The property 1) still holds if we replace ``closed'' with ``open''
in the both assumption and conclusion. 

The property 1) follows by the following fact. For a sequence of probability
measures $\left\{ Q_{XY}^{(n)}\right\} $ on the product Polish space
$\mathcal{X}\times\mathcal{Y}$, $Q_{XY}^{(n)}\to Q_{XY}$ implies
$Q_{X}^{(n)}\to Q_{X}$ and $Q_{Y}^{(n)}\to Q_{Y}$, where $\rightarrow$
denotes the weak convergence. By this fact, for any sequence $\left\{ Q_{XY}^{(n)}\right\} $
from the set $\mathcal{C}(E,F)$ such that $Q_{XY}^{(n)}\to Q_{XY}$,
we have $Q_{X}^{(n)}\to Q_{X}$ and $Q_{Y}^{(n)}\to Q_{Y}$. Moreover,
since $E,F$ are closed, we have $Q_{X}\in E,Q_{Y}\in F$, which in
turn implies $Q_{XY}\in\mathcal{C}(E,F)$, i.e., $\mathcal{C}(E,F)$
is closed. The property 2) follows from the property 1). This is because
by the property 1), for open $E,F$, we have that $\mathcal{C}(E^{c},\mathcal{P}(\mathcal{Y})),\mathcal{C}(\mathcal{P}(\mathcal{X}),F^{c})$
are closed, which implies $\mathcal{C}(E^{c},\mathcal{P}(\mathcal{Y}))\cup\mathcal{C}(\mathcal{P}(\mathcal{X}),F^{c})$
is closed as well, and hence, $\mathcal{C}(E,F)=[\mathcal{C}(E^{c},\mathcal{P}(\mathcal{Y}))\cup\mathcal{C}(\mathcal{P}(\mathcal{X}),F^{c})]^{c}$
is open. 

By the property 2), $\mathcal{C}_{\delta}(R_{X},R_{Y})$ is open,
and hence, $\mathcal{C}_{\le\delta}(R_{X},R_{Y})^{o}\supseteq\mathcal{C}_{\delta}(R_{X},R_{Y})$.
Therefore, the RHS of \eqref{eq:-86} is further upper bounded by
$\inf_{Q_{XY}\in\mathcal{C}_{\delta}(R_{X},R_{Y})}D(Q_{XY}\|P_{XY})\leq\mathbb{D}(R_{X},R_{Y}\|P_{XY})$.
Since $(R_{X},R_{Y})$ is arbitrary, we optimize over all $(R_{X},R_{Y})\in\mathcal{Q}_{1}\times\mathcal{Q}_{2}$
and obtain that 
\begin{align*}
\underline{\Theta}^{(\infty)}(\alpha,\beta) & \le\inf_{\substack{Q_{XY}:D(Q_{X}\|P_{X})>\alpha+\epsilon,\\
D(Q_{Y}\|P_{Y})>\beta+\epsilon
}
}D(Q_{XY}\|P_{XY}).
\end{align*}

By the time-sharing argument, 
\begin{align*}
\underline{\Theta}^{(\infty)}(\alpha,\beta) & \leq\inf_{\substack{Q_{XYW}:D(Q_{X|W}\|P_{X}|Q_{W})>\alpha+\epsilon,\\
D(Q_{Y|W}\|P_{Y}|Q_{W})>\beta+\epsilon
}
}D(Q_{XY|W}\|P_{XY}|Q_{W})\le\underline{\Theta}(\alpha+2\epsilon,\beta+2\epsilon).
\end{align*}
Letting $\epsilon\downarrow0$, we obtain $\underline{\Theta}^{(\infty)}(\alpha,\beta)\leq\lim_{\epsilon\downarrow0}\underline{\Theta}(\alpha+2\epsilon,\beta+2\epsilon)$.
Since $t\in(0,\min\{\alpha_{\max}-\alpha,\beta_{\max}-\beta\})\mapsto\underline{\Theta}(\alpha+t,\beta+t)$
is convex and nondecreasing, it is right continuous at $t=0$. Hence,
$\underline{\Theta}^{(\infty)}(\alpha,\beta)\le\underline{\Theta}(\alpha,\beta)$,
which, combined with \eqref{eq:-15}, implies that $\underline{\Theta}^{(\infty)}(\alpha,\beta)=\underline{\Theta}(\alpha,\beta)$.

We next prove \eqref{eq:-78} for $\alpha=\alpha_{\max}<\infty$ and
$\beta\in[0,\beta_{\max})$. Since $\alpha_{\max}$ is finite, $P_{X}$
is supported on a finite set; see the paragraph below \eqref{eq:amax}.
 Denote  $\hat{\mathcal{X}}:=\{x\in\mathcal{X}:-\log P_{X}(x)=\alpha_{\max}\}$.
Then, $-\frac{1}{n}\log P_{X}^{\otimes n}(\{x^{n}\})\ge\alpha_{\max}$
is equivalent to $x^{n}\in\hat{\mathcal{X}}^{n}$.  Therefore, 
\begin{align}
\underline{\Theta}^{(n)}(\alpha_{\max},\beta) & =-\frac{1}{n}\log\sup_{\substack{x^{n}\in\hat{\mathcal{X}}^{n},B\in\mathbb{B}_{\mathcal{Y}}^{\otimes n}:P_{Y}^{\otimes n}(B)\leq e^{-n\beta}}
}P_{XY}^{\otimes n}(\{x^{n}\}\times B)\nonumber \\
 & =\alpha_{\max}-\frac{1}{n}\log\sup_{\substack{B\in\mathbb{B}_{\mathcal{Y}}^{\otimes n}:P_{Y}^{\otimes n}(B)\leq e^{-n\beta}}
}(\prod_{i=1}^{n}P_{Y|X=x_{i}})(B).\label{eq:-81-1}
\end{align}
By an argument similar to the one in the case of $\alpha\in[0,\alpha_{\max}),\beta\in[0,\beta_{\max})$,
and by the time-sharing argument, we have 
\begin{align}
\limsup_{n\to\infty}-\frac{1}{n}\log\sup_{\substack{B\in\mathbb{B}_{\mathcal{Y}}^{\otimes n}:P_{Y}^{\otimes n}(B)\leq e^{-n\beta}}
}(\prod_{i=1}^{n}P_{Y|X=x_{i}})(B) & \le\inf_{\substack{R_{YXW}:R_{X}(\hat{\mathcal{X}}^{c})=0,\\
\,D(R_{Y|XW}\|P_{Y}|R_{XW})>\beta+\epsilon
}
}D(R_{Y|XW}\|P_{Y|X}|R_{XW}),\label{eq:-80-1}
\end{align}
Substituting \eqref{eq:-80-1} into \eqref{eq:-81-1} and taking limit
yield that $\underline{\Theta}^{(\infty)}(\alpha_{\max},\beta)$ is
upper bounded by the sum of $\alpha_{\max}$ and the RHS of \eqref{eq:-81-1}.
 It is not difficult to see that this sum is further upper bounded
by $\underline{\Theta}(\alpha_{\max},\beta+\epsilon)$, and then,
by the continuity (which is implied by the convexity), we finally
have $\underline{\Theta}^{(\infty)}(\alpha_{\max},\beta)\le\underline{\Theta}(\alpha_{\max},\beta)$. 

By symmetry, \eqref{eq:-78} also holds for the case of $\alpha\in[0,\alpha_{\max}),\beta=\beta_{\max}<\infty$.
As for the case of $\alpha=\alpha_{\max}<\infty,\beta=\beta_{\max}<\infty$,
in fact, it has already been proven in \cite{yu2021Graphs}, since
in this case, $P_{XY}$ is finitely supported. Combining all the cases
above, we complete the proof of \eqref{eq:-78}.

We next prove \eqref{eq:-79} for $\alpha,\beta>0$. For $\epsilon>0$,
let $(R_{X},R_{Y})$ be a pair of distributions such that $D(R_{X}\|P_{X})\le\alpha-\epsilon,\,D(R_{Y}\|P_{Y})\le\beta-\epsilon.$
For $\delta>0$, define 
\begin{align*}
A_{n} & =\bigcup_{T_{X}\in B_{\delta}(R_{X})}\mathcal{T}_{T_{X}},\qquad B_{n}=\bigcup_{T_{Y}\in B_{\delta}(R_{Y})}\mathcal{T}_{T_{Y}}.
\end{align*}
Then by Sanov's theorem \cite[Theorem 6.2.10]{Dembo}, we have
\begin{align*}
\limsup_{n\to\infty}-\frac{1}{n}\log P_{X}^{\otimes n}(A_{n}) & \leq\inf_{Q_{X}\in B_{\delta}(R_{X})}D(Q_{X}\|P_{X})\le D(R_{X}\|P_{X}),\\
\limsup_{n\to\infty}-\frac{1}{n}\log P_{Y}^{\otimes n}(B_{n}) & \leq\inf_{Q_{Y}\in B_{\delta}(R_{Y})}D(Q_{Y}\|P_{Y})\le D(R_{Y}\|P_{Y}),
\end{align*}
and 
\begin{align*}
\liminf_{n\to\infty}-\frac{1}{n}\log P_{XY}^{\otimes n}(A_{n}\times B_{n}) & \ge\inf_{Q_{XY}\in\overline{\mathcal{C}_{\delta}(R_{X},R_{Y})}}D(Q_{XY}\|P_{XY}).
\end{align*}
Obviously, $\mathcal{C}_{\delta}(R_{X},R_{Y})\subseteq\mathcal{C}_{\le\delta}(R_{X},R_{Y})$.
By the property 1) above, $\mathcal{C}_{\le\delta}(R_{X},R_{Y})$
is closed. Hence, $\overline{\mathcal{C}_{\delta}(R_{X},R_{Y})}\subseteq\mathcal{C}_{\le\delta}(R_{X},R_{Y})$,
which further implies that 
\begin{align}
\liminf_{n\to\infty}-\frac{1}{n}\log P_{XY}^{\otimes n}(A_{n}\times B_{n}) & \ge g(\delta):=\inf_{Q_{XY}\in\mathcal{C}_{\le\delta}(R_{X},R_{Y})}D(Q_{XY}\|P_{XY}).\label{eq:-49}
\end{align}
By definition, $g(0)=\mathbb{D}(R_{X},R_{Y}\|P_{XY})$, and $g(\delta)$
is nonincreasing for $\delta\ge0$. Hence, $g(\delta)\le g(0).$ We
next prove the lower semi-continuity of $g$ at $\delta=0$, i.e.,
$\lim_{\delta\downarrow0}g(\delta)\ge g(0)$, which, combined with
$g(\delta)\le g(0)$, yields the continuity of $g$ at $\delta=0$.

Let $\tau>0$. For any $\delta>0$, one can find a distribution $Q_{XY}^{(\delta,\tau)}$
such that $Q_{X}^{(\delta,\tau)}\in B_{\le\delta}(R_{X})$, $Q_{Y}^{(\delta,\tau)}\in B_{\le\delta}(R_{Y})$,
and $D(Q_{XY}^{(\delta,\tau)}\|P_{XY})\le g(\delta)+\tau$. Hence,
for any $\delta>0$, $Q_{XY}^{(\delta,\tau)}\in\mathcal{Q}_{\tau}:=\{Q_{XY}:D(Q_{XY}\|P_{XY})\le g(0)+\tau\}$.
For any positive sequence $\{\delta_{i}\}$ such that $\delta_{i}\to0$
as $i\to\infty$, $\{Q_{XY}^{(\delta_{i},\tau)}\}_{i\in\mathbb{N}}$
is a sequence from $\mathcal{Q}_{\tau}$. By the compactness of $\mathcal{Q}_{\tau}$
\cite{Dembo,Erven}, one can find a convergent subsequence of $\{Q_{XY}^{(\delta_{i},\tau)}\}_{i\in\mathbb{N}}$,
denoted as $\{Q_{XY}^{(\delta_{i_{k}},\tau)}\}_{i\in\mathbb{N}}$,
and whose limit is denoted as $Q_{XY}^{*}$. Since $Q_{XY}^{(\delta_{i_{k}},\tau)}\to Q_{XY}^{*}$
as $k\to\infty$ implies that $Q_{X}^{(\delta_{i_{k}},\tau)}\to Q_{X}^{*},\,Q_{Y}^{(\delta_{i_{k}},\tau)}\to Q_{Y}^{*}$.
On the other hand, as $\delta_{i}\to0$, $Q_{X}^{(\delta_{i_{k}},\tau)}\to R_{X},\,Q_{Y}^{(\delta_{i_{k}},\tau)}\to R_{Y}$.
Hence, $Q_{XY}^{*}\in\mathcal{C}(R_{X},R_{Y})$. Moreover, by the
lower semi-continuity of the relative entropy under the weak topology,
we have $\liminf_{k\to\infty}D(Q_{XY}^{(\delta_{i_{k}},\tau)}\|P_{XY})\ge D(Q_{XY}^{*}\|P_{XY})\ge g(0)$.
Furthermore, by the monotone convergence theorem, $\lim_{\delta\downarrow0}g(\delta)$
exists (here we allow limits to be equal to $\infty$). Therefore,
$\lim_{\delta\downarrow0}g(\delta)=\lim_{k\to\infty}g(\delta_{i_{k}})\ge g(0)-\tau$.
Since $\tau>0$ is arbitrary, we have $\lim_{\delta\downarrow0}g(\delta)\ge g(0)$,
i.e., the lower semi-continuity of $g$ at $\delta=0$. 

Substituting this lower continuity into \eqref{eq:-49} yields that
$\liminf_{n\to\infty}-\frac{1}{n}\log P_{XY}^{\otimes n}(A_{n}\times B_{n})\ge g(0).$
Since $(R_{X},R_{Y})$ is arbitrary such that $D(R_{X}\|P_{X})\le\alpha-\epsilon,D(R_{Y}\|P_{Y})\le\beta-\epsilon$
and $\epsilon>0$ is also arbitrary, we have 
\begin{align}
\overline{\Theta}^{(\infty)}(\alpha,\beta) & \ge\limsup_{\epsilon\downarrow0}\sup_{R_{X},R_{Y}:D(R_{X}\|P_{X})\le\alpha-\epsilon,D(R_{Y}\|P_{Y})\le\beta-\epsilon}\mathbb{D}(R_{X},R_{Y}\|P_{XY}).\label{eq:-8}
\end{align}
By the time-sharing argument, we have $\overline{\Theta}^{(\infty)}(\alpha,\beta)\geq\lim_{\epsilon\downarrow0}\overline{\Theta}(\alpha-\epsilon,\beta-\epsilon).$
Since $t\in(0,\min\{\alpha,\beta\})\mapsto\overline{\Theta}(\alpha-t,\beta-t)$
is concave and nondecreasing, it is left continuous at $t=0$. Hence,
$\overline{\Theta}^{(\infty)}(\alpha,\beta)\geq\overline{\Theta}(\alpha,\beta)$,
which, combined with \eqref{eq:-64}, implies that $\overline{\Theta}^{(\infty)}(\alpha,\beta)=\overline{\Theta}(\alpha,\beta)$.

\section{\label{sec:Proof-of-Theorem-strongsse-1}Proof of Theorem \ref{thm:strongqstability}}

The strong BL inequalities in Theorem \ref{thm:strongBL_single} imply
\eqref{eq:-69} and \eqref{eq:-63}. We next prove the asymptotic
sharpness of the inequalities in \eqref{eq:-63} for $q\ge1$ and
\eqref{eq:-69} for $q<0$, which follows from the equation in \eqref{eq:equivalence}
and the asymptotic sharpness of the inequalities in Theorem \ref{thm:strongsse-2}.

Setting $\hat{g}\leftarrow P_{X|Y}^{\otimes n}(f)$ in the $n$-dimensional
version of \eqref{eq:equivalence}, we obtain 
\[
\Vert P_{X|Y}^{\otimes n}(f|\cdot)\Vert_{q}=\begin{cases}
\sup_{g\ge0}\frac{\langle f,g\rangle}{\Vert g\Vert_{q'}} & q\ge1\\
\inf_{g\ge0}\frac{\langle f,g\rangle}{\Vert g\Vert_{q'}} & q\le1
\end{cases}
\]
By Theorem \ref{thm:strongsse-2}, for any $\alpha\in[0,\alpha_{\max}],\beta\in[0,\beta_{\max}]$,
there exists a sequence of $(A_{n},B_{n})$ such that $-\frac{1}{n}\log P_{X}^{\otimes n}({A}_{n})\downarrow\alpha,-\frac{1}{n}\log P_{Y}^{\otimes n}({B}_{n})\downarrow\beta$,
and $-\frac{1}{n}\log P_{XY}^{\otimes n}({A}_{n}\times{B}_{n})\to\underline{\Theta}(\alpha,\beta)$
as $n\to\infty$. Setting $f=1_{A_{n}},g=1_{B_{n}}$, we obtain for
$q\ge1$, 
\begin{align*}
\limsup_{n\to\infty}-\frac{1}{n}\log\Vert P_{X|Y}^{\otimes n}(A_{n}|\cdot)\Vert_{q} & \le\limsup_{n\to\infty}-\frac{1}{n}\log\frac{P_{XY}^{\otimes n}(A_{n}\times B_{n})}{P_{Y}^{\otimes n}(B_{n})^{1/q'}}=\underline{\Theta}(\alpha,\beta)-\frac{1}{q'}\beta.
\end{align*}
Since $\beta\in[0,\beta_{\max}]$ is arbitrary, we have 
\begin{align*}
\limsup_{n\to\infty}-\frac{1}{n}\log\Vert P_{X|Y}^{\otimes n}(A_{n}|\cdot)\Vert_{q} & \le\inf_{\beta\in[0,\beta_{\max}]}\underline{\Theta}(\alpha,\beta)-\frac{1}{q'}\beta\\
 & \le\inf_{\beta\in[0,\beta_{\max}]}\underline{\varphi}(\alpha,\beta)-\frac{1}{q'}\beta=\Theta_{q'}(\alpha).
\end{align*}

The asymptotic sharpness of the inequality in \eqref{eq:-69} for
$q<0$ follows similarly. We lastly prove the asymptotic sharpness
of \eqref{eq:-69} for $0<q<1$.

The case of $\alpha=0$ is trivial. We next consider the case of $\alpha>0$.
For $\epsilon>0$, let $R_{X}$ be a distribution such that $D(R_{X}\|P_{X})\le\alpha-\epsilon.$
For $\delta>0$, define $A_{n}=\bigcup_{T_{X}\in B_{\delta}(R_{X})}\mathcal{T}_{T_{X}}^{(n)}.$
Then by Sanov's theorem \cite[Theorem 6.2.10]{Dembo}, we have for
any $\delta,\epsilon>0$, 
\begin{align*}
\limsup_{n\to\infty}-\frac{1}{n}\log P_{X}^{\otimes n}(A_{n}) & \leq\inf_{Q_{X}\in B_{\delta}(R_{X})}D(Q_{X}\|P_{X})\le D(R_{X}\|P_{X}).
\end{align*}
Denote $\mu_{T_{X}T_{Y}}$ as the joint law of the pair of empirical
distributions $(T_{X},T_{Y})$ of $(X^{n},Y^{n})\sim P_{XY}^{\otimes n}$.
Observe that $P_{X|Y}^{\otimes n}(A_{n}|y^{n})$ remains the same
for all $y^{n}$ having the same type. Hence, $P_{X|Y}^{\otimes n}(A_{n}|y^{n})=\mu_{T_{X}|T_{Y}}(B_{\delta}(R_{X})|T_{Y})$
for any $y^{n}$ having type $T_{Y}$. By this identity, we obtain
\begin{align*}
-\frac{1}{n}\log\Vert P_{X|Y}^{\otimes n}(A_{n})\Vert_{q} & =-\frac{1}{qn}\log\int_{\mathcal{Y}^{n}}P_{X|Y}^{\otimes n}(A_{n}|y^{n})^{q}\mathrm{d}P_{Y}^{\otimes n}(y^{n})\\
 & =-\frac{1}{qn}\log\int_{\mathcal{P}(\mathcal{Y})}\mu_{T_{X}|T_{Y}}(B_{\delta}(R_{X})|T_{Y})^{q}\mathrm{d}\mu_{T_{Y}}.
\end{align*}
Suppose that $\Theta_{q'}(\alpha)<\infty$, otherwise, the desired
conclusion follows trivially. Denote $A_{\epsilon}:=\{T_{Y}:D(T_{Y}\|P_{Y})\le\Theta_{q'}(\alpha)+\epsilon\},$
which is compact \cite{Erven}. Then, 
\begin{align}
-\frac{1}{n}\log\Vert P_{X|Y}^{\otimes n}(A_{n})\Vert_{q} & \geq-\frac{1}{qn}\log[\int_{A_{\epsilon}}\mu_{T_{X}|T_{Y}}(B_{\delta}(R_{X})|T_{Y})^{q}\mathrm{d}\mu_{T_{Y}}+\mu_{T_{Y}}(A_{\epsilon}^{c})].\label{eq:-53}
\end{align}
By Sanov's theorem \cite[Theorem 6.2.10]{Dembo}, $\liminf_{n\to\infty}-\frac{1}{n}\log\mu_{T_{Y}}(A_{\epsilon}^{c})\ge\Theta_{q'}(\alpha)+\epsilon.$
By the inequality in \eqref{eq:-69}, we know that the limit of the
LHS of \eqref{eq:-53} satisfies $\liminf_{n\to\infty}-\frac{1}{n}\log\Vert P_{X|Y}^{\otimes n}(A_{n})\Vert_{q}\le\Theta_{q'}(\alpha)$.
Hence, the first term in the logarithm in the RHS of \eqref{eq:-53}
dominates in the sense that if we take $\liminf_{n\to\infty}$ for
both sides of \eqref{eq:-53}, then the RHS turns into $\liminf_{n\to\infty}-\frac{1}{qn}\log\int_{A_{\epsilon}}\mu_{T_{X}|T_{Y}}(B_{\delta}(R_{X})|T_{Y})^{q}\mathrm{d}\mu_{T_{Y}}$.
Moreover, by the compactness, there is a finite cover $\{B_{\delta}(T_{Y}^{(i)}):1\le i\le L\}$
of $A_{\epsilon}$. Further, combining this with Jensen's inequality
yields that for $0<q<1$, 
\begin{align*}
\int_{A_{\epsilon}}\mu_{T_{X}|T_{Y}}(B_{\delta}(R_{X})|T_{Y})^{q}\mathrm{d}\mu_{T_{Y}} & \leq\sum_{i=1}^{L}\int_{B_{\delta}(T_{Y}^{(i)})}\mu_{T_{X}|T_{Y}}(B_{\delta}(R_{X})|T_{Y})^{q}\mathrm{d}\mu_{T_{Y}}\\
 & =\sum_{i=1}^{L}\mu_{T_{Y}}(B_{\delta}(T_{Y}^{(i)}))\int_{B_{\delta}(T_{Y}^{(i)})}\mu_{T_{X}|T_{Y}}(B_{\delta}(R_{X})|T_{Y})^{q}\frac{1}{\mu_{T_{Y}}(B_{\delta}(T_{Y}^{(i)}))}\mathrm{d}\mu_{T_{Y}}\\
 & \leq\sum_{i=1}^{L}\mu_{T_{Y}}(B_{\delta}(T_{Y}^{(i)}))\Big[\frac{\int_{B_{\delta}(T_{Y}^{(i)})}\mu_{T_{X}|T_{Y}}(B_{\delta}(R_{X})|T_{Y})\mathrm{d}\mu_{T_{Y}}}{\mu_{T_{Y}}(B_{\delta}(T_{Y}^{(i)}))}\Big]^{q}\\
 & =\sum_{i=1}^{L}\mu_{T_{Y}}(B_{\delta}(T_{Y}^{(i)}))^{1-q}\mu_{T_{X}T_{Y}}(B_{\delta}(R_{X})\times B_{\delta}(T_{Y}^{(i)}))^{q}\\
 & \le L\max_{1\le i\le L}\mu_{T_{Y}}(B_{\delta}(T_{Y}^{(i)}))^{1-q}\mu_{T_{X}T_{Y}}(B_{\delta}(R_{X})\times B_{\delta}(T_{Y}^{(i)}))^{q},
\end{align*}
where all the summation $\sum_{i=1}^{L}$ can be taken over all $i$
such that $\mu_{T_{Y}}(B_{\delta}(T_{Y}^{(i)}))>0$, in order to ensure
that the distribution $\frac{\mu_{T_{Y}}}{\mu_{T_{Y}}(B_{\delta}(T_{Y}^{(i)}))}$
is well-defined. Furthermore, by Sanov's theorem \cite[Theorem 6.2.10]{Dembo},
\begin{align*}
\liminf_{n\to\infty}-\frac{1}{n}\log\mu_{T_{Y}}(B_{\delta}(T_{Y}^{(i)})) & \ge\inf_{Q_{Y}\in B_{\le\delta}(T_{Y}^{(i)})}D(Q_{Y}\|P_{Y}),\\
\liminf_{n\to\infty}-\frac{1}{n}\log\mu_{T_{X}T_{Y}}(B_{\delta}(R_{X})\times B_{\delta}(T_{Y}^{(i)})) & \ge\inf_{Q_{XY}\in\mathcal{C}_{\le\delta}(R_{X},T_{Y}^{(i)})}D(Q_{XY}\|P_{XY}),
\end{align*}
where $\mathcal{C}_{\le\delta}$ is defined in \eqref{eq:-87}. 

Combining all the points above yields that for fixed $L$,

\begin{align}
 & \liminf_{n\to\infty}-\frac{1}{n}\log\Vert P_{X|Y}^{\otimes n}(A_{n})\Vert_{q}\nonumber \\
 & \ge\min_{1\le i\le L}-\frac{1}{q'}\inf_{Q_{Y}\in B_{\le\delta}(T_{Y}^{(i)})}D(Q_{Y}\|P_{Y})+\inf_{Q_{XY}\in\mathcal{C}_{\le\delta}(R_{X},T_{Y}^{(i)})}D(Q_{XY}\|P_{XY})\nonumber \\
 & \ge\inf_{R_{Y}}-\frac{1}{q'}\inf_{Q_{Y}\in B_{\le\delta}(R_{Y})}D(Q_{Y}\|P_{Y})+\inf_{Q_{XY}\in\mathcal{C}_{\le\delta}(R_{X},R_{Y})}D(Q_{XY}\|P_{XY})\nonumber \\
 & \ge\inf_{Q_{Y}}-\frac{1}{q'}D(Q_{Y}\|P_{Y})+\inf_{S_{XY}\in\mathcal{C}_{\le2\delta}(R_{X},Q_{Y})}D(S_{XY}\|P_{XY})\:=:g(\delta),\label{eq:-90}
\end{align}
where the last inequality follows since for any $Q_{Y}\in B_{\le\delta}(R_{Y})$
and $S_{XY}\in\mathcal{C}_{\le\delta}(R_{X},R_{Y})$, by the triangle
inequality, it holds that $S_{Y}\in B_{\le2\delta}(Q_{Y})$. (In fact,
in the above derivation, we use the compactness or finite cover technique
to swap the liminf and the optimization operators coming along with
Sanov's theorem.) 

Suppose that $\Theta_{q'}(\alpha)<\infty$, otherwise, the desired
conclusion follows trivially. Hence, it suffices to consider $R_{Y},Q_{XY}$
in the infimum above respectively belong to some sublevel sets of
$D(Q_{Y}\|P_{Y})$ and $D(S_{XY}\|P_{XY})$. It is known that the
sublevel sets of the relative entropy are compact. By taking limit
as $\delta\downarrow0$ for \eqref{eq:-90}, we have $\liminf_{n\to\infty}-\frac{1}{n}\log\Vert P_{X|Y}^{\otimes n}(A_{n})\Vert_{q}\ge\lim_{\delta\downarrow0}g(\delta).$
Let $\{\delta_{i}\}$ be a decreasing positive sequence with limit
zero. Let $\{(Q_{Y}^{(i)},S_{XY}^{(i)})\}$ be a sequence from the
sublevel sets of $D(Q_{Y}\|P_{Y})$ and $D(S_{XY}\|P_{XY})$ such
that $-\frac{1}{q'}D(Q_{Y}^{(i)}\|P_{Y})+D(S_{XY}^{(i)}\|P_{XY})\le g(\delta)+\tau$.
Since the sublevel sets are compact, we can pass $\{(Q_{Y}^{(i)},S_{XY}^{(i)})\}$
to a convergent subsequence for which the limit is denoted as $(Q_{Y}^{*},S_{XY}^{*})$.
It is easy to see that $S_{X}^{*}=R_{X},S_{Y}^{*}=Q_{Y}^{*}$. Moreover,
by the lower semicontinuity of the relative entropy, 
\begin{align*}
-\frac{1}{q'}D(Q_{Y}^{*}\|P_{Y})+D(S_{XY}^{*}\|P_{XY}) & \le\liminf_{i\to\infty}-\frac{1}{q'}D(Q_{Y}^{(i)}\|P_{Y})+D(S_{XY}^{(i)}\|P_{XY})\\
 & \le\lim_{\delta\downarrow0}g(\delta)+\tau.
\end{align*}
Note that here $-\frac{1}{q'}>0$. Therefore, 
\begin{align*}
\liminf_{n\to\infty}-\frac{1}{n}\log\Vert P_{X|Y}^{\otimes n}(A_{n})\Vert_{q} & \ge-\frac{1}{q'}D(R_{Y}^{*}\|P_{Y})+D(Q_{XY}^{*}\|P_{XY})-\tau\\
 & \ge\inf_{R_{Y|X}}-\frac{1}{q'}D(R_{Y}\|P_{Y})+D(R_{X}R_{Y|X}\|P_{XY})-\tau.
\end{align*}
Since $\tau>0$ is arbitrary, it can be removed from the last line
above. Optimizing the RHS above over all $R_{X}$ such that $D(R_{X}\|P_{X})\le\alpha-\epsilon$
and taking limit as $\epsilon\downarrow0$, we obtain 
\begin{align*}
\Theta_{q}^{(\infty)}(\alpha) & \ge\lim_{\epsilon\downarrow0}\sup_{Q_{X}:D(Q_{X}\|P_{X})\le\alpha-\epsilon}\inf_{Q_{Y}}\mathbb{D}(Q_{X},Q_{Y}\|P_{XY})-\frac{D(Q_{Y}\|P_{Y})}{q'}.
\end{align*}
By the time-sharing argument, we have $\Theta_{q}^{(\infty)}(\alpha)\ge\lim_{\epsilon\downarrow0}\Theta_{q'}(\alpha-\epsilon).$
Since $t\in(0,\alpha)\mapsto\Theta_{q'}(\alpha-t)$ is concave and
nondecreasing, it is left continuous at $t=0$. Hence, $\Theta_{q}^{(\infty)}(\alpha)\ge\Theta_{q'}(\alpha)$,
which, combined with \eqref{eq:-69}, implies that $\Theta_{q}^{(\infty)}(\alpha)=\Theta_{q'}(\alpha)$
for the case $0<q<1$.

\section{\label{sec:Proof-of-Lemma-Theta}Proof of Lemma \ref{lem:Theta}}

We first prove \eqref{eq:sublinear}. For $\alpha<\alpha_{\min},\beta<\beta_{\min}$,
define 
\begin{align}
\underline{\Psi}(\alpha,\beta) & :=\lim_{n\to\infty}-\frac{1}{n}\log\max_{\substack{A\subseteq\mathcal{X}^{n},B\subseteq\mathcal{Y}^{n}:P_{X}^{\otimes n}(A)\in e^{-n\alpha}+[0,e^{-n\alpha_{\min}}],\\
P_{X}^{\otimes n}(B)\in e^{-n\beta}+[0,e^{-n\beta_{\min}}]
}
}\frac{P_{XY}^{\otimes n}(A\times B)}{P_{X}^{\otimes n}(A)P_{Y}^{\otimes n}(B)}.\label{eq:Upsilon-1-2}
\end{align}
Let $\alpha,s\ge0$ such that $\alpha+s<\alpha_{\min}$. Then for
any $A\subseteq\mathcal{X}^{n}$ such that $P_{X}^{\otimes n}(A)\in e^{-n\alpha}+[0,e^{-n\alpha_{\min}}]$,
we can partition $A$ into a number of subsets $\{A_{i}\}$ such that
$P_{X}^{\otimes n}(A_{i})\in e^{-n(\alpha+s)}+[0,e^{-n\alpha_{\min}}]$
for all $i$. We further have 
\begin{align*}
P_{X}^{\otimes n}(B|A) & =\sum_{i}\frac{P_{X}^{\otimes n}(A_{i})}{P_{X}^{\otimes n}(A)}P_{X}^{\otimes n}(B|A_{i})\leq\max_{i}P_{X}^{\otimes n}(B|A_{i})\\
 & \leq\max_{A\subseteq\mathcal{X}^{n}:P_{X}^{\otimes n}(A)\in e^{-n(\alpha+s)}+[0,e^{-n\alpha_{\min}}]}P_{X}^{\otimes n}(B|A).
\end{align*}
Hence, $\underline{\Psi}(\alpha,\beta)\le\underline{\Psi}(\alpha+s,\beta)$,
which implies $\underline{\Psi}(\alpha,\beta)$ is nonincreasing in
$\alpha\in(0,\alpha_{\min})$ given $\beta$. By symmetry, $\underline{\Psi}(\alpha,\beta)$
is also nonincreasing in $\beta\in(0,\beta_{\min})$ given $\alpha$.

On the other hand, by Theorem \ref{thm:strongsse}, for $0<\alpha<\alpha_{\min},0<\beta<\beta_{\min}$,
we have 
\begin{align}
\underline{\Psi}(\alpha,\beta) & =\lim_{n\to\infty}-\frac{1}{n}\log\max_{\substack{A\subseteq\mathcal{X}^{n},B\subseteq\mathcal{Y}^{n}:P_{X}^{\otimes n}(A)\in e^{-n\alpha}+[0,e^{-n\alpha_{\min}}],\\
P_{X}^{\otimes n}(B)\in e^{-n\beta}+[0,e^{-n\beta_{\min}}]
}
}P_{XY}^{\otimes n}(A\times B)-\alpha-\beta\nonumber \\
 & =\lim_{n\to\infty}-\frac{1}{n}\log\max_{\substack{A\subseteq\mathcal{X}^{n},B\subseteq\mathcal{Y}^{n}:P_{X}^{\otimes n}(A)\leq e^{-n\alpha}+e^{-n\alpha_{\min}},\\
P_{X}^{\otimes n}(B)\leq e^{-n\beta}+e^{-n\beta_{\min}}
}
}P_{XY}^{\otimes n}(A\times B)-\alpha-\beta\label{eq:-92}\\
 & =\underline{\Theta}(\alpha,\beta)-\alpha-\beta,\label{eq:-91}
\end{align}
where \eqref{eq:-92} follows since there is an optimal pair $(A,B)$
attaining the maximum in \eqref{eq:-92} satisfying $P_{X}^{\otimes n}(A)\in e^{-n\alpha}+[0,e^{-n\alpha_{\min}}]$
and $P_{X}^{\otimes n}(B)\in e^{-n\beta}+[0,e^{-n\beta_{\min}}]$,
since, otherwise, we can enlarge $A,B$ to increase $P_{XY}^{\otimes n}(A\times B)$;
see Remark \ref{rem:equality}. Combining \eqref{eq:-91} and the
monotonicity of $\underline{\Psi}$, we have the inequality in \eqref{eq:sublinear}.
Similarly, one can prove \eqref{eq:sublinear-1}.

We next consider \eqref{eq:sublinear-2}. The proof for this case
is also similar to the above. For $\alpha<\alpha_{\min}$, define
\begin{align}
\underline{\Psi}_{q}(\alpha) & :=\lim_{n\to\infty}-\frac{1}{qn}\log\max_{A\subseteq\mathcal{X}^{n}:P_{X}^{\otimes n}(A)\in e^{-n\alpha}+[0,e^{-n\alpha_{\min}}]}\frac{\mathbb{E}_{Y^{n}}[P_{X|Y}^{\otimes n}(A|Y^{n})^{q}]}{P_{X}^{\otimes n}(A)^{q}}\label{eq:Upsilon-1-1-1-1-1}\\
 & =\lim_{n\to\infty}-\frac{1}{qn}\log\max_{A\subseteq\mathcal{X}^{n}:P_{X}^{\otimes n}(A)\in e^{-n\alpha}+[0,e^{-n\alpha_{\min}}]}\mathbb{E}_{Y^{n}}[P_{X}^{\otimes n}(Y^{n}|A)^{q}].
\end{align}
Let $\alpha,s\ge0$ such that $\alpha+s<\alpha_{\min}$. Then for
any $A\subseteq\mathcal{X}^{n}$ such that $P_{X}^{\otimes n}(A)\in e^{-n\alpha}+[0,e^{-n\alpha_{\min}}]$,
we can partition $A$ into a number of subsets $\{A_{i}\}$ such that
$P_{X}^{\otimes n}(A_{i})\in e^{-n(\alpha+s)}+[0,e^{-n\alpha_{\min}}]$
for all $i$. We further have 
\begin{align}
\mathbb{E}_{Y^{n}}[P_{X}^{\otimes n}(Y^{n}|A)^{q}] & =\mathbb{E}_{Y^{n}}[\big(\sum_{i}\frac{P_{X}^{\otimes n}(A_{i})}{P_{X}^{\otimes n}(A)}P_{X}^{\otimes n}(Y^{n}|A_{i})\big)^{q}]\nonumber \\
 & \le\sum_{i}\frac{P_{X}^{\otimes n}(A_{i})}{P_{X}^{\otimes n}(A)}\mathbb{E}_{Y^{n}}[P_{X}^{\otimes n}(Y^{n}|A_{i})^{q}]\label{eq:-90-1}\\
 & \le\max_{i}\mathbb{E}_{Y^{n}}[P_{X}^{\otimes n}(Y^{n}|A_{i})^{q}]\nonumber \\
 & \le\max_{A\subseteq\mathcal{X}^{n}:P_{X}^{\otimes n}(A)\in e^{-n(\alpha+s)}+[0,e^{-n\alpha_{\min}}]}\mathbb{E}_{Y^{n}}[P_{X}^{\otimes n}(Y^{n}|A)^{q}],\nonumber 
\end{align}
where \eqref{eq:-90-1} follows by Jensen's inequality since $q\ge1$.
Hence, $\underline{\Psi}_{q}(\alpha)\ge\underline{\Psi}_{q}(\alpha+s)$,
which implies $\underline{\Psi}_{q}$ is nondecreasing on $(0,\alpha_{\min})$.

On the other hand, by Theorem \ref{thm:strongqstability}, for $0<\alpha<\alpha_{\min}$,
we have 
\begin{align}
\underline{\Psi}_{q}(\alpha) & =\underline{\Theta}_{q'}(\alpha)-\alpha.\label{eq:-77-1-1-1}
\end{align}
Combining \eqref{eq:-77-1-1-1} and the monotonicity of $\underline{\Psi}_{q}$,
we have the inequality in \eqref{eq:sublinear-2}. Similarly, one
can prove \eqref{eq:sublinear-3}.

 \bibliographystyle{unsrt}
\bibliography{ref}

\end{document}